\documentclass[12pt]{amsart}
\usepackage{graphicx}
\usepackage{amssymb,amsmath,mathtools}
\usepackage{amsthm}
\usepackage{indentfirst}
\usepackage{tikz}
\usepackage{color,graphicx}
\usepackage{hyperref}
\usepackage{color}
\usepackage{epstopdf}
\usepackage{lpic}
\usepackage{enumerate}
\usepackage{caption}
\usepackage{tgpagella}
\usepackage{float} 
\usepackage{multirow}
\usepackage[square,numbers]{natbib}

\usepackage{mathrsfs}

\allowdisplaybreaks[4]

\newcommand{\be}{\begin{equation}}

\newcommand{\ee}{\end{equation}}
\newtheorem{thm}{Theorem}[section]
\newtheorem{lem}{Lemma}[section]
\newtheorem{prop}{Proposition}[section]
\newtheorem{defn}{Definition}[section]
\newtheorem{rmk}{Remark}[section]
\newtheorem{cor}{Corollary}[section]
\newtheorem{ques}{Question}[section]

\newtheorem{exm}{Example}[section]
\newtheorem{manualtheoreminner}{Theorem}
\newenvironment{refthm}[1]{%
  \renewcommand\themanualtheoreminner{#1}%
  \manualtheoreminner
}{\endmanualtheoreminner}

\newcommand{\tr}{{\mathrm{tr}}}
\newcommand{\isom}{{\mathrm{Isom}}}

\numberwithin{equation}{section}
\numberwithin{figure}{section}

\begin{document}

\title[Geometry of Selberg's bisectors]{Geometry of Selberg's bisectors in the symmetric space $SL(n,\mathbb{R})/SO(n,\mathbb{R})$}

\author{Yukun Du}
\address[Y. Du]{Department of Mathematics, University of Georgia, Athens, GA 30601}
\email{yukun.du@uga.edu}


\keywords{}

\begin{abstract}
We study some properties of a family of symmetric spaces, namely\\ $SL(n,\mathbb{R})/SO(n,\mathbb{R})=:\mathcal{P}(n)$, where we replace the Riemannian metric on $\mathcal{P}(n)$ with a premetric suggested by Selberg. These include criteria for intersections of Selberg’s bisectors, the shape of Dirichlet-Selberg domains, and angle-like functions between Selberg's bisectors. These properties are generalizations of properties on the hyperbolic spaces $\mathbf{H}^n$ related to Poincar\'e's fundamental polyhedron theorem. 
\end{abstract}

\date{\today}
\maketitle
\tableofcontents
\vspace{12pt}
\section{Introduction}
We study in this paper several problems related to $\mathcal{P}(n)$, the symmetric space associated with the non-compact simple Lie group $SL(n,\mathbb{R})$. These problems are motivated by a proposed generalized version of Poincar\'e's algorithm on $\mathcal{P}(n)$\cite{kapovich2023geometric}, designed to determine the geometric finiteness of the subgroup of $SL(n,\mathbb{R})$ generated by given elements $\gamma_1$, \dots, $\gamma_k\in SL(n,\mathbb{R})$. Different from the original algorithm developed for the hyperbolic $n$-space, we replace the Riemannian metric on $\mathcal{P}(n)$ with a premetric suggested by Selberg, \cite{selberg1962discontinuous}. 
\subsection{Main results}
We begin with a question concerning the ridge cycle condition in the generalized Poincar\'e's algorithm for the symmetric space $\mathcal{P}(n)$. To formulate this condition, we construct an angle-like function for pairs of hyperplanes in $\mathcal{P}(n)$, satisfying particular natural properties described in Section \ref{sec:s5}:
\begin{refthm}{\ref{thm:main:2}}
Let $\sigma_1 = A^\perp$ and $\sigma_2 = B^\perp$ be distinct co-oriented hyperplanes in $\mathcal{P}(n)$.
\begin{enumerate}
    \item Suppose that the set of generalized eigenvalues of $(A,B)$ contains nonreal numbers, denoted by $\lambda_1,\dots,\lambda_k$ and $\lambda_1^*,\dots,\lambda_k^*$. The following serves as an invariant angle function:
    \[
        \theta(\sigma_1, \sigma_2) = \dfrac{1}{k}\sum_{i=1}^k\left|\arg(\lambda_i)\right|.
    \]
    \item Suppose that all distinct generalized eigenvalues of $(A,B)$ are real or infinity, ordered as $\lambda_k>\dots>\lambda_1$, where $k\geq 3$. The following serves as an invariant angle function (which is the limit as $\lambda_k\to\infty$ if $\infty$ is a generalized eigenvalue):
    \[
        \theta(\sigma_1, \sigma_2) = \arccos \frac{\sum_{i=1}^k\frac{\lambda_{i+1}+\lambda_i}{\lambda_{i+1}-\lambda_i}}{\sqrt{\left(\sum_{i=1}^k\frac{1}{\lambda_{i+1}-\lambda_i}\right)\left(\sum_{i=1}^k\frac{(\lambda_{i+1}+\lambda_i)^2}{\lambda_{i+1}-\lambda_i}\right)}}.
    \]
    \item If $(A,B)$ has at most $2$ distinct generalized eigenvalues (could have no eigenvalues) and all of these are real, then $(\sigma_1,\sigma_2)$ is not in the domain of any invariant angle function.
\end{enumerate}
\end{refthm}

Another aspect of our research is motivated by the proposed Poincar\'e's algorithm for $\mathcal{P}(n)$. A natural question emerges regarding how to ascertain whether two hyperplanes intersect. Based on a result due to Finsler, \cite{finsler1936vorkommen}, we characterized intersecting hyperplanes as follows:
\begin{refthm}{\ref{thm:main:1}}
Let $A^\perp$ and $B^\perp$ be hyperplanes in $\mathcal P(n)$, and suppose that the pencil $(A,B)$ is in Lancaster–Rodman normal form. Then the hyperplanes $A^\perp$ and $B^\perp$ are disjoint if and only if there exist $\alpha\in\overline{\mathbb R}$ and $\epsilon\in\{1,-1\}$, with $A-\alpha B\ne O_n$, such that:
\begin{itemize}
    \item Every block is either zero, one-dimensional with eigenvalue $\alpha_j\in\overline{\mathbb R}$ and sign characteristic $\epsilon_j\in\{1,-1\}$, or two-dimensional with eigenvalue $\alpha$ and sign characteristic $\epsilon$.
    \item For every one-dimensional block, $\epsilon_j=-\epsilon$ whenever $\alpha_j<\alpha$, and $\epsilon_j=\epsilon$ whenever $\alpha_j>\alpha$, with no restriction when $\alpha_j=\alpha$.
\end{itemize}
\end{refthm}
Additionally, we establish a sufficient condition for the disjointness of Selberg bisectors $Bis(X, Y)$ and $Bis(Y, Z)$, given in Theorem \ref{thm:4:2:2}.

The proposed Poincar\'e's algorithm contains a crucial sub-algorithm dedicated to computing the partially ordered set (poset) structure of finitely-sided polyhedra in $\mathcal{P}(n)$. Our sub-algorithm exhibits both similarities and differences from the original algorithm formulated for hyperbolic spaces by Epstein and Petronio\cite{epstein1994exposition}. A detailed exposition of this algorithm is included in Section \ref{sec:s7}.

Notably, Poincar\'e's algorithm for a given point $X\in\mathcal{P}(n)$ and a given discrete subgroup $\Gamma<SL(n,\mathbb{R})$ may not terminate if the Dirichlet-Selberg domain $DS(X,\Gamma)$ is infinitely-sided. We classify Abelian subgroups of $SL(3,\mathbb{R})$ with positive eigenvalues based on whether their Dirichlet-Selberg domains are finitely-sided, as illustrated in Proposition \ref{2-gen} and Theorem \ref{main:s4}. Detailed examples of finitely-sided Dirichlet domains are provided in Subsection \ref{subsec:1:2}.

Using Dirichlet-Selberg domains in $\mathcal{P}(n)$, we confirm in Theorem \ref{schot} that certain subgroups in $SL(n,\mathbb{R})$ are free groups over their generators. Such subgroups exhibit Dirichlet-Selberg domains with pairwise disjoint facets, akin to Schottky groups in $SO^+(n,1)$\cite{maskit1967characterization}. The freeness of such subgroups was initially proved by Tits\cite{tits1972free} through a different approach.

\subsection{Examples of Dirichlet-Selberg domains}\label{subsec:1:2}
We implement the algorithm described in Section \ref{sec:s7} using \emph{Mathematica}. Our implementation allowed us to obtain the poset structures of specific finitely-sided Dirichlet-Selberg domains in $\mathcal{P}(3)$.
\begin{exm}
    Let $\Gamma = \langle g\rangle$, where
    \[
    g = \begin{pmatrix}
        2 & 1 & 1\\1 & 2 & 0\\1 & 0 & 1
    \end{pmatrix}\in SL(3,\mathbb{Z}),
    \]
    and
    \[
    X = \begin{pmatrix}
        2 & 3 & -2\\ 3 & 5 & -3\\ -2 & -3 & 3
    \end{pmatrix}\in \mathcal{P}(3).
    \]
    Notably, $X^k\in \mathcal{P}(3)$ for all $k\in\mathbb{R}$, constituting a geodesic in $\mathcal{P}(3)$.

    Computation suggests that the Dirichlet-Selberg domain $DS(X,\Gamma)$ comprises two facets: $Bis(X,g.X)$ and $Bis(X,g^{-1}.X)$. This domain does not contain any ridges or other proper faces.

    The Dirichlet-Selberg domain $DS(X^2,\Gamma)$ includes $8$ facets, denoted by $F_i$, where $i=\pm 1,\pm 2,\pm 3, \pm 4$, with each $F_i$ contained in $Bis(X^2,g^i.X^2)$. Additionally, this domain contains $12$ ridges, expressed as $F_i\cap F_{i'}$, where $(i,i')$ satisfies one of the following conditions:
    \begin{itemize}
        \item $i=1$ and $-3\leq i'\leq 3$.
        \item $i = -1$ and $-4\leq i'\leq 2$.
        \item $i' = i+1$ and $-3\leq i\leq 3$, $i\neq -1,0$.
    \end{itemize}
    
    The Dirichlet-Selberg domain $DS(X^3,\Gamma)$ includes $12$ facets, denoted by $F_i$, $i=\pm 1,\dots, \pm 6$, with each $F_i$ contained in $Bis(X^3,g^i.X^3)$. Additionally, this domain contains $18$ ridges, expressed as $F_i\cap F_{i'}$, where $(i,i')$ satisfies one of the following conditions:
    \begin{itemize}
        \item $i=1$ and $-5\leq i'\leq 3$.
        \item $i = -1$ and $-6\leq i'\leq 2$.
        \item $i' = i+1$ and $-3\leq i\leq 5$, $i\neq -1,0$.
    \end{itemize}
\end{exm}
We also computed the poset structure of Dirichlet-Selberg domains for two-generated Abelian subgroups of $SL(3,\mathbb{Z})$:
\begin{exm}
    Let $\Gamma = \langle g,h\rangle$, where
    \[
        g = \begin{pmatrix}
        2 & 1 & 1\\1 & 2 & 0\\1 & 0 & 1
        \end{pmatrix}, h = \begin{pmatrix}
            2 & 0 & -1\\ 0 & 1 & 1\\ -1 & 1 & 2
        \end{pmatrix}\in SL(3,\mathbb{Z}),
    \]
    and $X$ is defined as in the previous example. Notably, $g$ commutes with $h$, hence $\Gamma$ is a two-generated Abelian subgroup of $SL(3,\mathbb{Z})$.

    Our computation suggests that the Dirichlet-Selberg domain $DS(X,\Gamma)$ contains $10$ facets $F_\gamma \subset Bis(X,\gamma.X)$, where $\gamma$ takes the following elements:
    \[
    g,\ g^{-1},\ h,\ h^{-1},\ gh,\ g^{-1}h^{-1},\ gh^{-1},\ g^{-1}h,\ g^2h,\ \mathrm{or}\ g^{-2}h^{-1}.
    \]
    We depict these elements in the graphs below:
    \begin{table}[h!]
        \centering
        \begin{minipage}{.48\textwidth}
            \centering
            \begin{tikzpicture}
                \fill[black] (1,0) circle (1.5pt) node[above] {\scriptsize $g$};
                \fill[black] (-1,0) circle (1.5pt) node[above] {\scriptsize $g^{-1}$};
                \fill[black] (-1.5,0.866) circle (1.5pt) node[above] {\scriptsize $g^{-1}h$};
                \fill[black] (-0.5,0.866) circle (1.5pt) node[above] {\scriptsize $h$};
                \fill[black] (0.5,0.866) circle (1.5pt) node[above] {\scriptsize $g h$};
                \fill[black] (1.5,0.866) circle (1.5pt) node[above] {\scriptsize $g^2 h$};
                \fill[black] (-1.5,-0.866) circle (1.5pt) node[above] {\scriptsize $g^{-2}h^{-1}$};
                \fill[black] (-0.5,-0.866) circle (1.5pt) node[above] {\scriptsize $g^{-1}h^{-1}$};
                \fill[black] (0.5,-0.866) circle (1.5pt) node[above] {\scriptsize $h^{-1}$};
                \fill[black] (1.5,-0.866) circle (1.5pt) node[above] {\scriptsize $gh^{-1}$};
            \end{tikzpicture}
        \end{minipage}
        \begin{minipage}{.48\textwidth}
            \centering
            \begin{tikzpicture}
                \draw[-] (1,0) -- (0.5,0.866);
                \draw[-] (1,0) -- (0.5,-0.866);
                \draw[-] (1,0) -- (1.5,0.866);
                \draw[-] (1,0) -- (1.5,-0.866);
                \draw[-] (-1,0) -- (-0.5,0.866);
                \draw[-] (-1,0) -- (-0.5,-0.866);
                \draw[-] (-1,0) -- (-1.5,0.866);
                \draw[-] (-1,0) -- (-1.5,-0.866);
                \draw[-] (1,0) -- (-0.5,0.866);
                \draw[-] (1,0) -- (-0.5,-0.866);
                \draw[-] (-1,0) -- (0.5,0.866);
                \draw[-] (-1,0) -- (0.5,-0.866);
                \draw[-] (-1.5,0.866) -- (-0.5,0.866);
                \draw[-] (-0.5,0.866) -- (0.5,0.866);
                \draw[-] (0.5,0.866) -- (1.5,0.866);
                \draw[-] (-1.5,-0.866) -- (-0.5,-0.866);
                \draw[-] (-0.5,-0.866) -- (0.5,-0.866);
                \draw[-] (0.5,-0.866) -- (1.5,-0.866);
                \fill[gray] (1,0) circle (1.5pt) node[above] {};
                \fill[gray] (-1,0) circle (1.5pt) node[above] {};
                \fill[gray] (-1.5,0.866) circle (1.5pt) node[above] {};
                \fill[gray] (-0.5,0.866) circle (1.5pt) node[above] {};
                \fill[gray] (0.5,0.866) circle (1.5pt) node[above] {};
                \fill[gray] (1.5,0.866) circle (1.5pt) node[above] {};
                \fill[gray] (-1.5,-0.866) circle (1.5pt) node[above] {};
                \fill[gray] (-0.5,-0.866) circle (1.5pt) node[above] {};
                \fill[gray] (0.5,-0.866) circle (1.5pt) node[above] {};
                \fill[gray] (1.5,-0.866) circle (1.5pt) node[above] {};
            \end{tikzpicture}
        \end{minipage}
        \caption{In the left graph, a point labeled with $\gamma\in \Gamma$ specifies the facet $F_\gamma$ of $DS(X,\Gamma)$. In the right graph, an edge between points $\gamma$ and $\gamma'$ (labels omitted) specifies a ridge $F_\gamma\cap F_{\gamma'}$ of $DS(X,\Gamma)$.}
    \end{table}
    Here, we arranged the points in a hexagonal pattern for clarity. This domain contains $18$ ridges, also shown in the graph. 

    The Dirichlet-Selberg domain $DS(X^2,\Gamma)$ contains $26$ facets and $75$ ridges, shown in the graph below.
    \begin{table}[h!]
        \centering
        \begin{minipage}{.48\textwidth}
            \centering
            \begin{tikzpicture}[scale=0.75,every node/.style={scale=0.75}]
                \fill[black] (1,0) circle (1.5pt) node[above] {\scriptsize $g$};
                \fill[black] (-1,0) circle (1.5pt) node[above] {\scriptsize $g^{-1}$};
                \fill[black] (-1.5,0.866) circle (1.5pt) node[above] {\scriptsize $g^{-1}h$};
                \fill[black] (-0.5,0.866) circle (1.5pt) node[above] {\scriptsize $h$};
                \fill[black] (0.5,0.866) circle (1.5pt) node[above] {\scriptsize $g h$};
                \fill[black] (1.5,0.866) circle (1.5pt) node[above] {\scriptsize $g^2 h$};
                \fill[black] (-1.5,-0.866) circle (1.5pt) node[above] {\scriptsize $g^{-2}h^{-1}$};
                \fill[black] (-0.5,-0.866) circle (1.5pt) node[above] {\scriptsize $g^{-1}h^{-1}$};
                \fill[black] (0.5,-0.866) circle (1.5pt) node[above] {\scriptsize $h^{-1}$};
                \fill[black] (1.5,-0.866) circle (1.5pt) node[above] {\scriptsize $gh^{-1}$};
                \fill[black] (2,0) circle (1.5pt) node[above] {\scriptsize$g^2$};
                \fill[black] (3,0) circle (1.5pt) node[above] {\scriptsize$g^3$};
                \fill[black] (2.5,-0.866) circle (1.5pt) node[above] {\scriptsize$g^2h^{-1}$};
                \fill[black] (3.5,0.866) circle (1.5pt) node[above] {\scriptsize$g^4h$};
                \fill[black] (-2,0) circle (1.5pt) node[above] {\scriptsize$g^{-2}$};
                \fill[black] (-3,0) circle (1.5pt) node[above] {\scriptsize$g^{-3}$};
                \fill[black] (-2.5,0.866) circle (1.5pt) node[above] {\scriptsize$g^{-2}h$};
                \fill[black] (-3.5,-0.866) circle (1.5pt) node[above] {\scriptsize$g^{-4}h^{-1}$};
                \fill[black] (0,1.732) circle (1.5pt) node[above] {\scriptsize$gh^2$};
                \fill[black] (1,1.732) circle (1.5pt) node[above] {\scriptsize$g^2h^2$};
                \fill[black] (2,1.732) circle (1.5pt) node[above] {\scriptsize$g^3h^2$};
                \fill[black] (1.5,2.598) circle (1.5pt) node[above] {\scriptsize$g^3h^3$};
                \fill[black] (0,-1.732) circle (1.5pt) node[above] {\scriptsize$g^{-1}h^{-2}$};
                \fill[black] (-1,-1.732) circle (1.5pt) node[above] {\scriptsize$g^{-2}h^{-2}$};
                \fill[black] (-2,-1.732) circle (1.5pt) node[above] {\scriptsize$g^{-3}h^{-2}$};
                \fill[black] (-1.5,-2.598) circle (1.5pt) node[above] {\scriptsize$g^{-3}h^{-3}$};
            \end{tikzpicture}
        \end{minipage}
        \begin{minipage}{.48\textwidth}
            \centering
            \begin{tikzpicture}[scale=0.75,every node/.style={scale=0.75}]
                \draw[-] (1,0) -- (0.5,0.866);
                \draw[-] (1,0) -- (0.5,-0.866);
                \draw[-] (1,0) -- (1.5,0.866);
                \draw[-] (1,0) -- (1.5,-0.866);
                \draw[-] (-1,0) -- (-0.5,0.866);
                \draw[-] (-1,0) -- (-0.5,-0.866);
                \draw[-] (-1,0) -- (-1.5,0.866);
                \draw[-] (-1,0) -- (-1.5,-0.866);
                \draw[-] (1,0) -- (-0.5,0.866);
                \draw[-] (1,0) -- (-0.5,-0.866);
                \draw[-] (-1,0) -- (0.5,0.866);
                \draw[-] (-1,0) -- (0.5,-0.866);
                \draw[-] (-1.5,0.866) -- (-0.5,0.866);
                \draw[-] (-0.5,0.866) -- (0.5,0.866);
                \draw[-] (0.5,0.866) -- (1.5,0.866);
                \draw[-] (-1.5,-0.866) -- (-0.5,-0.866);
                \draw[-] (-0.5,-0.866) -- (0.5,-0.866);
                \draw[-] (0.5,-0.866) -- (1.5,-0.866);
                \draw[-] (1,0) -- (-1,0);
                \draw[-] (0.5,0.866) -- (0.5,-0.866);
                \draw[-] (-0.5,0.866) -- (-0.5,-0.866);
                \draw[-] (0.5,0.866) -- (-0.5,-0.866);
                \draw[-] (1,0) -- (2,0);
                \draw[-] (1.5,-0.866) -- (2,0);
                \draw[-] (1.5,-0.866) -- (-1,0);
                \draw[-] (2,0) -- (3,0);
                \draw[-] (2,0) -- (2.5,-0.866);
                \draw[-] (1.5,-0.866) -- (2.5,-0.866);
                \draw[-] (2.5,-0.866) -- (3,0);
                \draw[-] (3,0) -- (3.5,0.866);
                \draw[-] (0.5,-0.866) -- (0,-1.732);
                \draw[-] (-0.5,-0.866) -- (0,-1.732);
                \draw[-] (-0.5,-0.866) -- (-1,-1.732);
                \draw[-] (-1.5,-0.866) -- (-1,-1.732);
                \draw[-] (-1.5,-0.866) -- (-2,-1.732);
                \draw[-] (-2,-1.732) -- (-1,-1.732);
                \draw[-] (-2,-1.732) -- (-1.5,-2.598);
                \draw[-] (-1,-1.732) -- (-1.5,-2.598);
                \draw[-] (-2,0) -- (-1,0);
                \draw[-] (-1.5,-0.866) -- (-3.5,-0.866);
                \draw[-] (-1.5,0.866) -- (-2.5,0.866);
                \draw[-] (0.5,0.866) -- (0,1.732);
                \draw[-] (-0.5,0.866) -- (0,1.732);
                \draw[-] (0.5,0.866) -- (1,1.732);
                \draw[-] (1.5,0.866) -- (2,1.732);
                \draw[-] (-2,-1.732) -- (-0.5,-0.866);
                \draw[-] (-1.5,-0.866) -- (-3,0);
                \draw[-] (-1.5,-0.866) -- (0.5,0.866);
                \draw[-] (-0.5,-0.866) -- (-2,0);
                \draw[-] (-0.5,-0.866) -- (-3,0);
                \draw[-] (-0.5,-0.866) -- (1.5,0.866);
                \draw[-] (0.5,-0.866) -- (-2,0);
                \draw[-] (1.5,-0.866) -- (3,0);
                \draw[-] (2.5,-0.866) -- (1,0);
                \draw[-] (2,0) -- (3.5,0.866);
                \draw[-] (1,0) -- (-1.5,0.866);
                \draw[-] (-1,0) -- (-2.5,0.866);
                \draw[-] (-1,0) -- (1,1.732);
                \draw[-] (-2,0) -- (1.5,0.866);
                \draw[-] (-3,0) -- (0.5,0.866);
                \draw[-] (0.5,0.866) -- (2,1.732);
                \draw[-] (-0.5,0.866) -- (1,1.732);
                \draw[-] (-0.5,0.866) -- (1.5,2.598);
                \draw[-] (-3.5,-0.866) to[out=-15,in=-165] (-0.5,-0.866);
                \draw[-] (0.5,-0.866) to[out=30,in=-90] (2,1.732);
                \draw[-] (0.5,-0.866) to[out=45,in=-105] (1.5,0.866);
                \draw[-] (1,0) to[out=45,in=-105] (2,1.732);
                \draw[-] (-2,0) to[out=-15,in=-165] (1,0);
                \draw[-] (-3,0) to[out=-15,in=-165] (-1,0);
                \draw[-] (-0.5,-0.866) to[out=45,in=-105] (1,1.732);
                \draw[-] (0.5,0.866) to[out=45,in=-105] (1.5,2.598);
                \draw[-] (-2.5,0.866) to[out=15,in=165] (-0.5,0.866);
                \draw[-] (-1.5,0.866) to[out=15,in=165] (0.5,0.866);
                \draw[-] (-2.5,0.866) to[out=30,in=150] (0.5,0.866);
                \draw[-] (-1.5,0.866) to[out=30,in=150] (1.5,0.866);
                \fill[gray] (1,0) circle (1.5pt) node[above] {};
                \fill[gray] (-1,0) circle (1.5pt) node[above] {};
                \fill[gray] (-1.5,0.866) circle (1.5pt) node[above] {};
                \fill[gray] (-0.5,0.866) circle (1.5pt) node[above] {};
                \fill[gray] (0.5,0.866) circle (1.5pt) node[above] {};
                \fill[gray] (1.5,0.866) circle (1.5pt) node[above] {};
                \fill[gray] (-1.5,-0.866) circle (1.5pt) node[above] {};
                \fill[gray] (-0.5,-0.866) circle (1.5pt) node[above] {};
                \fill[gray] (0.5,-0.866) circle (1.5pt) node[above] {};
                \fill[gray] (1.5,-0.866) circle (1.5pt) node[above] {};
                \fill[gray] (2,0) circle (1.5pt) node[above] {};
                \fill[gray] (3,0) circle (1.5pt) node[above] {};
                \fill[gray] (2.5,-0.866) circle (1.5pt) node[above] {};
                \fill[gray] (3.5,0.866) circle (1.5pt) node[above] {};
                \fill[gray] (-2,0) circle (1.5pt) node[above] {};
                \fill[gray] (-3,0) circle (1.5pt) node[above] {};
                \fill[gray] (-2.5,0.866) circle (1.5pt) node[above] {};
                \fill[gray] (-3.5,-0.866) circle (1.5pt) node[above] {};
                \fill[gray] (0,1.732) circle (1.5pt) node[above] {};
                \fill[gray] (1,1.732) circle (1.5pt) node[above] {};
                \fill[gray] (2,1.732) circle (1.5pt) node[above] {};
                \fill[gray] (1.5,2.598) circle (1.5pt) node[above] {};
                \fill[gray] (0,-1.732) circle (1.5pt) node[above] {};
                \fill[gray] (-1,-1.732) circle (1.5pt) node[above] {};
                \fill[gray] (-2,-1.732) circle (1.5pt) node[above] {};
                \fill[gray] (-1.5,-2.598) circle (1.5pt) node[above] {};
            \end{tikzpicture}
        \end{minipage}
        \caption{In the left graph, a point labeled with $\gamma\in \Gamma$ specifies the facet $F_\gamma$ of $DS(X^2,\Gamma)$. In the right graph, an edge between points $\gamma$ and $\gamma'$ (labels omitted) specifies a ridge $F_\gamma\cap F_{\gamma'}$ of $DS(X^2,\Gamma)$.}
    \end{table}
\end{exm}
\begin{exm}
    Let us consider the group $\Gamma = SL(3,\mathbb{Z})$. Since the stabilizer of the identity matrix $I$ in $\Gamma$, which is $SO(3)\cap SL(3,\mathbb{Z})$, has an order of $24$, $DS(I,\Gamma)$ does not serve as a fundamental domain of $\Gamma$. Nonetheless, we can introduce
    \[
    X_\epsilon = I + \epsilon \begin{pmatrix}
        1 & -1 & -1 \\ -1 & 0 & -1\\ -1 & -1 & -1
    \end{pmatrix},\ \epsilon\in\mathbb{R}_+,
    \]
    which has a trivial stabilizer in $\Gamma$ for small $\epsilon$. Subsequently, we consider the half-spaces $\{Y|s(\gamma.X_\epsilon,Y)\geq s(X_\epsilon,Y)\}$ for $\gamma\in SO(3)\cap SL(3,\mathbb{Z})$, and take the limit $\epsilon\to 0^+$. This procedure yields half-spaces that partition $DS(I,\mathbb{Z})$ into $24$ convex polyhedra, bounding a fundamental domain $P$ of $\Gamma = SL(3,\mathbb{Z})$. Our program shows that $P$ consists of $9$ facets, each contained in one of the following hyperplanes of $\mathcal{P}(3)$:
    \[
    \begin{split}
        & \begin{pmatrix}
        0 & 1 & 0\\ 1 & 1 & 0\\ 0 & 0 & 0
    \end{pmatrix}^\perp, \begin{pmatrix}
        0 & 0 & 1\\ 0 & 0 & 0\\ 1 & 0 & 1
    \end{pmatrix}^\perp, \begin{pmatrix}
        0 & 0 & 0\\ 0 & 0 & 1\\ 0 & 1 & 1
    \end{pmatrix}^\perp, \begin{pmatrix}
        0 & 1 & 1\\ 1 & 1 & 1\\ 1 & 1 & 1
    \end{pmatrix}^\perp, \\
    & \begin{pmatrix}
        1 & 0 & 0\\ 0 & -1 & 0\\ 0 & 0 & 0
    \end{pmatrix}^\perp, \begin{pmatrix}
        0 & 0 & 0\\ 0 & 1 & 0\\ 0 & 0 & -1
    \end{pmatrix}^\perp, \begin{pmatrix}
        0 & 1 & 0\\ 1 & 0 & 1\\ 0 & 1 & 0
    \end{pmatrix}^\perp, \begin{pmatrix}
        0 & 1 & 1\\ 1 & 0 & 0\\ 1 & 0 & 0
    \end{pmatrix}^\perp, \begin{pmatrix}
        0 & 0 & 1\\ 0 & 0 & 1\\ 1 & 1 & 0
    \end{pmatrix}^\perp.
    \end{split}
    \]
    Furthermore, $P$ has $7$ vertices, namely:
    \[
    \begin{split}
        & \begin{pmatrix}
            1 & 0 & 0 \\0 & 1 & 0\\ 0 & 0 & 1
        \end{pmatrix},\begin{pmatrix}
            1 & -1/2 & -1/2 \\-1/2 & 1 & 0\\ -1/2 & 0 & 1
        \end{pmatrix},\begin{pmatrix}
            1 & -1/2 & 0 \\-1/2 & 1 & -1/2\\ 0 & -1/2 & 1
        \end{pmatrix},\begin{pmatrix}
            1 & 0 & -1/2 \\0 & 1 & -1/2\\ -1/2 & -1/2 & 1
        \end{pmatrix},\\
        & \begin{pmatrix}
            1 & -1/2 & -1/2 \\-1/2 & 1 & 1/2\\ -1/2 & 1/2 & 1
        \end{pmatrix},\begin{pmatrix}
            1 & -1/2 & 1/2 \\-1/2 & 1 & -1/2\\ 1/2 & -1/2 & 1
        \end{pmatrix},\begin{pmatrix}
            1 & 1/2 & -1/2 \\1/2 & 1 & -1/2\\ -1/2 & -1/2 & 1
        \end{pmatrix},\\
    \end{split}
    \]
    and $3$ vertices on the Satake boundary, namely:
    \[
    \begin{pmatrix}
        1 & 0 & 0\\ 0 & 0 & 0\\ 0 & 0 & 0
    \end{pmatrix},\begin{pmatrix}
        1 & 0 & 0\\ 0 & 1 & 0\\ 0 & 0 & 0
    \end{pmatrix},\begin{pmatrix}
        1 & -1/2 & 0\\ -1/2 & 1 & 0\\ 0 & 0 & 0
    \end{pmatrix}.
    \]
    We observe that the Satake vertices of $P$ project onto the following points in the hyperbolic plane $\mathcal{P}(2)$:
    \[
    \begin{pmatrix}
        1 & 0 \\ 0 & 0
    \end{pmatrix}, \begin{pmatrix}
        1 & 0 \\ 0 & 1
    \end{pmatrix},\begin{pmatrix}
        1 & -1/2 \\ -1/2 & 1
    \end{pmatrix}.
    \]
    The polyhedron formed by these points constitutes one-half of the fundamental domain of the modular group $SL(2,\mathbb{Z})$.
\end{exm}
\subsection{Organization of the paper}
In Section \ref{sec:s2}, we review the background material and preliminaries. In Section \ref{sec:s5}, we prove Theorem \ref{thm:main:2}, indicating the existence of invariant angle functions. In Section \ref{sec:s3}, we prove Theorem \ref{thm:main:1}, the criterion that determines if two hyperplanes in $\mathcal{P}(n)$ intersect. We propose a sufficient condition for the disjointness of Selberg bisectors in Subsection \ref{sec:4:2}. In Section \ref{sec:s7}, we describe and implement an algorithm that computes the poset structure of finitely-sided polyhedra in $\mathcal{P}(n)$. In Section \ref{sec:s4}, we classify Abelian subgroups of $SL(3,\mathbb{R})$ with positive eigenvalues, according to the finite-sidedness of their Dirichlet-Selberg domains. In Section \ref{sec:s8}, we show the existence of Schottky groups in $SL(n,\mathbb{R})$, which is generic for $n$ even. We present future directions in Section \ref{sec:99}. Research beside the main results, including the motivation of the angle formula \eqref{equ:main:2}, is included in the appendix. 
\vspace{12pt}
\section{Background}\label{sec:s2}
\subsection{The symmetric space \texorpdfstring{$SL(n,\mathbb{R})/SO(n,\mathbb{R})$}{Lg}}
The space $\mathcal{P}(n) = SL(n,\mathbb{R})/SO(n,\mathbb{R})$ is the Riemannian symmetric space for the Lie group $SL(n,\mathbb{R})$. In Cartan's classification of irreducible simply connected symmetric spaces, $\mathcal{P}(n)$ is the non-compact symmetric space of type $(A_{n-1}\mathrm{I})$\cite{helgason1979differential}. This family of symmetric spaces is one of the most important examples among symmetric spaces; for instance, every symmetric space of non-compact type isometrically embeds in some $\mathcal{P}(n)$ as a totally geodesic submanifold\cite{eberlein1996geometry}.

Using the Cartan decomposition of the Lie group $SL(n,\mathbb{R})$, and the Killing form on the Lie algebra $\mathfrak{sl}(n)$\cite{eberlein1996geometry}, one obtains the \textbf{symmetric matrices model} of the space $\mathcal{P}(n)$:
\begin{defn}
The symmetric space $\mathcal{P}(n)$ is the set
    \begin{equation}
        \mathcal{P}(n) = \mathcal{P}_{mat}(n) = \{X\in Sym_n(\mathbb{R})|\det(X) = 1,\ X\succ0\},
    \end{equation}
    equipped with the metric tensor
    \[
    \langle A,B\rangle = \mathrm{tr}(X^{-1}AX^{-1} B),\ \forall A,B\in T_X\mathcal{P}(n).
    \]
Here, $Sym_n(\mathbb{R})$ is the vector space of $n\times n$ real symmetric matrices, and $X\succ0$ (or $X\succeq 0$) means that $X$ is positive definite (or positive semi-definite, respectively).
\end{defn}
Proofs of the following properties can be found in\cite{helgason1979differential}:
\begin{prop}\label{prop:2:1}
\begin{enumerate}
    \item The Riemannian distance between $X, Y\in \mathcal{P}(n)$ equals:
\begin{equation}
    d(X,Y) = \sqrt{\tr((\log(X^{-1}Y))^2)},
\end{equation}
where $\log$ denotes the principal matrix logarithm for matrices with positive eigenvalues.
    \item The identity component $\isom_0(\mathcal{P}(n))$ of the isometry group equals $SL(n,\mathbb{R})$ if $n$ is odd, and $PSL(n,\mathbb{R}) = SL(n,\mathbb{R})/\{\pm I\}$ if $n$ is even. Here, the action $SL(n,\mathbb{R})\curvearrowright \mathcal{P}(n)$ is given by congruence transformations: $g.X = g^T Xg$.
\end{enumerate}
\end{prop}

We also consider another model of the symmetric space $\mathcal{P}(n)$, which is an open region in the projective space $\mathbb{R}\mathbf{P}^{(n-1)(n+2)/2} = \mathbf{P}(Sym_n(\mathbb{R}))$:
\begin{defn}
    The \textbf{projective model} of $\mathcal{P}(n)$ is the following set:
    \begin{equation}\label{equ:s1:1}
        \mathcal{P}(n) = \mathcal{P}_{proj}(n) = \{[X]\in \mathbf{P}(Sym_n(\mathbb{R}))|X\succ0\},
    \end{equation}
    which is identified with the symmetric matrix model by the following diffeomorphism:
    \[
    \mathcal{P}_{proj}(n)\to \mathcal{P}_{mat}(n),\quad [X]\mapsto (\det(X))^{-1/n})\cdot X.
    \]
\end{defn}
One defines the \textbf{standard Satake compactification} and \textbf{Satake boundary} of $\mathcal{P}(n)$ through the projective model.
\begin{defn}\label{defn:satake}
    The standard Satake compactification of $\mathcal{P}(n)$ is the set
    \[
    \overline{\mathcal{P}(n)}_S = \{[X]\in \mathbf{P}(Sym_n(\mathbb{R}))|X\succeq 0\},
    \]
    and the Satake boundary of $\mathcal{P}(n)$ is the set
    \[
    \partial_S \mathcal{P}(n) = \overline{\mathcal{P}(n)}_S\backslash \mathcal{P}(n).
    \]
\end{defn}

Throughout the paper, we consider the vector space $Sym_n(\mathbb{R})$ equipped with a bilinear form $\langle A,B\rangle: = \mathrm{tr}(A\cdot B)$ and understand the orthogonality as well as orthogonal complements with respect to this bilinear form.

\subsection{Planes, Selberg's bisectors, and convex polyhedra}
The direct analog of Dirichlet domains\cite{ratcliffe1994foundations} in $\mathcal{P}(n)$ is non-convex and impractical to study. This is due to the nonlinear nature of the Riemannian distance on $\mathcal{P}(n)$ defined in Proposition \ref{prop:2:1}. Selberg proposed a two-point invariant under the $SL(n,\mathbb{R})$-action on $\mathcal{P}(n)$ as a substitute of the Riemannian distance, \cite{selberg1962discontinuous}:
\begin{defn}
    For $X,Y\in \mathcal{P}(n)$, define the \textbf{Selberg's invariant} from $X$ to $Y$ as
    \[
    s(X,Y) = \mathrm{tr}(X^{-1}Y).
    \]
\end{defn}
\begin{prop}
    The function $s(-,-)$ satisfies the following properties:
    \begin{itemize}
        \item The function $s(-,-)$ behaves similarly to a distance function: $s(X,Y)\geq n$ for any $X,Y\in \mathcal{P}(n)$, and $s(X,Y) = n$ if and only if $X = Y$.
        \item The function $s(-,-)$ is $SL(n,\mathbb{R})$-invariant: for any $g\in SL(n,\mathbb{R})$ and $X,Y\in \mathcal{P}(n)$, one has $s(X,Y) = s(g.X,g.Y)$.
    \end{itemize}
\end{prop}
For any $X\in\mathcal{P}(n)$, Selberg's invariant $s(X,Y)$ is linear in the entries of $Y$. Consequently, the bisectors in $\mathcal{P}(n)$ with respect to Selberg's invariant are also linear:
\begin{prop}
    Define the (Selberg) \textbf{bisector} $Bis(X,Y)$ for $X,Y\in\mathcal{P}(n)$:
    \[
    Bis(X,Y) = \{Z\in\mathcal{P}(n)|s(X,Z) = s(Y,Z)\}.
    \]
    Then for any $X, Y\in \mathcal{P}(n)$, $Bis(X, Y)$ is defined by a linear equation over the entries of the symmetric matrix $Z$.
\end{prop}
This observation allows Selberg to define a polyhedral analog of the Dirichlet domain:
\begin{defn}
    The \textbf{Dirichlet-Selberg domain} for a discrete subset $\Gamma\subset SL(n,\mathbb{R})$ centered at $X\in \mathcal{P}(n)$ is the set
    \[
    DS(X,\Gamma) = \{Y\in\mathcal{P}(n)|s(X,Y)\leq s(g.X,Y),\ \forall g\in\Gamma\}.
    \]
\end{defn}
To realize the polyhedral nature of Dirichlet-Selberg domains, we turn to define \textbf{planes}, \textbf{hyperplanes}, \textbf{half-spaces} and \textbf{convex polyhedra} in $\mathcal{P}(n)$, analogously to the corresponding definitions in hyperbolic spaces\cite{ratcliffe1994foundations}:
\begin{defn}\label{def:s1:1}
\begin{enumerate}
    \item A \textbf{$d$-plane} of $\mathcal{P}(n)$ is the non-empty intersection of $\mathcal{P}(n)$ with a $(d+1)$-dimensional linear subspace of the vector space $Sym_n(\mathbb{R}) = \mathbb{R}^{n(n+1)/2}$. An $((n-1)(n+2)/2-1)$-plane of $\mathcal{P}(n)$ is called a \textbf{hyperplane} of $\mathcal{P}(n)$.
    \item The complement of a hyperplane in $\mathcal{P}(n)$ consists of two connected components; each of them is called an \textbf{open half-space} of $\mathcal{P}(n)$. A \textbf{closed half-space} is the closure of an open half-space.
    \item A (closed) \textbf{convex polyhedron} in $\mathcal{P}(n)$ is the intersection of a locally finite collection of closed half-spaces in $\mathcal{P}(n)$.
\end{enumerate}
\end{defn}

\begin{defn}
\begin{enumerate}
    \item The \textbf{dimension} of a convex polyhedron $P$ is the dimension of the minimal plane $\Sigma\subset \mathcal{P}(n)$ containing $P$. Such a plane is unique and denoted by $span(P)$.
    \item A \textbf{facet} of a convex polyhedron $P$ is a maximal subset $F\subset \partial P$ such that $F$ itself is a convex polyhedron in $\mathcal{P}(n)$. We denote by $\mathcal{S}(P)$ the set of facets of $P$.
    \item A \textbf{face} of a convex polyhedron $P$ is defined inductively as follows:
    \begin{itemize}
        \item The convex polyhedron $P$ is a face of itself.
        \item A facet of a face of $P$ is also a face of $P$.
    \end{itemize}
    We denote by $\mathcal{F}(P)$ the set of proper faces of $P$.
    \item A \textbf{ridge} of $P$ is a codimension $2$ face of $P$. We denote by $\mathcal{R}(P)$ the set of ridges of $P$.
    \item A face $F$ of $P$ is said to be minimal if there are no other faces of $P$ which are contained in $F$.
\end{enumerate}
\end{defn}
\begin{rmk}
    For a convex polyhedron $P$ in $\mathcal{P}(n)$, the set $\mathcal{F}(P)$ forms a poset, where the inclusion relation between faces defines the partial order. The minimal faces of $P$ correspond to the minimal elements of the poset $\mathcal{F}(P)$.
\end{rmk}
The group $SL(n,\mathbb{R})$ acts on convex polyhedra in $\mathcal{P}(n)$. We define \textbf{fundamental polyhedra} for a discrete subgroup $\Gamma< SL(n,\mathbb{R})$ analogously to fundamental polyhedra in hyperbolic spaces \cite{ratcliffe1994foundations}.
\begin{prop}[\cite{kapovich2023geometric}]
    For a discrete subgroup $\Gamma< SL(n,\mathbb{R})$ and a point $X\in \mathcal{P}(n)$, the Dirichlet-Selberg domain $DS(X,\Gamma)$ is a convex polyhedron in $\mathcal{P}(n)$. Moreover, if $Stab_\Gamma(X)$ is trivial, $DS(X,\Gamma)$ is a fundamental polyhedron for $\Gamma$.
\end{prop}

Sometimes (especially in Section \ref{sec:s5}), we realize a hyperplane $\sigma$ in $\mathcal{P}(n)$ as the normal space of a matrix.
\begin{defn}
    The normal space of a nonzero matrix $A\in Sym_n(\mathbb{R})$ is defined as
    \[
     A^\perp = \{X\in \mathcal{P}(n)|\mathrm{tr}(X\cdot A) = 0\},
    \]
    which is a hyperplane in $\mathcal{P}(n)$ whenever it is non-empty. We designate $A$ as a \textbf{normal vector} of the hyperplane $A^\perp$. A hyperplane associated with a normal vector is called a \textbf{co-oriented hyperplane}. 
    
    The normal vector of a hyperplane is unique up to a nonzero multiple. Identical co-oriented hyperplanes with normal vectors that differ by a positive multiple are regarded as the same co-oriented hyperplanes. Identical co-oriented hyperplanes with normal vectors that differ by a negative multiple from each other are said to be oppositely oriented. If $\sigma$ is a co-oriented hyperplane given by $A^\perp$, then the co-oriented hyperplane with the opposite orientation is denoted by $-\sigma$ or $(-A)^\perp$.

    We say that a co-oriented hyperplane $\sigma$ \textbf{lies between} two co-oriented hyperplanes $A^\perp$ and $B^\perp$ if the normal vector associated with $\sigma$ is a positive linear combination of $A$ and $B$.
\end{defn}

The idea of a co-oriented hyperplane lying between two co-oriented hyperplanes is presented in Figure \ref{fig:my_label1}.
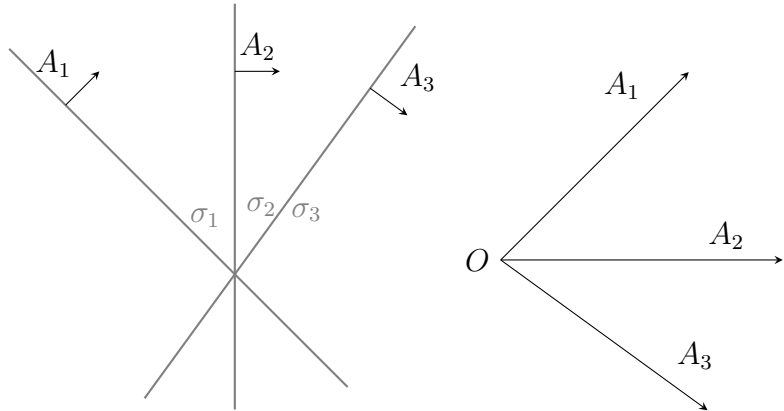
\begin{figure}[!ht]
    \centering
    \begin{tikzpicture}[scale=1.5]
    \draw[gray,thick] (-0.8,-1.1) -- (1.6,2.2) node[pos=0.5,right]{$\sigma_3$};
    \draw [-stealth] (1.2,1.65) -- (1.53,1.41) node[pos=0.5,above right]{$A_3$};
    \draw[gray,thick] (0,-1.2) -- (0,2.4) node[pos=0.5,right]{$\sigma_2$};
    \draw [-stealth] (0,1.8) -- (0.4,1.8) node[pos=0.5,above]{$A_2$};
    \draw[gray,thick] (1,-1) -- (-2,2) node[pos=0.5,right]{$\sigma_1$};
    \draw [-stealth] (-1.5,1.5) -- (-1.2,1.8) node[pos=0.5,above left]{$A_1$};
    \end{tikzpicture}
    \begin{tikzpicture}[scale=2.5]
    \draw [-stealth] (0,0) -- (1.1,-0.8) node[pos=0.8,above right]{$A_3$};
    \draw [-stealth] (0,0) -- (1.5,0) node[pos=0.8,above]{$A_2$};
    \draw [-stealth] (0,0) -- (1,1) node[pos=0.8,above left]{$A_1$};
    \draw (0,0) node[left]{$O$};
    \end{tikzpicture}
    \caption{The hyperplane $\sigma_2$ lies between the hyperplanes $\sigma_1$ and $\sigma_3$.}
    \label{fig:my_label1}
\end{figure}
Open half-spaces in $\mathcal{P}(n)$ can be also described in terms of normal vectors:
\begin{equation*}
    \{X\in \mathcal{P}(n)|\mathrm{tr}(X\cdot A) > 0\}.
\end{equation*}
One replaces the $>$ with $\geq$ for closed half-spaces. Thus, we define the co-orientation for facets of convex polyhedra in $\mathcal{P}(n)$:
\begin{defn}
    Let $P$ be a convex polyhedron in $\mathcal{P}(n)$. A facet $F\in \mathcal{S}(P)$ associated with a normal vector $A$ of $span(F)$ is called a \textbf{co-oriented facet} of $P$. 
    
    For any facet $F\in \mathcal{S}(P)$, the convex polyhedron $P$ lies within one of the two closed half-spaces bounded by $span(F)$. When there are no special instructions, we make a convention that the normal vector $A$ associated with $F$ is selected so that
    \[
    P\subset \{X\in\mathcal{P}(n)|\mathrm{tr}(X\cdot A)\leq 0\},
    \]
    and say that $A$ is outward-pointing (and $-A$ is inward-pointing). We also make a convention that $span(F)$ is associated with the same normal vector as $F$.
\end{defn}
\subsection{Poincar\'e's Fundamental Polyhedron Theorem and Poincar\'e's Algorithm}\label{subsec:2:2}
The well-known Poincar\'e's Fundamental Polyhedron Theorem \cite{ratcliffe1994foundations} is an essential tool in geometric group theory. Initially proven for spaces of constant curvature, Poincar\'e's Theorem determines if a given exact polyhedron $P\subset \mathbf{H}^n$ with a side pairing $\Phi\subset SO^+(n,1)$ is a fundamental domain of the subgroup $\Gamma = \langle \Phi\rangle$. Here, the side pairing $\Phi$ is a set $\{g_F\in SO^+(n,1)|F\in\mathcal{S}(P)\}$, and $g_F: F'\to F$, called a side pairing transformation, is an isometry between a pair of facets $F$ and $F'$.

Poincar\'e's Theorem in hyperbolic spaces asserts that $P$ is a fundamental domain of $\Gamma$ if two conditions are satisfied. One condition requires that $M = P/\sim$, the quotient space of $P$ by the facet-pairing $\Phi$, is a hyperbolic manifold or orbifold. This condition is usually formulated in terms of \textbf{ridge-cycles} and the angle sum of these ridge-cycles \cite{ratcliffe1994foundations}. The other condition requires that $M$ is complete, and is formulated as a \textbf{cusp link condition}, which is omitted here.

An algorithm, based on Poincar\'e's Theorem, checks if a subgroup of $SO^+(n,1)$ is geometrically finite. The algorithm was initially suggested by Riley \cite{riley1983applications} for the case $n=3$ and by Epstein and Petronio \cite{epstein1994exposition} for the general case.

\textbf{Poincar\'e's Algorithm.}
\begin{enumerate}
    \item Starting with $l = 1$, compute the subset $\Gamma_l\subset \Gamma$ of elements represented by words of length $\leq l$ in the letters of $g_i$ and $g_i^{-1}$, $i=1,\dots,m$. The result is a finite subset of $SO^+(n,1)$.
    \item Compute the Dirichlet domain $D(x,\Gamma_l)$ centered at $x$ for the finite set $\Gamma_l$. Namely, we compute the equations for all ridges of $D(x,\Gamma_l)$. Epstein and Petronio\cite{epstein1994exposition} provide an algorithm for this task.
    \item Having the data for all ridges of $D(x,\Gamma_l)$, we check if this convex polyhedron is exact. That is, for any $g\in \Gamma_l$ and facets $F_{g}$, $F_{g^{-1}}$ of $D(x,\Gamma_l)$ contained in $Bis(x,g.x)$ and $Bis(x,g^{-1}.x)$ respectively, we check if $g.F_{g^{-1}} = F_g$. This is solved by linear programming, e.g., \cite{eaves1982optimal}.
    \item We check if this convex polyhedron satisfies the ridge cycle condition in Poincar\'e's Fundamental Polyhedron Theorem.
    \item If the condition is not satisfied, replace $l$ with $(l+1)$ and repeat the steps above.
    \item If the ridge cycle condition is satisfied, then Proposition \ref{prop:2:3:1} implies that $D = D(x,\Gamma_l)$ satisfies the requirements for Poincar\'e's theorem. Therefore, $D(x,\Gamma_l)$ is the fundamental domain for the group $\Gamma' = \langle\Gamma_l\rangle$. Following \cite{riley1983applications}, check if $g_i\in\Gamma'$ for each generator $g_i$, $i=1,\dots,m$.
    \item If there is a generator $g_i\notin \Gamma'$, we replace $l$ with $(l+1)$ and repeat the steps from the beginning.
    \item If all generators $g_i\in \Gamma'$, then $\Gamma = \Gamma'$ is a discrete subgroup of $SO^+(n,1)$. Moreover, $\Gamma$ is finitely presented; we derive the relators for the presentation of $\Gamma$ from the ridge-cycle data of $D(x,\Gamma_l)$.
\end{enumerate}
We note that we do not check the completeness condition in the algorithm, due to the following fact:
\begin{prop}[\cite{kapovich2023geometric}]\label{prop:2:3:1}
    Suppose that $\Gamma$ is a finite subset of $SO^+(n,1)$, $x\in \mathbf{H}^n$, and the Dirichlet domain $D(x,\Gamma)$ satisfies the ridge cycle condition. Then the quotient space for $D(x,\Gamma)$ is complete, hence the polyhedron $D(x,\Gamma)$ satisfies the assumptions for Poincar\'e's Fundamental Polyhedron Theorem.
\end{prop}

We have seen in the previous subsection that Dirichlet-Selberg domains in the space $\mathcal{P}(n)$ are convex polyhedra. Therefore, we can tentatively generalize Poincar\'e's Algorithm to the group $SL(n,\mathbb{R})$ acting on $\mathcal{P}(n)$, with the Riemannian distance being replaced by Selberg's invariant. This is the motivation for many results in our paper.
\subsection{Matrix pencils and generalized eigenvalues}
Some of our main results use matrix pencils. We briefly review the concepts related to our research.
\begin{defn}
    A real (or complex) \textbf{matrix pencil} is a set $\{A-\lambda B|\lambda\in\mathbb{R}\}$ (or $\lambda\in\mathbb{C}$, respectively), where $A$ and $B$ are real $n\times n$ matrices. We denote this matrix pencil by $(A,B)$.

    We say a matrix pencil $(A,B)$ is \textbf{regular} if $\det(A-\lambda B)\neq 0$ for at least one value $\lambda\in\mathbb{C}$ (equivalently, for almost every $\lambda$). We say $(A,B)$ is \textbf{singular} if both $A$ and $B$ are singular and $A-\lambda B$ is singular for all $\lambda\in\mathbb{C}$. We define the \emph{rank} of the pencil $(A,B)$ as the maximal rank of matrix in the pencil:
    \[
    \operatorname{rank}(A,B) = \max_{\lambda\in\mathbb{C}}\,\operatorname{rank}(A-\lambda B).
    \]
\end{defn}

We define the \textbf{generalized eigenvalues} of a matrix pencil:
\begin{defn}
    A \textbf{generalized eigenvalue} of a matrix pencil $(A,B)$ is a number $\lambda_0\in\mathbb{C}$ such that
    \[
    \operatorname{rank}(A - \lambda_0 B)<\operatorname{rank}(A,B).
    \]
    For a pencil $(A,B)$, the (geometric) \textbf{multiplicity} of a generalized eigenvalue $\lambda_0\neq\infty$ is defined as
    \[
    \operatorname{rank}(A,B) - \operatorname{rank}(A - \lambda_0 B).
    \]
    We adopt the convention that $\infty$ is a generalized eigenvalue of $(A,B)$ if
    \[
    \operatorname{rank}(B)<\operatorname{rank}(A,B).
    \]
    The geometric multiplicity of $\infty$ is $\operatorname{rank}(A,B) - \operatorname{rank}(B)$.
\end{defn}

A matrix pencil $(A,B)$ is \textbf{symmetric} if both $A$ and $B$ are symmetric matrices. We define \textbf{definiteness} for symmetric matrix pencils:
\begin{defn}
    We say that a symmetric matrix pencil $(A,B)$ is (semi-) definite, if either $A$ or $B$ is (semi-) definite and nonzero, or if $A-\lambda B$ is (semi-) definite and nonzero for at least one number $\lambda\in\mathbb{R}$.
\end{defn}

We define congruence transformations of symmetric matrix pencils as $$(A,B)\to (Q^{\mathrm{T}}AQ, Q^{\mathrm{T}}BQ),$$where $Q\in GL(n,\mathbb{R})$, and $A,B\in Sym_n(\mathbb{R})$. Generalized eigenvalues are invariant under these transformations:
\begin{prop}\label{prop:2}
    For any $Q\in GL(n,\mathbb{R})$, the matrix pencils $(A,B)$ and $(Q^{\mathrm{T}}AQ,Q^{\mathrm{T}}BQ)$ have the same generalized eigenvalues as well as the same multiplicities of them.
\end{prop}
\begin{proof}
    This is straightforward from the fact that $\operatorname{rank}(Q^{\mathrm{T}}AQ-\lambda Q^{\mathrm{T}}BQ) = \operatorname{rank}(A-\lambda B)$ for $\det(Q)\neq 0$.
\end{proof}

If $A'$ and $B'$ are linearly independent linear combinations of $A$ and $B$, the generalized eigenvalues of $(A',B')$ relate to those of $(A,B)$ by a M\"obius transformation:
\begin{lem}\label{lem:mob}
    Suppose that $\lambda_1,\dots,\lambda_m$ are the generalized eigenvalues of the matrix pencil $(A,B)$. Then for any $p,q,r,s\in\mathbb{R}$ with $ps-qr\neq 0$, the generalized eigenvalues of $(p A+q B,r A+s B)$ are $\lambda_i': = \frac{p\lambda_i +q}{r\lambda_i + s}$, $i=1,\dots,m$.
\end{lem}
\begin{proof}
    For non-infinite eigenvalues, notice that
    \begin{equation*}
        (p A+q B) - \frac{p\lambda +q}{r\lambda + s}(r A+s B) = \frac{(ps-qr)(A - \lambda B)}{r\lambda + s}.
    \end{equation*}
    Since $ps-qr\neq 0$, one has that $\frac{p\lambda +q}{r\lambda + s}$ is a generalized eigenvalue of $(p A+q B,r A+s B)$ if and only if $\lambda$ is a generalized eigenvalue of $(A,B)$.

    For the infinite eigenvalue, similarly notice that
    \begin{equation*}
        A'-\frac{p}{r}B' = \frac{r A'-p B'}{r} = \frac{(qr-ps)B}{r}.
    \end{equation*}
    The equation implies that $\frac{p}{r}$ is a generalized eigenvalue of $(A',B')$ if and only if $\infty$ is a generalized eigenvalue of $(A,B)$. This agrees with the statement of the lemma if one interprets the formal expression $\frac{p\infty + q}{r\infty +s}$ as $\frac{p}{r}$.

    In both cases, the corresponding eigenvalues are related by a M\"obius transformation: $\overline{\mathbb{C}}\to \overline{\mathbb{C}},\lambda\mapsto \frac{p\lambda +q}{r\lambda+s}$.
\end{proof}

Our work uses the \emph{Lancaster-Rodman normal form} of real matrix pencils under congruence transformation \cite[Theorem 9.2]{lancaster2005canonical}. First, we introduce the following real symmetric matrices,
\[
R_k \coloneqq (\delta_{i+j,k+1})_{i,j=1}^k,\ N_k \coloneqq (\delta_{i+j,k+2})_{i,j=1}^k,
\]
the Kronecker product matrix
\[
\widehat{R}_{2k}\coloneqq R_k\otimes\operatorname{diag}(-1,1),
\]
and $\widehat{N}_{2k}\coloneqq (\delta_{i+j,2k+3})_{i,j=1}^{2k}$. We note that $N_{k} = \operatorname{diag}(0,R_{k-1})$ and $\widehat{N}_{2k} = \operatorname{diag}(0_2,R_{2k-2})$ for $k>0$. 
\begin{thm}\label{thm:LR}
    An $n$-dimensional real symmetric matrix pencil $(A,B)$ is congruent to a real symmetric matrix pencil of the form:
    \[ 
    (A',B') = \bigl(\bigoplus_{j=1}^p(A^\mathrm{r}_j,B^\mathrm{r}_j)\bigl)\oplus \bigl(\bigoplus_{j=1}^q(A^\mathrm{c}_j,B^\mathrm{c}_j)\bigl)\oplus \bigl(\bigoplus_{j=1}^r(A^\mathrm{s}_j,B^\mathrm{s}_j)\bigl)\oplus (O_u,O_u),
    \]
    where:
    \begin{itemize}
        \item Each pencil block $(A^\mathrm{r}_j,B^\mathrm{r}_j)$ is a \emph{block associated with a real eigenvalue} (possibly infinity) $\alpha_j\in \mathbb{R}\cup\{\infty\}$,
        \[
        (A^\mathrm{r}_j,B^\mathrm{r}_j) = (\epsilon_j(\alpha_j R_{k_j} + N_{k_j}),\epsilon_j R_{k_j}),
        \]
        for $\alpha_j\in \mathbb{R}$, or
        \[
        (A^\mathrm{r}_j,B^\mathrm{r}_j) = (\epsilon_j R_{k_j}, -\epsilon_j N_{k_j})
        \]
        for an infinite eigenvalue. Here, $\epsilon_j\in \{1,-1\}$ is called the \emph{sign characteristic}.
        \item Each pencil block $(A^\mathrm{c}_j,B^\mathrm{c}_j)$ is a \emph{block associated with nonreal eigenvalues} $\beta_j,\beta_j^*\in \mathbb{C}\setminus \mathbb{R}$,
        \[
        (A^\mathrm{c}_j,B^\mathrm{c}_j) = (\mu_j R_{2l_j}+ \nu_j \widehat{R}_{2l_j}+ \widehat{N}_{2l_j}, R_{2l_j}),
        \]
        where $\beta_j = \mu_j+\mathrm{i}\nu_j$, $\mu_j\in\mathbb{R}$ and $\nu_j>0$.
        \item Each pencil block $(A^\mathrm{s}_j,B^\mathrm{s}_j)$ is a \emph{nontrivial singular block},
        \[
        (A^\mathrm{s}_j,B^\mathrm{s}_j) = \left(N_{2m_j+1},\left(\begin{array}{ccc}
            O_{m_j} & \mathbf{0} & R_{m_j} \\
            \mathbf{0} & 0 & \mathbf{0} \\
            R_{m_j} & \mathbf{0} & O_{m_j}
        \end{array}\right)\right).
        \]
    \end{itemize}
    Here, $p,q,r,u\geq 0$, $k_j,l_j,m_j>0$, and the total dimension satisfies
    \[
    \sum_{j=1}^p k_j + \sum_{j=1}^q 2l_j + \sum_{j=1}^r (2m_j+1) + u = n.
    \]
    We call such $(A',B')$ a \emph{Lancaster-Rodman normal form} of the real matrix pencil $(A,B)$. Such a normal form is unique up to block permutations. 
\end{thm}
To simplify the statement of Theorem \ref{thm:main:1}, we adopt the normal form of $\epsilon(R_k,-N_k)$ for the infinite eigenvalue case instead of $\epsilon(R_k,N_k)$ in \cite{lancaster2005canonical}. Our adapted form is equivalent to the original form by a congruence of $((-1)^i\delta_{ij})_{i,j = 1}^k = \operatorname{diag}(1,-1,1,-1,\dots)$ and a sign change for $\epsilon$ when $k$ is even.

The Lancaster-Rodman normal form of a regular pencil is closely related to the Jordan normal form:
\begin{thm}[{\cite[Theorem 12.2]{lancaster2005canonical}}]
\label{thm:jordan_cor}
    Let $(A,B)$ be a real symmetric pencil with $B$ invertible. Suppose $Q\in\operatorname{GL}(n,\mathbb{R})$ such that
    \[
    (A',B') = (Q^\mathsf{T}AQ,Q^\mathsf{T}BQ)
    \]
    is in Lancaster-Rodman normal form. Then $Q^{-1}B^{-1}AQ$ is the real Jordan normal form of $B^{-1}A$. Moreover, the Jordan blocks correspond to the (finite) real- and nonreal-eigenvalue blocks of $(A',B')$. 
\end{thm}
\subsection{Miscellaneous}
The following \emph{Cauchy-Binet identity} is used in the proof of Theorem \ref{thm:main:2}.

\begin{lem}\label{cau}
For any numbers $a_i,b_i,c_i$ and $d_i$, where $i=1,\dots,n$,
\begin{equation*}
    \left(\sum_{i=1}^n a_ic_i\right)\left(\sum_{j=1}^n b_jd_j\right) -     \left(\sum_{i=1}^n a_id_i\right)\left(\sum_{j=1}^n b_jc_j\right) = \frac{1}{2}\sum_{i\neq j}(a_ib_j-a_jb_i)(c_id_j-c_jd_i).
\end{equation*}
\end{lem}
\vspace{12pt}
\section{Angle-like Functions between Hyperplanes}\label{sec:s5}
Our first goal is to formulate an angle sum condition for the Poincar\'e's Algorithm in $\mathcal{P}(n)$. The tiling condition for convex polyhedra in $\mathcal{P}(n)$ can be formulated with an angle sum condition analogously to \cite{ratcliffe1994foundations}, with the Riemannian angle being replaced by an angle-like function satisfying specific natural properties as described below, \cite{kapovich2023geometric}:
\begin{defn}\label{defn:s5:1}
    An invariant angle function $\theta(-,-)$ is a function for pairs of co-oriented hyperplanes $(\sigma_1,\sigma_2)$ in $\mathcal{P}(n)$ with the following properties:
    \begin{enumerate}
    \item For any co-oriented hyperplanes $\sigma_1$ and $\sigma_2$, $0 \leq \theta(\sigma_1, \sigma_2)\leq \pi$. Furthermore, $\theta(\sigma_1, \sigma_2) = 0$ if and only if $\sigma_1 = \sigma_2$, while $\theta(\sigma_1, \sigma_2) = \pi$ if and only if $\sigma_1 = -\sigma_2$.
    \item For any co-oriented hyperplanes $\sigma_1$ and $\sigma_2$ and any $g\in SL(n,\mathbb{R})$, $\theta(g.\sigma_1, g.\sigma_2) = \theta(\sigma_1, \sigma_2)$.
    \item For any co-oriented hyperplanes $\sigma_1$ and $\sigma_2$, $\theta(\sigma_2, \sigma_1) = \theta(\sigma_1, \sigma_2)$, $\theta(-\sigma_1, \sigma_2) = \pi - \theta(\sigma_1, \sigma_2)$.
    \item For any co-oriented hyperplane $\sigma_2$ lying between $\sigma_1$ and $\sigma_3$, $\theta(\sigma_1, \sigma_2) + \theta(\sigma_2, \sigma_3) = \theta(\sigma_1, \sigma_3)$.
\end{enumerate}
\end{defn}

For generic pairs of co-oriented hyperplanes, we explicitly construct an invariant angle function, which is presented in the main theorem below.

\begin{thm}\label{thm:main:2}
Let $\sigma_1 = A^\perp$ and $\sigma_2 = B^\perp$ be co-oriented hyperplanes in $\mathcal{P}(n)$ and suppose that $A$ and $B$ in the matrix pencil $(A,B)$ are linearly independent, so $\sigma_1$ and $\sigma_2$ are distinct.
\begin{enumerate}
    \item Suppose that the set of generalized eigenvalues of $(A,B)$ contains nonreal numbers, denoted by $\lambda_1,\dots,\lambda_k$ and $\lambda_1^*,\dots,\lambda_k^*$. The following serves as an invariant angle function:
    \begin{equation}\label{equ:main:1}
        \theta(\sigma_1, \sigma_2) = \dfrac{1}{k}\sum_{i=1}^k\left|\arg(\lambda_i)\right|.
        \end{equation}
    \item Suppose that all distinct generalized eigenvalues of $(A,B)$ are real or infinity, ordered as $\lambda_k>\dots>\lambda_1$, where $k\geq 3$. The following serves as an invariant angle function (which is the limit as $\lambda_k\to\infty$ if $\infty$ is a generalized eigenvalue):
    \begin{equation}\label{equ:main:2}
        \theta(\sigma_1, \sigma_2) = \arccos \frac{\sum_{i=1}^k\frac{\lambda_{i+1}+\lambda_i}{\lambda_{i+1}-\lambda_i}}{\sqrt{\left(\sum_{i=1}^k\frac{1}{\lambda_{i+1}-\lambda_i}\right)\left(\sum_{i=1}^k\frac{(\lambda_{i+1}+\lambda_i)^2}{\lambda_{i+1}-\lambda_i}\right)}}.
    \end{equation}
    \item If $(A,B)$ has at most $2$ distinct generalized eigenvalues and all of these are real (including the case with no generalized eigenvalues), then $(\sigma_1,\sigma_2)$ is not in the domain of any invariant angle function.
\end{enumerate}
\end{thm}

\begin{rmk}
When $\infty$ serves as an eigenvalue of the pencil $(A, B)$, we take the limit $\lambda_k\to \infty$ of the formula \eqref{equ:main:2} and obtain that
\begin{equation}
    \theta(\sigma_1, \sigma_2) = \arccos \frac{\sum_{i=1}^{k-2}\frac{\lambda_{i+1}+\lambda_i}{\lambda_{i+1}-\lambda_i}}{\sqrt{\left(\sum_{i=1}^{k-2}\frac{1}{\lambda_{i+1}-\lambda_i}\right)\left(4\sum_{i=1}^{k-2}\frac{\lambda_{i+1}^2 - \lambda_{i+1}\lambda_i + \lambda_i^2}{\lambda_{i+1}-\lambda_i}\right)}},
\end{equation}
where $\lambda_{k-1}>\dots>\lambda_1$ are the (finite) eigenvalues of $(A,B)$.
\end{rmk}

We say that a given pair of co-oriented hyperplanes $(\sigma_1,\sigma_2) = (A^\perp,B^\perp)$ in $\mathcal{P}(n)$ is of type (1), (2), or (3) if the set of generalized eigenvalues of $(A,B)$ corresponds to case (1), (2), or (3) in Theorem \ref{thm:main:2}, respectively. The following fact is a direct consequence of Lemma \ref{lem:mob}:
\begin{prop}\label{prop:3:1:2}
\begin{enumerate}
    \item For any $g\in SL(n,\mathbb{R})$ the hyperplane pairs $(\sigma_1,\sigma_2)$ and $(g.\sigma_1,g.\sigma_2)$ share the same type.
    \item If $\sigma_3$ lies between $\sigma_1$ and $\sigma_2$, both $(\sigma_1,\sigma_3)$ and $(\sigma_2,\sigma_3)$ belong to the same type as $(\sigma_1,\sigma_2)$.
\end{enumerate}
\end{prop}

Following Proposition \ref{prop:3:1:2}, we will prove the three statements in Theorem \ref{thm:main:2} individually in the subsequent subsections.

\subsection{Proof of Theorem \ref{thm:main:2}, case (1)}\label{subsec:1}
We will prove that the function
\[
\theta(\sigma_1, \sigma_2) = \dfrac{1}{k}\sum_{i=1}^k\left|\arg(\lambda_i)\right|
\]
defined for all pairs $(\sigma_1,\sigma_2)$ of type (1) satisfies the properties listed in Definition \ref{defn:s5:1}.
\begin{proof}[Proof of Theorem \ref{thm:main:2}, case (1)]
To begin, we establish the well-definedness of the function for pairs of co-oriented hyperplanes in \eqref{equ:main:1}. According to Lemma \ref{lem:mob}, $\frac{c_2}{c_1}\lambda_i$ and $\frac{c_2}{c_1}\lambda_i^*$ will be the nonreal generalized eigenvalues of $(c_1A,c_2 B)$, for any $c_1,c_2>0$, where $i=1,\dots,k$. The arguments of these numbers equal the arguments of $\lambda_i$ and $\lambda_i^*$, respectively. Consequently, \eqref{equ:main:1} implies that $\theta((c_1A)^\perp,(c_2B)^\perp) = \theta(A^\perp,B^\perp)$ for any $c_1,c_2>0$, i.e., the expression \eqref{equ:main:1} remains the same for $(c_1A,c_2 B)$.

Furthermore, since $|\arg(\lambda_i^*)| = |-\arg(\lambda_i)| = |\arg(\lambda_i)|$, the outcome of \eqref{equ:main:1} remains unchanged when replacing $\lambda_i$ with $\lambda_i^*$.

Next, we will verify properties (1) to (4) in Definition \ref{defn:s5:1} for the function $\theta$ defined by \eqref{equ:main:1}. The property (1) is obvious. Regarding the property (2), we notice that $(g^{-1})^{\mathrm{T}}.A$ and $(g^{-1})^{\mathrm{T}}.B$ serve as normal vectors of $g.\sigma_1$ and $g.\sigma_2$, respectively. Since the pencil $((g^{-1})^{\mathrm{T}}.A, (g^{-1})^{\mathrm{T}}.B)$ shares the same generalized eigenvalues as $(A,B)$, the angle $\theta(g.\sigma_1,g.\sigma_2) = \theta(\sigma_1,\sigma_2)$.

To verify the property (3), notice that the pencil $(B,A)$ possesses generalized eigenvalues $\lambda_i^{-1}$ and $\lambda_i^{-1*}$, where $i=1,\dots,k$. Since $\arg(\lambda_i^{-1}) = -\arg(\lambda_i)$, it follows $\theta(\sigma_2,\sigma_1) = \theta(\sigma_1,\sigma_2)$. Furthermore, the generalized eigenvalues of the pencil $(-A,B)$ equal $-\lambda_i$ and $-\lambda_i^{*}$, where $i=1,\dots,k$. Since $|\arg(-\lambda_i)| = \pi - |\arg(\lambda_i)|$, we deduce that $\theta(-\sigma_1,\sigma_2) = \pi-\theta(\sigma_1,\sigma_2)$.

Lastly, we verify the property (4). The normal vector $C$ of $\sigma_3$ is a positive linear combination of $A$ and $B$. Since positive rescalings of $A$ and $B$ preserve the angles $\theta(A^\perp, C^\perp)$, $\theta(C^\perp, A^\perp)$ and $\theta(A^\perp, B^\perp)$, we assume that $C = A+B$. Under this condition, Lemma \ref{lem:mob} shows that the nonreal generalized eigenvalues of $(A,C)$ are $(1+\lambda_i)$ and $(1+\lambda_i^*)$, while the nonreal generalized eigenvalues of $(C,B)$ are $\frac{\lambda_i}{1+\lambda_i}$ and $\frac{\lambda_i^*}{1+\lambda_i^*}$, where $i=1,\dots,k$.

We note that $\arg\left(\frac{\lambda}{1+\lambda}\right)>0$ if and only if $\arg(\lambda)>0$. Thus,
\begin{equation*}
\begin{split}
        &\theta(\sigma_1, \sigma_3) + \theta(\sigma_2, \sigma_3) = \frac{1}{k}\sum\left(|\arg(1+\lambda_i)|+\left|\arg\left(\frac{\lambda_i}{1+\lambda_i}\right)\right|\right)\\
        & =\frac{1}{k}\sum\left(\left|\arg(1+\lambda_i) + \arg\left(\frac{\lambda_i}{1+\lambda_i}\right)\right|\right) = \frac{1}{k}\sum\left(|\arg(\lambda_i)|\right) = \theta(\sigma_1,\sigma_2).
\end{split}
\end{equation*}
This concludes the verification of the property (4) in Definition \ref{defn:s5:1}. In summary, the function $\theta$ defined by \eqref{equ:main:1} serves as an invariant angle function.
\end{proof}
\subsection{Proof of Theorem \ref{thm:main:2}, case (2)}\label{subsec:2}
We will prove that the function
\[
\theta(\sigma_1, \sigma_2) = \arccos \frac{\sum_{i=1}^k\frac{\lambda_{i+1}+\lambda_i}{\lambda_{i+1}-\lambda_i}}{\sqrt{\left(\sum_{i=1}^k\frac{1}{\lambda_{i+1}-\lambda_i}\right)\left(\sum_{i=1}^k\frac{(\lambda_{i+1}+\lambda_i)^2}{\lambda_{i+1}-\lambda_i}\right)}}.
\]
defined for all pairs $(\sigma_1,\sigma_2)$ of type (2) satisfies the properties listed in Definition \ref{defn:s5:1}. For simplicity, we define
\[
t(x_1,\dots,x_k) = \frac{\sum_{i=1}^k\frac{x_{i+1}+x_{i}}{x_{i+1}-x_{i}}}{\sqrt{\left(\sum_{i=1}^k\frac{1}{x_{i+1}-x_{i}}\right)\left(\sum_{i=1}^k\frac{(x_{i+1}+x_{i})^2}{x_{i+1}-x_{i}}\right)}},
\]
and $\overline{t}(x_1,\dots,x_k) = t(x_{\sigma_k},\dots, x_{\sigma_{1}})$, where $\{\sigma_1,\dots,\sigma_k\}$ represents the permutation of $\{1,\dots,k\}$ such that $x_{\sigma_k}\geq \dots \geq x_{\sigma_{1}}$. Our first lemma concerns the compositions of $\overline{t}$ and M\"obius transformations:
\begin{lem}\label{lem:3:3:0:1}
    Let $\varphi$ be a M\"obius transformation on $\overline{\mathbb{R}} = \mathbb{R}\cup\{\infty\}$, and let $\lambda_k>\dots>\lambda_1$ represent real numbers. If $\varphi$ is orientation-preserving, then
    \begin{equation}\label{equ:3:3}
    \overline{t}(\varphi(\lambda_1),\dots, \varphi(\lambda_k)) = t(\varphi(\lambda_1),\dots, \varphi(\lambda_k)).
    \end{equation}
    If $\varphi$ is orientation-reversing, then
    \begin{equation}\label{equ:3:4}
    \overline{t}(\varphi(\lambda_1),\dots, \varphi(\lambda_k)) = -t(\varphi(\lambda_1),\dots, \varphi(\lambda_k)).
    \end{equation}
\end{lem}
\begin{proof}
    Let $(\sigma_1,\dots,\sigma_k)$ denote the permutation of $(1,\dots,k)$ such that
    \[
    \varphi(\lambda_{\sigma_k})>\dots>\varphi(\lambda_{\sigma_1}).
    \]
    If $\varphi$ is orientation-preserving, then $(\sigma_1,\dots,\sigma_k)$ is a cyclic permutation, satisfying $t(x_1,\dots,x_k) = t(x_{\sigma_1},\dots,x_{\sigma_k})$ for any $x_1,\dots,x_k$. Hence, equation \eqref{equ:3:3} holds.

    If $\varphi$ is orientation-reversing, then $(\sigma_1,\dots,\sigma_k)$ is a cyclic permutation of $(k,\dots,1)$. Note that $t(x_k,\dots,x_1) = -t(x_1,\dots,x_k)$ for any $x_1,\dots,x_k$. Therefore, equation \eqref{equ:3:4} holds.
\end{proof}
We also need the following lemma:
\begin{lem}\label{lem:ineq}
    Let $\lambda_k>\dots>\lambda_1$ be real numbers, then the following inequalities hold:
    \begin{subequations}
        \begin{eqnarray}\label{equ:3:3:a}
            & \sum_{1\leq i\neq j\leq k}\dfrac{(\lambda_{i+1}+\lambda_i -\lambda_{j+1}-\lambda_j)^2}{(\lambda_{i+1}-\lambda_i)(\lambda_{j+1}-\lambda_j)}> 0,\\ \label{equ:3:3:b}
            & \sum_{i=1}^k\dfrac{(2+\lambda_{i+1}+\lambda_i)^2}{\lambda_{i+1}-\lambda_i}> 0. 
        \end{eqnarray}
    \end{subequations}
\end{lem}
\begin{proof}
    Denote
    \[
    s_i = \lambda_{i+1}+\lambda_i,\ d_i = \lambda_{i+1} - \lambda_i,
    \]
    where the index is taken modulo $k$. Then the numbers $s_i$, $i=1,\dots,k$ satisfy the following inequalities:
    \[
    s_1<s_2<\dots<s_{k-1}, \quad s_1<s_{k}<s_{k-1},
    \]
    and the numbers $d_i$, $i=1,\dots,k$ satisfy
    \[
    d_1,\dots,d_{k-1}>0,\quad d_k = -\sum_{i = 1}^{k-1}d_i<0.
    \]
    In terms of $s_i$ and $d_i$, inequalities \eqref{equ:3:3:a} and \eqref{equ:3:3:b} reduce to
    \[
    \sum_{1\leq i\neq j\leq k}\frac{(s_i -s_j)^2}{d_id_j}> 0,\quad \sum_{i=1}^k\frac{(2+s_i)^2}{d_i}> 0.
    \]
    Assume that $j$ is the number satisfying $s_j\leq s_k\leq s_{j+1}$, $1\leq j\leq k-2$, then
    \begin{align*}
     & \frac{(s_1-s_{k-1})^2}{d_1d_{k-1}}>\frac{(s_1 - s_k)^2 + (s_k - s_{k-1})^2}{d_1d_{k-1}}> -\frac{(s_1 - s_k)^2}{d_1d_k} - \frac{(s_{k-1} - s_k)^2}{d_{k-1}d_k},\\
     & \frac{(s_i - s_1)^2}{d_1d_i}>-\frac{(s_i-s_k)^2}{d_kd_i},\quad \forall\  j+1\leq i<k-1,\\
     & \frac{(s_i - s_{k-1})^2}{d_{k-1}d_i}>-\frac{(s_i-s_k)^2}{d_kd_i},\quad \forall\  1< i\leq j.
    \end{align*}
    These inequalities yield \eqref{equ:3:3:a}.

    We divide the proof of inequality \eqref{equ:3:3:b} in two cases. If $s_k+2\leq 0$, then
    \[
    -\frac{(2+s_k)^2}{d_k}<\frac{(2+s_k)^2}{d_1}<\frac{(2+s_1)^2}{d_1}.
    \]
    If $s_k+2\geq 0$, then
    \[
    -\frac{(2+s_k)^2}{d_k}<\frac{(2+s_k)^2}{d_{k-1}}<\frac{(2+s_{k-1})^2}{d_{k-1}}.
    \]
    For both cases, inequality \eqref{equ:3:3:b} holds.
\end{proof}
We return to the proof of Theorem \ref{thm:main:2}.
\begin{proof}[Proof of Theorem \ref{thm:main:2}, case (2)]
First, we show that \eqref{equ:main:2} always yields real values. That is,
\[
-1\leq \frac{\sum_{i=1}^k\frac{\lambda_{i+1}+\lambda_i}{\lambda_{i+1}-\lambda_i}}{\sqrt{\left(\sum_{i=1}^k\frac{1}{\lambda_{i+1}-\lambda_i}\right)\left(\sum_{i=1}^k\frac{(\lambda_{i+1}+\lambda_i)^2}{\lambda_{i+1}-\lambda_i}\right)}}\leq 1,
\]
for any real numbers $\lambda_k>\dots>\lambda_1$. By the Cauchy-Binet identity\cite{weisstein2002crc},
{\small
\begin{equation*}
        \left(\sum_{i=1}^k\frac{1}{\lambda_{i+1}-\lambda_i}\right)\left(\sum_{i=1}^k\frac{(\lambda_{i+1}+\lambda_i)^2}{\lambda_{i+1}-\lambda_i}\right) - \left(\sum_{i=1}^k\frac{\lambda_{i+1}+\lambda_i}{\lambda_{i+1}-\lambda_i}\right)^2 = \frac{1}{2}\sum_{i\neq j}\frac{(\lambda_{i+1}+\lambda_i -\lambda_{j+1}-\lambda_j)^2}{(\lambda_{i+1}-\lambda_i)(\lambda_{j+1}-\lambda_j)}.
\end{equation*}}
Lemma \ref{lem:ineq} implies that the right-hand side is positive.

Next, we will prove properties (1) to (4) in Definition \ref{defn:s5:1}. Properties (1) and (2) are proved similarly to the corresponding arguments in Subsection \ref{subsec:1}. To show that the property (3) holds, notice that the generalized eigenvalues of $(B,A)$ are $\lambda_i^{-1}$, which are values of an orientation-reversing M\"obius transformation of $\lambda_i$, $i=1,\dots,k$. By Lemma \ref{lem:3:3:0:1},
\[
\cos \theta(\sigma_2,\sigma_1) = \overline{t}(\lambda_1^{-1},\dots,\lambda_k^{-1}) = -\frac{\sum_{i=1}^k\frac{\lambda_{i+1}^{-1}+\lambda_{i}^{-1}}{\lambda_{i+1}^{-1}-\lambda_{i}^{-1}}}{\sqrt{\left(\sum_{i=1}^k\frac{1}{\lambda_{i+1}^{-1}-\lambda_{i}^{-1}}\right)\left(\sum_{i=1}^k\frac{(\lambda_{i+1}^{-1}+\lambda_{i}^{-1})^2}{\lambda_{i+1}^{-1}-\lambda_{i}^{-1}}\right)}}.
\]
Since
\begin{equation*}
    \begin{split}
        & \sum_{i=1}^k\frac{1}{\lambda_i^{-1}-\lambda_{i+1}^{-1}} = \sum_{i=1}^k\frac{\lambda_i\lambda_{i+1}}{\lambda_{i+1}-\lambda_{i}} = \sum_{i=1}^k\left(\frac{\lambda_{i+1}-\lambda_i}{4}+\frac{\lambda_i\lambda_{i+1}}{\lambda_{i+1}-\lambda_{i}}\right) = \sum_{i=1}^k\frac{(\lambda_{i+1}+\lambda_i)^2/4}{\lambda_{i+1}-\lambda_i},\\
        & \sum_{i=1}^k\frac{(\lambda_i^{-1}+\lambda_{i+1}^{-1})^2}{\lambda_i^{-1}-\lambda_{i+1}^{-1}} = \sum_{i=1}^k\left((\lambda_{i}^{-1} - \lambda_{i+1}^{-1}) +\frac{4}{\lambda_{i+1}-\lambda_i}\right) = \sum_{i=1}^k\frac{4}{\lambda_{i+1}-\lambda_i},\\
        & \sum_{i=1}^k\frac{\lambda_i^{-1}+\lambda_{i+1}^{-1}}{\lambda_i^{-1}-\lambda_{i+1}^{-1}} = \sum_{i=1}^k\frac{\lambda_{i+1}+\lambda_i}{\lambda_{i+1}-\lambda_i},
    \end{split}
\end{equation*}
we have $\overline{t}(\lambda_1^{-1},\dots,\lambda_k^{-1}) = \overline{t}(\lambda_1,\dots,\lambda_k)$, implying that $\theta(\sigma_2,\sigma_1) = \theta(\sigma_1,\sigma_2)$. Moreover, the generalized eigenvalues of $(-A,B)$ are $-\lambda_1>\dots>-\lambda_k$. Therefore,
\[
\cos \theta(-\sigma_1,\sigma_2) = \overline{t}(-\lambda_1,\dots,-\lambda_k) = -\frac{\sum_{i=1}^k\frac{\lambda_{i+1}+\lambda_i}{\lambda_{i+1}-\lambda_i}}{\sqrt{\left(\sum_{i=1}^k\frac{1}{\lambda_{i+1}-\lambda_i}\right)\left(\sum_{i=1}^k\frac{(\lambda_{i+1}+\lambda_i)^2}{\lambda_{i+1}-\lambda_i}\right)}},
\]
which implies that $\overline{t}(-\lambda_1,\dots,-\lambda_k) = -\overline{t}(\lambda_1,\dots,\lambda_k)$, i.e., $\theta(-\sigma_1,\sigma_2)= \pi - \theta(\sigma_1,\sigma_2)$.

Finally, we will prove the property (4). If we denote $\theta = \theta(\sigma_1,\sigma_2)$, $\theta_1 = \theta(\sigma_1,\sigma_3)$ and $\theta_2 = \theta(\sigma_3,\sigma_2)$, the property (4) reduces to
\begin{equation}\label{equ:1}\tag{$*$}
    \cos(\theta) = \cos(\theta_1)\cos(\theta_2) - \sin(\theta_1)\sin(\theta_2).
\end{equation}

Similarly to the proof given in Subsection \ref{subsec:1}, we assume that $\sigma_3 = (A+B)^\perp$ without loss of generality. The generalized eigenvalues of $(A,A+B)$ are $(1+\lambda_i)$ and $(1+\lambda_i^*)$, and the generalized eigenvalues of $(A+B,B)$ are $\frac{\lambda_i}{1+\lambda_i}$ and $\frac{\lambda_i^*}{1+\lambda_i^*}$, where $i=1,\dots,k$. Both sets of generalized eigenvalues are orientation-preserving M\"obius transformations of $\lambda_i$ and $\lambda_i^*$, $i=1,\dots,k$. Lemma \ref{lem:3:3:0:1} implies that
\begin{equation*}
        \cos(\theta_1) = \frac{\sum_{i=1}^k\frac{2+\lambda_{i+1}+\lambda_i}{\lambda_{i+1}-\lambda_i}}{\sqrt{\left(\sum_{i=1}^k\frac{1}{\lambda_{i+1}-\lambda_i}\right)\left(\sum_{i=1}^k\frac{(2+\lambda_{i+1}+\lambda_i)^2}{\lambda_{i+1}-\lambda_i}\right)}},
\end{equation*}
and
\begin{equation*}
    \begin{split}
        & \cos(\theta_2) = \frac{\sum_{i=1}^k\frac{\lambda_{i+1}+\lambda_i+2\lambda_i\lambda_{i+1}}{\lambda_{i+1}-\lambda_i}}{\sqrt{\left(\sum_{i=1}^k\frac{(1+\lambda_i)(1+\lambda_{i+1})}{\lambda_{i+1}-\lambda_i}\right)\left(\sum_{i=1}^k\frac{(\lambda_{i+1}+\lambda_i+2\lambda_i\lambda_{i+1})^2}{(1+\lambda_i)(1+\lambda_{i+1})(\lambda_{i+1}-\lambda_i)}\right)}}\\
        = & \frac{\sum_{i=1}^k\frac{(2+\lambda_i+\lambda_{i+1})(\lambda_{i+1}+\lambda_i)}{2(\lambda_{i+1}-\lambda_i)} + \frac{\lambda_i-\lambda_{i+1}}{2}}{\sqrt{\left(\sum_{i=1}^k\frac{(2+\lambda_i+\lambda_{i+1})^2}{4(\lambda_{i+1}-\lambda_i)}+\frac{\lambda_i-\lambda_{i+1}}{4}\right)\left(\sum_{i=1}^k\frac{(\lambda_{i+1}+\lambda_i)^2}{\lambda_{i+1}-\lambda_i} + \frac{\lambda_i^2}{1+\lambda_i} - \frac{\lambda_{i+1}^2}{1+\lambda_{i+1}}\right)}}\\
        = & \frac{\sum_{i=1}^k\frac{(2+\lambda_i+\lambda_{i+1})(\lambda_{i+1}+\lambda_i)}{\lambda_{i+1}-\lambda_i}}{\sqrt{\left(\sum_{i=1}^k\frac{(2+\lambda_i+\lambda_{i+1})^2}{\lambda_{i+1}-\lambda_i}\right)\left(\sum_{i=1}^k\frac{(\lambda_{i+1}+\lambda_i)^2}{\lambda_{i+1}-\lambda_i}\right)}}.
    \end{split}
\end{equation*}
By applying the Cauchy-Binet identity, we have
\begin{equation*}
    \begin{split}
        & \sin(\theta_1) = \frac{\sqrt{\left(\sum_{i=1}^k\frac{1}{\lambda_{i+1}-\lambda_i}\right)\left(\sum_{i=1}^k\frac{(2+\lambda_{i+1}+\lambda_i)^2}{\lambda_{i+1}-\lambda_i}\right) - \left(\sum_{i=1}^k\frac{2+\lambda_{i+1}+\lambda_i}{\lambda_{i+1}-\lambda_i}\right)^2}}{\sqrt{\left(\sum_{i=1}^k\frac{1}{\lambda_{i+1}-\lambda_i}\right)\left(\sum_{i=1}^k\frac{(2+\lambda_{i+1}+\lambda_i)^2}{\lambda_{i+1}-\lambda_i}\right)}}\\
        = & \frac{\sqrt{\frac{1}{2}\sum_{i\neq j}\frac{(\lambda_{i+1}+\lambda_i -\lambda_{j+1}-\lambda_j)^2}{(\lambda_{i+1}-\lambda_i)(\lambda_{j+1}-\lambda_j)}}}{\sqrt{\left(\sum_{i=1}^k\frac{1}{\lambda_{i+1}-\lambda_i}\right)\left(\sum_{i=1}^k\frac{(2+\lambda_{i+1}+\lambda_i)^2}{\lambda_{i+1}-\lambda_i}\right)}},
    \end{split}
\end{equation*}
and
\begin{equation*}
    \begin{split}
        & \sin(\theta_2) = \frac{\sqrt{\left(\sum_{i=1}^k\frac{(2+\lambda_i+\lambda_{i+1})^2}{\lambda_{i+1}-\lambda_i}\right)\left(\sum_{i=1}^k\frac{(\lambda_{i+1}+\lambda_i)^2}{\lambda_{i+1}-\lambda_i}\right) - \left(\sum_{i=1}^k\frac{(2+\lambda_i+\lambda_{i+1})(\lambda_{i+1}+\lambda_i)}{\lambda_{i+1}-\lambda_i}\right)^2}}{\sqrt{\left(\sum_{i=1}^k\frac{(2+\lambda_i+\lambda_{i+1})^2}{\lambda_{i+1}-\lambda_i}\right)\left(\sum_{i=1}^k\frac{(\lambda_{i+1}+\lambda_i)^2}{\lambda_{i+1}-\lambda_i}\right)}}\\
        = & \frac{\sqrt{\frac{1}{2}\sum_{i\neq j}\frac{4(\lambda_{i+1}+\lambda_i -\lambda_{j+1}-\lambda_j)^2}{(\lambda_{i+1}-\lambda_i)(\lambda_{j+1}-\lambda_j)}}}{\sqrt{\left(\sum_{i=1}^k\frac{(2+\lambda_i+\lambda_{i+1})^2}{\lambda_{i+1}-\lambda_i}\right)\left(\sum_{i=1}^k\frac{(\lambda_{i+1}+\lambda_i)^2}{\lambda_{i+1}-\lambda_i}\right)}}.
    \end{split}
\end{equation*}

Inequalities \eqref{equ:3:3:a} and \eqref{equ:3:3:b} imply that
\[
\begin{split}
    & \sin(\theta_1)\sin(\theta_2) = \frac{\frac{1}{2}\left|\sum_{i\neq j}\frac{2(\lambda_{i+1}+\lambda_i -\lambda_{j+1}-\lambda_j)^2}{(\lambda_{i+1}-\lambda_i)(\lambda_{j+1}-\lambda_j)}\right|}{\left|\sum_{i=1}^k\frac{(2+\lambda_{i+1}+\lambda_i)^2}{\lambda_{i+1}-\lambda_i}\right|\sqrt{\left(\sum_{i=1}^k\frac{1}{\lambda_{i+1}-\lambda_i}\right)\left(\sum_{i=1}^k\frac{(\lambda_{i+1}+\lambda_i)^2}{\lambda_{i+1}-\lambda_i}\right)}} \\
    & = \frac{\frac{1}{2}\sum_{i\neq j}\frac{2(\lambda_{i+1}+\lambda_i -\lambda_{j+1}-\lambda_j)^2}{(\lambda_{i+1}-\lambda_i)(\lambda_{j+1}-\lambda_j)}}{\left(\sum_{i=1}^k\frac{(2+\lambda_{i+1}+\lambda_i)^2}{\lambda_{i+1}-\lambda_i}\right)\sqrt{\left(\sum_{i=1}^k\frac{1}{\lambda_{i+1}-\lambda_i}\right)\left(\sum_{i=1}^k\frac{(\lambda_{i+1}+\lambda_i)^2}{\lambda_{i+1}-\lambda_i}\right)}}.
\end{split}
\]

By combining the equations above and using the Cauchy-Binet identity again, we have
\begin{equation*}
    \begin{split}
        &\cos(\theta_1)\cos(\theta_2) - \sin(\theta_1)\sin(\theta_2)\\
        = & \frac{\left(\sum_{i=1}^k\frac{2+\lambda_{i+1}+\lambda_i}{\lambda_{i+1}-\lambda_i}\right)\left(\sum_{i=1}^k\frac{(2+\lambda_i+\lambda_{i+1})(\lambda_{i+1}+\lambda_i)}{\lambda_{i+1}-\lambda_i}\right) - \frac{1}{2}\sum_{i\neq j}\frac{2(\lambda_{i+1}+\lambda_i -\lambda_{j+1}-\lambda_j)^2}{(\lambda_{i+1}-\lambda_i)(\lambda_{j+1}-\lambda_j)} }{\left(\sum_{i=1}^k\frac{(2+\lambda_{i+1}+\lambda_i)^2}{\lambda_{i+1}-\lambda_i}\right)\sqrt{\left(\sum_{i=1}^k\frac{1}{\lambda_{i+1}-\lambda_i}\right)\left(\sum_{i=1}^k\frac{(\lambda_{i+1}+\lambda_i)^2}{\lambda_{i+1}-\lambda_i}\right)}}\\
        = & \frac{\left(\sum_{i=1}^n\frac{\lambda_{i+1}+\lambda_i}{\lambda_{i+1}-\lambda_i}\right)\left(\sum_{i=1}^n\frac{(2+\lambda_{i+1}+\lambda_i)^2}{\lambda_{i+1}-\lambda_i}\right)}{\left(\sum_{i=1}^k\frac{(2+\lambda_{i+1}+\lambda_i)^2}{\lambda_{i+1}-\lambda_i}\right)\sqrt{\left(\sum_{i=1}^k\frac{1}{\lambda_{i+1}-\lambda_i}\right)\left(\sum_{i=1}^k\frac{(\lambda_{i+1}+\lambda_i)^2}{\lambda_{i+1}-\lambda_i}\right)}} = \cos(\theta).
    \end{split}
\end{equation*}
This proves the property (4) in Definition \ref{defn:s5:1}. In conclusion, the function $\theta$ given by \eqref{equ:main:2} is an invariant angle function.
\end{proof}
\subsection{Proof of Theorem \ref{thm:main:2}, case (3)}
According to Theorem \ref{thm:LR}, when $(A,B)$ has at most two distinct eigenvalues, all of which are real, the Lancaster-Rodman normal form consists of blocks with at most two distinct real eigenvalues, nontrivial singular blocks, and potentially a zero block. The following lemmas obtain $\operatorname{GL}^+(n,\mathbb{R})$-transformations that handle the real eigenvalue case and singular case respectively.
\begin{lem}\label{lem:new:1}
    Let $r\geq 1$ and use the notation $R_r$ and $N_r$ as in Theorem \ref{thm:LR}. 
    Then for any $s>0$ and $t\in\mathbb{R}$, there exists an element $g\in GL^+(r,\mathbb{R})$ satisfying that
    \[
    g^\mathsf{T}R_r g = s R_r + t N_r,\quad g^\mathsf{T}N_rg = N_r.
    \]
\end{lem}
\begin{proof}
    First we let
    \[
    U_r = \left(\left(\frac{r+1}{2}-i\right)\delta_{j-i,1}\right)_{i,j = 1}^r.
    \]
    We claim that
    \[
    U_r^\mathsf{T} N_r + N_r U_r = O_r,\quad U_r^\mathsf{T} R_r + R_r U_r = N_r.
    \]
    Indeed, the $(i,j)$-entry of $U_r^\mathsf{T} N_r + N_r U_r$ gives 
    \[
    \begin{split}
        & \sum_{k=1}^r\left(\frac{r+1}{2}-k\right)(\delta_{i-k,1}\delta_{k+j,r+2} + \delta_{j-k,1}\delta_{i+k,r+2}) \\
        = & \left(\frac{r+3}{2}-i+\frac{r+3}{2}-j\right)\delta_{i+j,r+3} = (r+3-(i+j))\delta_{i+j,r+3} = 0.
    \end{split}
    \]
    On the other hand, the $(i,j)$-entry of $U_r^\mathsf{T} R_r + R_r U_r$ gives
    \[
    \begin{split}
        & \sum_{k=1}^r\left(\frac{r+1}{2}-k\right)(\delta_{i-k,1}\delta_{k+j,r+1} + \delta_{j-k,1}\delta_{i+k,r+1}) \\
        = & \left(\frac{r+3}{2}-i+\frac{r+3}{2}-j\right)\delta_{i+j,r+2} = (r+3-(i+j))\delta_{i+j,r+2} = \delta_{i+j,r+2},
    \end{split}
    \]
    which is the $(i,j)$-entry of $N_r$.

    The matrix exponential $u_r = \exp(t U_r)$ thus gives
    \[
    u_r^\mathsf{T} N_ru_r = N_r,\quad u_r^\mathsf{T} R_ru_r = R_r + tN_r.
    \]
    In addition, $h_r = \mathrm{diag}(s^{1-i+r/2})_{i=1}^r$ gives
    \[
    h_r^\mathsf{T} N_rh_r = N_r,\quad h_r^\mathsf{T} R_rh_r = sR_r.
    \]
    Together, $g = u_rh_r$ satisfies the claim of the lemma.
\end{proof}
We require a similar base-change construction for singular pairs.
\begin{lem}\label{lem:sing_trans}
    For the pair
    \[
    (A^\mathrm{s}_r,B^\mathrm{s}_r) = \left(N_{2r+1},\left(\begin{array}{ccc}
            O_{r} & \mathbf{0}_r & R_{r} \\
            \mathbf{0}_r & 0 & \mathbf{0}_r \\
            R_r & \mathbf{0}_r & O_{r}
        \end{array}\right)\right),
    \]
    and $a,b,c,d\in\mathbb{R}$ with $ad-bc>0$, there is $g\in \operatorname{GL}^+(2r+1,\mathbb{R})$ such that:
    \[
    g^\mathsf{T}A^\mathrm{s}_rg = aA^\mathrm{s}_r + bB^\mathrm{s}_r,\quad g^\mathsf{T}B^\mathrm{s}_rg = cA^\mathrm{s}_r + dB^\mathrm{s}_r.
    \]
\end{lem}
\begin{proof}
    By considering the singular value decomposition of $\left(\begin{array}{cc}
        a & b \\
        c & d
    \end{array}\right)$, it suffices to construct $u  = u_\theta$, $h = h_{\alpha,\beta}$ satisfying that
    \[
    u.A = A\cos\theta + B\sin\theta,\quad u.B = B\cos\theta - A\sin\theta,
    \]
    and
    \[
    h.A = \alpha A, \quad h.B = \beta B,
    \]
    where $\theta\in\mathbb{R}$ and $\alpha,\beta>0$. The matrix $h$ can be taken as
    \[
    h = \operatorname{diag}(\alpha^{i-(r+1)}\beta^{-i+(r+3/2)})_{i=1}^{2r+1}.
    \]
    For the construction of $u$, let
    \[
    H_r = \left((i-r)\delta_{j-i,1}+j\delta_{i-j,1}\right)_{i,j=1}^r,\quad U_r = \operatorname{diag}(H_{r+1},H_r^{\mathsf{T}}),
    \]
    we claim that
    \[
    U_r^\mathsf{T} A^\mathrm{s}_r + A^\mathrm{s}_r U_r = B^\mathrm{s}_r,\quad U_r^\mathsf{T} B^\mathrm{s}_r + B^\mathrm{s}_r U_r = -A^\mathrm{s}_r.
    \]
    Writing these identities in block form, it suffices to show
    \[
    H_r\left(\mathbf{0}_r,R_r\right) + \left(\mathbf{0}_r,R_r\right)H_{r+1} = \left(R_r,\mathbf{0}_r\right),\quad H_r\left(R_r,\mathbf{0}_r\right) + \left(R_r,\mathbf{0}_r\right)H_{r+1} = -\left(\mathbf{0}_r,R_r\right).
    \]
    Indeed, for $1\leq i\leq r$ and $1\leq j\leq r+1$, the $(i,j)$-entry of $H_r\left(\mathbf{0}_r,R_r\right) + \left(\mathbf{0}_r,R_r\right)H_{r+1}$ gives
    \[
    \begin{split}
        & \sum_{k=1}^r((i-r)\delta_{k-i,1}+k\delta_{i-k,1})\delta_{k+j,r+2} + \sum_{k=1}^{r+1}\delta_{i+k,r+2}((k-r-1)\delta_{j-k,1}+j\delta_{k-j,1}) \\
        = & (i-r)\delta_{i+j,r+1} + (i-1)\delta_{i+j,r+3} + (j-r-2)\delta_{i+j,r+3} + j\delta_{i+j,r+1} \\
        = & (i+j-r)\delta_{i+j,r+1} + (i+j-r-3)\delta_{i+j,r+3} = \delta_{i+j,r+1},
    \end{split}
    \]
     which is the $(i,j)$-entry of $\left(R_r,\mathbf{0}_r\right)$. Here, the first summation $\sum_{k=1}^r\delta_{k-i,1}\delta_{k+j,r+2} = \delta_{i+j,r+1}$ does not hold for $i=r$ and $j=1$, but the factor $(i-r)$ vanishes in this case.

    The $(i,j)$-entry of $H_r\left(R_r,\mathbf{0}_r\right) + \left(R_r,\mathbf{0}_r\right)H_{r+1} $ gives
    \[
    \begin{split}
        & \sum_{k=1}^r((i-r)\delta_{k-i,1}+k\delta_{i-k,1})\delta_{k+j,r+1} + \sum_{k=1}^{r+1}\delta_{i+k,r+1}((k-r-1)\delta_{j-k,1}+j\delta_{k-j,1}) \\
        = & (i-r)\delta_{i+j,r} + (i-1)\delta_{i+j,r+2} + (j-r-2)\delta_{i+j,r+2} + j\delta_{i+j,r} \\
        = & (i+j-r)\delta_{i+j,r} + (i+j-r-3)\delta_{i+j,r+2} = -\delta_{i+j,r+2},
    \end{split}
    \]
    which is the $(i,j)$-entry of $-\left(\mathbf{0}_r,R_r\right)$. Here again, the second summation $\sum_{k=1}^r\delta_{i-k,1}\delta_{k+j,r+1} = \delta_{i+j,r+2}$ does not hold for $i=1$ and $j=r+1$, but the factor $(i-1)$ vanishes in this case.

    The matrix exponential $u_\theta = \exp(\theta U_r)$ thus gives the desired transformations.
\end{proof}

Lemmas \ref{lem:new:1} and \ref{lem:sing_trans} together imply the following:
\begin{lem}\label{lem:3:4:0:1}
    (1) Let $A,B\in\operatorname{Sym}_n(\mathbb{R})$ be linearly independent and the pencil $(A,B)$ has at most one distinct eigenvalue. Let $C = A$ if $(A,B)$ has no eigenvalues, and let $C = A-\lambda B$ if $\lambda\in\mathbb{R}$ is the only eigenvalue.
    
    Then, for any $s>0$ and $t\in\mathbb{R}$, there is an element $g\in GL^+(n,\mathbb{R})$ such that:
    \[
    g.C = C,\quad g.B = s B+t C.
    \]

    (2) Let $A,B\in\operatorname{Sym}_n(\mathbb{R})$ be linearly independent and the pencil $(A,B)$ has exactly two distinct eigenvalues $\lambda,\lambda'\in\mathbb{R}$, and let $C = A-\lambda B$, $C' = A - \lambda' B$. Then for any $s,s'>0$, there is an element $g\in GL^+(n,\mathbb{R})$ such that:
    \[
    g.C = s C,\quad g.C' = s' C'.
    \]
\end{lem}
\begin{proof}
    (1) Suppose that the pencil $(A,B)$ has no eigenvalues, or only one distinct eigenvalue $\lambda\in\mathbb{R}$. According to Theorem \ref{thm:LR}, we may assume that the matrix pencil is in the Lancaster-Rodman normal form:
    \[
    A = \mathrm{diag}(A^\mathrm{r}_1,\dots,A^\mathrm{r}_k,A^\mathrm{s}_1,\dots,A^\mathrm{s}_l,O_u),
    \]
    and
    \[
    B = \mathrm{diag}(B^\mathrm{r}_1,\dots,B^\mathrm{r}_k,B^\mathrm{s}_1,\dots,B^\mathrm{s}_l,O_u).
    \]
    Write
    \[
    C = A - \lambda B = \mathrm{diag}(C^\mathrm{r}_1,\dots,C^\mathrm{r}_k,C^\mathrm{s}_1,\dots,C^\mathrm{s}_l,O_u),
    \]
    where we make the convention $\lambda = 0$ in the case without generalized eigenvalues.
    
    For each real eigenvalue block $A^{\mathrm{r}}_j = \epsilon_j(\lambda R_{r_j}+N_{r_j})$ and $B^{\mathrm{r}}_j = \epsilon_jR_{r_j}$ (might be empty), 
    the block $C^\mathrm{r}_j = \epsilon_jN_{r_j}$. According to Lemma \ref{lem:new:1}, for any $s>0$ and $t\in\mathbb{R}$, there exist elements $g_j\in GL^+(r_j,\mathbb{R})$ such that
    \[
    g_j.C^\mathrm{r}_j = C^\mathrm{r}_j,\quad g_j. B^\mathrm{r}_j = s B^\mathrm{r}_j+ t C^\mathrm{r}_j.
    \]
    For each singular block $(A^\mathrm{s}_j,B^\mathrm{s}_j)$ of dimension $2r_j'+1$, with the linear combination $C^\mathrm{s}_j$, Lemma \ref{lem:sing_trans} implies that there exists $h_j\in \operatorname{GL}^+(2r_j'+1,\mathbb{R})$, such that
    \[
    h_j.C^\mathrm{s}_j = C^\mathrm{s}_j,\quad h_j. B^\mathrm{s}_j = s B^\mathrm{s}_j+ t C^\mathrm{s}_j.
    \]
    Let $g = \mathrm{diag} (g_1,\dots,g_k,h_1,\dots,h_l,I_u)\in GL^+(n,\mathbb{R})$, then $g.C = C$ and $g.B = s B+t C$.

    (2) Suppose that the pencil $(A,B)$ has exactly two distinct eigenvalues $\lambda,\lambda'\in\mathbb{R}$. We may assume that the matrix pencil is in the Lancaster-Rodman normal form:
    \[
    A = \mathrm{diag}(A^\mathrm{r}_1,\dots,A^\mathrm{r}_k,A^\mathrm{r'}_1,\dots,A^\mathrm{r'}_{k'},A^\mathrm{s}_1,\dots,A^\mathrm{s}_l,O_u),
    \]
    and 
    \[
    B = \mathrm{diag}(B^\mathrm{r}_1,\dots,B^\mathrm{r}_k,B^\mathrm{r'}_1,\dots,B^\mathrm{r'}_{k'},B^\mathrm{s}_1,\dots,B^\mathrm{s}_l,O_u).
    \]
    Write
    \[
    C = A - \lambda B = \mathrm{diag}(C^\mathrm{r}_1,\dots,C^\mathrm{r}_k,C^\mathrm{r'}_1,\dots,C^\mathrm{r'}_{k'},C^\mathrm{s}_1,\dots,C^\mathrm{s}_l,O_u),
    \]
    and
    \[
    C' = A - \lambda' B = \mathrm{diag}(C^\mathrm{'r}_1,\dots,C^\mathrm{'r}_k,C^\mathrm{'r'}_1,\dots,C^\mathrm{'r'}_{k'},C^\mathrm{'s}_1,\dots,C^\mathrm{'s}_l,O_u).
    \]
    For the blocks with eigenvalue $\lambda$, $C^\mathrm{r}_j = \epsilon_jN_{r_j}$, and $C^\mathrm{'r}_j = \epsilon_j(N_{r_j} + (\lambda - \lambda')R_{r_j})$; for the blocks with eigenvalue $\lambda'$, $C^\mathrm{r'}_j = \epsilon'_j(N_{r_j'} + (\lambda' - \lambda)R_{r_j'})$, and $C^\mathrm{'r'}_j = \epsilon'_jN_{r_j'}$. According to Lemma \ref{lem:new:1}, for any $s,s'>0$, there exist matrices $g_j$, $j=1,\dots, k$, such that
    \begin{equation*}
    g_j.C^\mathrm{r}_j = C^\mathrm{r}_j,\quad g_j.B^\mathrm{r}_j = \frac{s'}{s}B^\mathrm{r}_j + \frac{s'-s}{s(\lambda -\lambda')}C^\mathrm{r}_j,
    \end{equation*}
    and matrices $g_j'$, $j=1,\dots, k'$, such that
    \begin{equation*}
    g_j'.C^\mathrm{'r'}_j = C^\mathrm{'r'}_j,\quad g_j'.B^\mathrm{r'}_j = \frac{s}{s'}B^\mathrm{r'}_j + \frac{s-s'}{s'(\lambda' -\lambda)}C^\mathrm{'r'}_j.
    \end{equation*}
    For each singular block $(A^\mathrm{s}_j,B^\mathrm{s}_j)$ of dimension $2r_j''+1$, with the distinct linear combinations $C^\mathrm{s}_j$ and $C^\mathrm{'s}_j$, Lemma \ref{lem:sing_trans} implies that there exists $h_j\in \operatorname{GL}^+(2r_j''+1,\mathbb{R})$, such that
    \[
    h_j.C^\mathrm{s}_j = sC^\mathrm{s}_j,\quad h_j. C^\mathrm{'s}_j = s'C^\mathrm{'s}_j.
    \]
    Let $g = \mathrm{diag}(\sqrt{s}g_1,\dots,\sqrt{s}g_k,\sqrt{s'}g_1',\dots,\sqrt{s'}g_{k'}',h_1,\dots,h_l,I_u)\in GL^+(n,\mathbb{R})$, then $g.C = sC$ and $g.C' = s'C'$.
\end{proof}
Lastly, we come back to the proof of the main theorem.
\begin{proof}[Proof of Theorem \ref{thm:main:2}, case (3)]
    Recall that a positive scaling of the normal vector does not change the associated co-oriented hyperplane. Therefore, we can replace the part ``$g\in SL(n,\mathbb{R})$'' in Definition \ref{defn:s5:1} with ``$g\in GL^+(n,\mathbb{R})$''. We will prove the statement (3) of the theorem by exhausting all possible cases. By Lemma \ref{lem:mob}, up to a change of basis, we assume every generalized eigenvalue of $(A,B)$ is finite.

    (1). Suppose that $(A,B)$ has only one eigenvalue $\lambda\in\mathbb{R}$, and $\lambda>0$. By Lemma \ref{lem:3:4:0:1}, there is an element $g\in GL^+(n,\mathbb{R})$, such that $g.C = C$ and $g.B = \lambda B+C = A$. For $k\in\mathbb{Z}$, denote $A_k = g^{1-k}.A$ and $\sigma_k = A_k^\perp$. Then $A_1 = A$, $A_2 = B$, and
    \begin{equation*}
    \theta(\sigma_k,\sigma_{k+1}) = \theta(A_{k}^\perp,A_{k+1}^\perp) = \theta((g^{1-k}.A)^\perp,(g^{1-k}.B)^\perp) = \theta(A^\perp,B^\perp) = \theta(\sigma_1,\sigma_2).
    \end{equation*}
    The relations $g.C = C$ and $g.B = \lambda B+C = A$ imply that $g.(A-\lambda B) = (A-\lambda B)$ and $g.(A - B) = \lambda(A - B)$. Thus,
    \[
    (A_k-\lambda A_{k+1}) = g^{1-k}.(A-\lambda B) = A-\lambda B,\quad (A_k - A_{k+1}) = g^{1-k}.(A - B) = \lambda^{1-k}(A-B),
    \]
    and
    \[
    (\lambda^k - 1)A_k = (\lambda^k-\lambda)A_{k+1} + (\lambda - 1)A_1,
    \]
    for all $k>1$. Therefore, $A_k$ is a positive linear combination of $A_{1}$ and $A_{k+1}$, i.e., $\theta_k$ lies between $\theta_1$ and $\theta_{k+1}$. The property (4) of invariant functions implies that
    \begin{equation*}
    \theta(\sigma_1,\sigma_m) = \sum_{k=1}^{m-1}\theta(\sigma_k,\sigma_{k+1}) = (m-1)\theta(\sigma_1,\sigma_2)
    \end{equation*}
    for any $m>1$. However, for $m$ large enough, $\theta(\sigma_1,\sigma_m) = (m-1)\theta(\sigma_1,\sigma_2)>\pi$, a contradiction with the property (1) of invariant angle functions.

    (2). Suppose that $(A,B)$ has only one eigenvalue $\lambda\in\mathbb{R}$, and $\lambda<0$. The previous case implies that $\theta((-A)^\perp,B^\perp)$ does not exist. If $\theta(A^\perp,B^\perp)$ exists, then $\theta((-A)^\perp,B^\perp) = \pi - \theta(A^\perp,B^\perp)$, a contradiction.

    (3). Suppose that $(A,B)$ has no eigenvalues or has exactly one distinct eigenvalue $\lambda = 0$. In either case, $C = A$. There exists an element $g\in GL^+(n,\mathbb{R})$ such that
    \[
    g.A = A,\quad g.B = B+A.
    \]
    Thus
    \[\theta(A^\perp,B^\perp) = \theta(A^\perp, (A+B)^\perp) + \theta(B^\perp, (A+B)^\perp)> \theta(A^\perp, (A+B)^\perp) = \theta((g.A)^\perp,(g.B)^\perp),
    \]
    a contradiction with property (2) of invariant angles.
    
    (4). Suppose that the pencil $(A,B)$ has exactly two distinct eigenvalues $\lambda$ and $\lambda'$, both $A$ and $B$ are positive linear combinations of $C = A-\lambda B$ and $C' = A-\lambda' B$. Lemma \ref{lem:3:4:0:1} implies the existence of an element $g\in GL^+(n,\mathbb{R})$ such that $g.A = B$, $g.C\sim C$, and $g.C'\sim C'$. Here, $g.C\sim C$ denotes that $g.C$ differs from $C$ by a positive multiple. Let $A_k = g^{1-k}.A$, and $\theta_k = A_k^\perp$, then
    \[
    \theta(\sigma_1,\sigma_m) = (m-1)\theta(\sigma_1,\sigma_2),
    \]
    similarly to case (1). The angle exceeds $\pi$ for $m$ large enough, leading to a contradiction with the property (1) of invariant angle functions.

    (5-7). Suppose that the pencil $(A,B)$ has exactly two distinct eigenvalues $\lambda$ and $\lambda'$, each of $A$ and $B$ is either a positive linear combination of $C = A-\lambda B$ and $C' = A-\lambda' B$, or a positive linear combination of $-C$ and $-C'$. Then either $(A,-B)$, $(-A,B)$, or $(-A,-B)$ is in case (4). Therefore, $\theta(A^\perp,B^\perp)$ does not exist.

    (8). Suppose that the pencil $(A,B)$ has exactly two distinct eigenvalues $\lambda$ and $\lambda'$, where $A$ is a positive linear combination of $C$ and $C'$, while $B$ is a positive linear combination of $-C$ and $C'$. Note that $C'$ is a positive linear combination of $A$ and $B$. Lemma \ref{lem:3:4:0:1} yields an element $g\in GL^+(n,\mathbb{R})$ such that $g.C = s C$, $g.C' = s' C'$, where $s'>s$. The latter assumption implies that $g.A$ is a positive linear combination of $A$ and $C'$, and $g.B$ is a positive linear combination of $B$ and $C'$. Consequently,
    \begin{equation}
    \theta(A^\perp,B^\perp) = \theta(A^\perp,(g.A)^\perp) +\theta((g.A)^\perp,(g.B)^\perp) + \theta(B^\perp,(g.B)^\perp)>\theta((g.A)^\perp,(g.B)^\perp),
    \end{equation}
    which is a contradiction with property (2) of invariant angles.

    (9-11). Suppose that the pencil $(A,B)$ has exactly two distinct eigenvalues $\lambda$ and $\lambda'$, where $A$ is either a positive linear combination of $C$ and $C'$, or a positive linear combination of $-C$ and $-C'$, and $B$ is either a positive linear combination of $C$ and $-C'$, or a positive linear combination of $-C$ and $C'$. Then either $(A,-B)$, $(-A,B)$, or $(-A,-B)$ is in case (8). Hence, $\theta(A^\perp,B^\perp)$ does not exist.

    In conclusion, for all symmetric matrix pencils $(A,B)$ of type (3), $(A^\perp,B^\perp)$ is not in the domain of any invariant angle function.
\end{proof}

\vspace{12pt}
\section{Criterion for intersecting hyperplanes}\label{sec:s3}
We will describe in this section a necessary and sufficient condition for two hyperplanes in $\mathcal{P}(n)$ to be disjoint. We will also present a sufficient condition for the bisectors $Bis(X,Y)$ and $Bis(Y,Z)$ to be disjoint, in terms of the Selberg's invariants and angles between points $X$, $Y$, and $Z\in\mathcal{P}(n)$.
\subsection{A criterion for intersecting hyperplanes}\label{sec:4:1}
We seek an equivalent condition for $\bigcap_{i\in\mathcal{I}}\sigma_i \neq \varnothing$, where $\mathcal{I}$ is a finite set. Let $A_i\in Sym_n(\mathbb{R})$ be the normal vector of $\sigma_i$ for $i\in\mathcal{I}$, we denote the collection of symmetric matrices,
\[
\mathcal{A} = \{A_i\in Sym_n(\mathbb{R})| i\in\mathcal{I}\}.
\]
Moreover, we denote the collection of hyperplanes,
\[
\Sigma = \{\sigma_i\subset \mathcal{P}(n)|i\in\mathcal{I}\}.
\]
The definiteness of a collection of symmetric $n \times n$ matrices is defined as follows:
\begin{defn}
    We say the collection $\mathcal{A} = \{A_i\in Sym_n(\mathbb{R})| i\in\mathcal{I}\}$ is (semi-) definite if there exist numbers $c_i\in\mathbb{R}$ for $i\in\mathcal{I}$ such that
    \[
    A = \sum_{i\in\mathcal{I}}c_i A_i
    \]
    is a non-zero positive (semi-) definite matrix.
\end{defn}
\begin{rmk}
    The collection $\{A\}$ consisting of a single $n\times n$ symmetric matrix is definite if and only if $A$ is either positive or negative definite. The collection $\{A,B\}$ is definite if and only if the symmetric matrix pencil $(A,B)$ is definite.
\end{rmk}
We will need more notations, related to the Satake compactification $\overline{\mathcal{P}(n)}\subset \mathbf{P}(Sym_n(\mathbb{R}))$ (See Definition \ref{defn:satake}):
\begin{defn}
    For $A\in Sym_n(\mathbb{R})$, define
    \[
    N(A) = \{X\in \mathbf{P}(Sym_n(\mathbb{R}))| \mathrm{tr}(A\cdot X) = 0\},
    \]
    and define $\overline{A^\perp} = \overline{\mathcal{P}(n)}\cap N(A)$.
\end{defn}
The main result of this subsection establishes a relationship between the definiteness of $\mathcal{A} = \{A_i\}$ and the emptiness of the intersection $\bigcap A_i^\perp$:
\begin{lem}\label{lem:1}
    The collection $\mathcal{A} = \{A_i\}_{i=1}^k$ of $n\times n$ symmetric matrices is semi-definite if and only if the intersection $\bigcap_{i=1}^k A_i^\perp$ is empty. Furthermore, $\mathcal{A}$ is (strictly) definite if and only if $\bigcap \overline{A_i^\perp}=\varnothing$.
\end{lem}
\begin{proof}
    Throughout the proof, we use the projective model of the space $\mathcal{P}(n)$.
    
    First, we prove the statement for $k=1$, i.e., $\mathcal{A}$ consists of a single matrix $\{A\}$. By applying the $SL(n,\mathbb{R})$-action, we assume that $A$ is a diagonal matrix $\mathrm{diag}(I_p,-I_q,O_s)$ without loss of generality. 
    
    We begin by assuming that $A$ is definite, i.e., $A = I_n$ or $A = -I_n$. For any point in $\overline{\mathcal{P}(n)}$ represented by an $n\times n$ semi-definite matrix $X$, either $\mathrm{tr}(A.X) = \mathrm{tr}(X)>0$ or $\mathrm{tr}(A.X) = -\mathrm{tr}(X)<0$ holds. In both cases, $\overline{A^\perp} = \varnothing$.

    Next, assume that $A$ is semi-definite but not definite, i.e., $A = \pm \mathrm{diag}(I_p,O_s)$ with $s>0$. In this case, $A^\perp = \varnothing$ for the same reason. However, the matrix $X = \mathrm{diag}(O_p, I_s)$ represents a point in $\partial_S\mathcal{P}(n)$ that satisfies $\mathrm{tr}(A.X) = 0$, implying that $\overline{A^\perp}$ is non-empty.

    Lastly, assume that $A$ is indefinite, i.e., $A = \mathrm{diag}(I_p,-I_q,O_s)$ with $p,q>0$. A matrix $X = \mathrm{diag}(I_p/p,I_q/q,I_s)$ represents a point in $\mathcal{P}(n)$ that satisfies $\mathrm{tr}(A.X) = 0$, implying that $A^\perp$ is non-empty.

    Having established the statement for $k=1$, we now extend it to general $k\in\mathbb{N}$. We first assume that $\mathcal{A} = \{A_1,\dots,A_k\}$ is definite. By definition, there exist real numbers $c_i\in\mathbb{R}$ for $i=1,\dots,k$ such that $\sum c_iA_i$ is positive definite. If $\bigcap \overline{A_i^\perp}$ is non-empty, say $X\in \bigcap \overline{A_i^\perp}$, then $\mathrm{tr}(X.(\sum c_iA_i)) = \sum c_i.\mathrm{tr}(X.A_i) = 0$. However, the statement for $k=1$ implies that $\mathrm{tr}(X\cdot (\sum c_iA_i))>0$, leading to a contradiction.

    Next, assume that $\mathcal{A}$ is semi-definite but not definite. Then $\bigcap A_i^\perp = \varnothing$, for the same reason as in the previous case. On the other hand, if $\bigcap \overline{A_i^\perp}$ is empty, then the subspace $\bigcap N(A_i)\subset \mathbf{P}(Sym(n))$ is disjoint from the closed convex region $\overline{\mathcal{P}(n)}\subset \mathbf{P}(Sym(n))$. Therefore, there exists a support hyperplane $N(B)\subset \mathbf{P}(Sym(n))$ such that $\bigcap N(A_i)\subseteq N(B)$ and $N(B)\cap \overline{\mathcal{P}(n)} = \varnothing$. The first condition implies that $B\in span(A_1,\dots,A_k)$, while the second condition, together with the statement for $k=1$, implies that $B$ is definite. However, our assumption that $\mathcal{A}$ is indefinite contradicts this conclusion.

    Lastly, assume that $\mathcal{A}$ is not semi-definite. Analogously to the previous case, if $\bigcap A_i^\perp$ is empty, there exists a supporting hyperplane $N(B)\subset \mathbf{P}(Sym(n))$ such that $\bigcap N(A_i)\subseteq N(B)$ and $N(B)\cap \mathcal{P}(n) = \varnothing$. This leads to a contradiction for a similar reason.
\end{proof}

\subsection{Classification of disjoint hyperplane pairs}\label{sec:4:3}
In the case $k=2$, the Lancaster-Rodman normal form provides a more straightforward criterion for the disjointness condition. We begin by discussing the definiteness of $A-t B$ for $t\in \overline{\mathbb R}$ (realized as $-B$ for $t = \infty$), where $(A,B)$ is a Lancaster-Rodman block of a symmetric matrix pencil:
\begin{lem}\label{lem:LR_pd}
    Let $(A,B)$ be a nonzero Lancaster-Rodman block pencil as detailed in Theorem \ref{thm:LR}. The following characterizes the definiteness of $A - tB$:
    \begin{itemize}
        \item If $(A,B) = (A^{\mathrm{r}},B^{\mathrm{r}}) = (\alpha R_k+N_k,R_k)$ for a real eigenvalue $\alpha\in \mathbb{R}$, or $(R_k,-N_k)$ for $\alpha = \infty$:
        \begin{itemize}
            \item If $k = 1$, then $A - tB$ is positive definite if $t<\alpha$, zero if $t = \alpha$, and negative definite if $t>\alpha$.
            \item If $k = 2$, then $A-tB$ is positive semi-definite if $t = \alpha$, and indefinite if $t\neq \alpha$.
            \item If $k\geq 3$, then $A - tB$ is always indefinite.
        \end{itemize}
        \item If $(A,B) = (A^\mathrm{c},B^\mathrm{c})$ with a pair of nonreal eigenvalues, then $A-tB$ is always indefinite.
        \item If $(A,B) = (A^\mathrm{s},B^\mathrm{s})$ is singular, then $A-tB$ is always indefinite.
    \end{itemize}
\end{lem}
\begin{proof}
    First, we suppose $(A,B) = (\alpha R_k+N_k, R_k)$ for $\alpha\in\mathbb{R}$, then $A-tB = (\alpha - t)R_k + N_k$ for $t\in \mathbb{R}$, or $-B = -R_k$ for $t = \infty$. 
    
    If $k = 1$, $A - tB = [\alpha - t]$, which is positive if $t<\alpha$, negative if $t>\alpha$, and zero if $t = \alpha$; when $t = \infty$, $A - tB = [-1]$ is negative. 
    
    If $k = 2$, $A - tB = \begin{pmatrix}
        0 & \alpha - t \\ \alpha - t & 1
    \end{pmatrix}$, which is indefinite if $\alpha - t\neq 0$, or $t \neq \alpha$, and is positive semi-definite if $t = \alpha$. $A - tB = \begin{pmatrix}
        0 & -1 \\ -1 & 0
    \end{pmatrix}$ is indefinite if $t = \infty$.

    When $k\geq 3$ and $t\neq \infty$, the $\{2,k\}$-principal minor of $A-tB$ equals to $-1$. When $t = \infty$, the $\{1,k\}$-principal minor equals to $-1$. In both cases, $A-tB$ is indefinite.

    Next, suppose $(A,B) = (R_k,-N_k)$ so $\alpha = \infty$, then $A - tB = R_k + t N_k$ for $t\in\mathbb{R}$, or $-B = N_k$ if $t = \infty$.

    If $k = 1$, $A - tB = [1]$ and thus positive if $t\neq \infty$, and is zero if $t = \infty$.

    If $k = 2$, $A - tB = \begin{pmatrix}
        0 & 1 \\ 1 & t
    \end{pmatrix}$ is indefinite for $t\neq \infty$, and $A - tB = \operatorname{diag}(0,1)$ is positive semi-definite for $t = \infty$.

    When $k\geq 3$ and $t\neq \infty$, the $\{1,k\}$-principal minor of $A-tB$ equals to $-1$. When $t = \infty$, the $\{2,k\}$-principal minor equals to $-1$. In both cases, $A-tB$ is indefinite.

    If $(A,B)$ is a block with nonreal eigenvalues or a nontrivial singular block, then $A$ and $B$ have zero trace. For every $t\in\overline{\mathbb R}$, the matrix $A-tB$ is nonzero and real symmetric with zero trace, and hence indefinite.
\end{proof}
Using Lemma \ref{lem:LR_pd}, we have the following main theorem of this section, which characterizes the disjointness of hyperplanes in $\mathcal{P}(n)$ in terms of the corresponding normal matrix pencil $(A,B)$.
\begin{thm}\label{thm:main:1}
    Let $A^\perp$ and $B^\perp$ be hyperplanes in $\mathcal P(n)$, and suppose that the pencil $(A,B)$ is in Lancaster–Rodman normal form.
    \begin{enumerate}
        \item The compactifications $\overline{A^\perp}$ and $\overline{B^\perp}$ are disjoint if and only if:
        \begin{itemize}
            \item Every block is one-dimensional, with eigenvalue $\alpha_j\in\overline{\mathbb R}$ and sign characteristic $\epsilon_j\in\{1,-1\}$.
            \item There exist $t\in\mathbb R\setminus\{\alpha_j\}_j$ and $\epsilon\in\{1,-1\}$ such that $\epsilon_j=-\epsilon$ whenever $\alpha_j<t$, and $\epsilon_j=\epsilon$ whenever $\alpha_j>t$.
        \end{itemize}
        \item The hyperplanes $A^\perp$ and $B^\perp$ are disjoint if and only if there exist $\alpha\in\overline{\mathbb R}$ and $\epsilon\in\{1,-1\}$, with $A-\alpha B\ne O_n$, such that:
        \begin{itemize}
            \item Every block is either zero, one-dimensional with eigenvalue $\alpha_j\in\overline{\mathbb R}$ and sign characteristic $\epsilon_j\in\{1,-1\}$, or two-dimensional with eigenvalue $\alpha$ and sign characteristic $\epsilon$.
            \item For every one-dimensional block, $\epsilon_j=-\epsilon$ whenever $\alpha_j<\alpha$, and $\epsilon_j=\epsilon$ whenever $\alpha_j>\alpha$, with no restriction when $\alpha_j=\alpha$.
        \end{itemize}
    \end{enumerate}
\end{thm}
\begin{proof}
    By Lemma \ref{lem:1}, $\overline{A^\perp}$ and $\overline{B^\perp}$ are disjoint if and only if the pencil $(A,B)$ is strictly definite, i.e., $A - tB\succ 0$ or $\prec 0$ for certain $t\in\overline{\mathbb R}$. Also,  $A^\perp$ and $B^\perp$ are disjoint if and only if $(A,B)$ is (semi-) definite, i.e., $A - tB\succeq 0$ or $\preceq 0$ for certain $t\in\overline{\mathbb R}$. We prove the equivalence between the given conditions and the definiteness respectively.
    
    \textbf{Claim (1).} Suppose $A - tB$ is strictly definite for certain $t\in\overline{\mathbb{R}}$. By Lemma \ref{lem:LR_pd}, this happens only if every block of $(A,B)$ is one-dimensional with a real eigenvalue, and $t$ is not an eigenvalue $\alpha_j$. If $A - tB\succ 0$, then Lemma \ref{lem:LR_pd} also requires that $\epsilon_j = 1$ for $\alpha_j > t$, and $\epsilon_j = -1$ for $\alpha_j < t$. If $A - tB\prec 0$, then Lemma \ref{lem:LR_pd} requires the opposite sign characteristics. These together form the condition in claim (1).

    \textbf{Claim (2).} Suppose $A - tB$ is definite for certain $t\in\overline{\mathbb{R}}$. By Lemma \ref{lem:LR_pd}, this happens only if every block of $(A,B)$ is either zero, or a one- or two-dimensional block with a real eigenvalue. If $A - tB\succeq 0$, then Lemma \ref{lem:LR_pd} also requires that:
    \begin{itemize}
        \item For a one-dimensional block with eigenvalue $\alpha_j$ and sign characteristics $\epsilon_j$, $\epsilon_j = 1$ for $\alpha_j > t$, $\epsilon_j = -1$ for $\alpha_j < t$, with no restriction when $\alpha_j = t$.
        \item For any two-dimensional block, its eigenvalue must be $\alpha = t$ and the sign characteristics must be $\epsilon = 1$.
    \end{itemize}
    If $A - tB\preceq 0$, then Lemma \ref{lem:LR_pd} requires the opposite sign characteristics. These together form the condition in claim (2).
\end{proof}
\subsection{A sufficient condition for intersecting bisectors}\label{sec:4:2}
In this section, we formulate a sufficient condition for disjointness of bisectors $Bis(X,Y)$ and $Bis(Y,Z)$ in terms of the distances and angle between $X,Y$ and $Z\in\mathcal{P}(n)$. This is analogous to a sufficient condition proposed in the context of hyperbolic spaces, \cite{kapovich2018geometric}. We begin with dividing the model flat of $\mathcal{P}(n)$\cite{kapovich2017anosov} into chambers:
\begin{defn}
    The \textbf{model flat} of $\mathcal{P}(n)$ is the set:
    \[
    F_{mod} = \{diag(x_1,\dots,x_n)|x_i>0,\ \prod x_i = 1\}.
    \]
\end{defn}
Consider a given point $Y\in \mathcal{P}(n)$ and a given maximal totally-geodesic flat submanifold $F\ni Y$ in $\mathcal{P}(n)$. There exists an element $g\in SL(n,\mathbb{R})$ that maps $F$ to the model flat $F_{mod}$ and maps $Y$ to the identity, \cite{kapovich2017anosov}. As $I = g.Y = (Y^{1/2}g)^T(Y^{1/2}g)$, we have that $Y^{1/2}g\in SO(n)$. 
\begin{defn}\label{defn:4:4}
    We divide the model flat $F_{mod}$ of $\mathcal{P}(n)$ into $(2^n-2)$ chambers, denoted by
    \[
    \Delta^\mathcal{I} = \{X = \mathrm{diag}(x_i)\in F_{mod}|0<x_i< 1,\ \forall i\in\mathcal{I};\quad x_i> 1,\ \forall i\notin\mathcal{I}\}.
    \]

    For any number $t\in (0,1)$, define
    \[
    \Delta^\mathcal{I}_t = \left\{X\in \Delta^\mathcal{I}\left| \frac{\min |\log x_i|}{\max |\log x_i|}\geq t\right\}\right. .
    \]
    $\Delta^\mathcal{I}_t$ is a cone contained in the chamber $\Delta^\mathcal{I}$ and is away from the chamber boundary.
\end{defn}
\begin{thm}\label{thm:4:2:2}
    Let $X,Y,Z$ be points in $\mathcal{P}(n)$, and $L = \min(s(Y,X),s(Y,Z))$. There exist elements $g_X$ and $g_Z\in SL(n,\mathbb{R})$ that map $Y$ to the identity and map $X$ and $Z$ into $F_{mod}$, respectively. Let $\theta$ be the maximum angle between the corresponding column vectors of $Y^{1/2}g_X$ and $Y^{1/2}g_Z\in SO(n)$.
    
    Suppose that there is a number $t\in (0,1)$ and a subset $\mathcal{I}\subset \{1,\dots,n\}$ such that the points $g_X.X\in \Delta^\mathcal{I}_t$, $g_Z.Z\in \Delta^{\mathcal{I}^{\mathsf{c}}}_t$, and
    \begin{equation}\label{equ:4:3}
        0<\frac{1+\sqrt{n-2}\sin\theta}{\cos\theta - \sqrt{n-2}\sin\theta}\leq \sqrt{t}\cdot\left(\frac{L-1}{n-1}\right)^{t/2}.
    \end{equation}
    Then the bisectors $Bis(X,Y)$ and $Bis(Y,Z)$ in $\mathcal{P}(n)$ are disjoint.
\end{thm}
Lemmas for proving Theorem \ref{thm:4:2:2} are presented below.
\begin{lem}\label{lem:4:2:2:0}
    Let $X = \mathrm{diag}(x_i)\in \Delta^\mathcal{I}_t$ and $s(I,X)\geq L$. For any $i\in\mathcal{I}$ and $j\in \mathcal{I}^{\mathsf{c}}$,
    \[
    \frac{|x_i^{-1}-1|}{|x_j^{-1}-1|}\geq t\cdot\left(\frac{L-1}{n-1}\right)^t.
    \]
\end{lem}
\begin{proof}
    Without loss of generality, assume that $\mathcal{I} = \{1,\dots,k\}$, where $k<n$. Since $X\in \Delta^\mathcal{I}_t$, there exists $u>0$ such that
    \[
    e^{tu}\leq x_i\leq e^u,\ \forall i>k; \quad e^{-u}\leq x_i\leq e^{-tu},\ \forall i\leq k.
    \]
    Since $s(I,X) = \sum x_i\geq L$,
    \[
    (n-k)(e^u-1) \geq L - ke^{-tu} - (n-k)\geq L - n.
    \]
    Let $e^u - 1 = v$, then 
    \[
    v\geq \frac{L-n}{n-k}\geq \frac{L-n}{n-1}.
    \]
    For any $i\in\mathcal{I}$ and $j\in \mathcal{I}^{\mathsf{c}}$,
    \[
    \frac{|x_i^{-1}-1|}{|x_j^{-1}-1|}\geq \frac{e^{tu} - 1}{1-e^{-u}} = \frac{(1+v)^t-1}{1-(1+v)^{-1}}.
    \]
    It is self-evident that
    \[
    \frac{(1+v)^t-1}{1-(1+v)^{-1}} = (1+v)\cdot \frac{(1+v)^t-1}{(1+v)-1}\geq (1+v)\frac{d}{dv}(1+v)^t = t(1+v)^t.
    \]
    Therefore,
    \[
    \frac{|x_i^{-1}-1|}{|x_j^{-1}-1|}\geq t(1 +v)^t\geq t\cdot\left(\frac{L-1}{n-1}\right)^t.
    \]
\end{proof}
\begin{lem}\label{lem:4:2:2:1}
    Suppose that $g = \begin{pmatrix}
        g_1 & g_2\\ g_3 & g_4
    \end{pmatrix}\in SO(n)$, where $g_1\in Mat_k(\mathbb{R})$. Then, $g = g_+g_-^{-1}$, where
    \[
    g_+ = \begin{pmatrix}(g_1^{-1})^{\mathrm{T}} & -(g_1^{-1})^{\mathrm{T}}g_3^{\mathrm{T}} \\ O & I\end{pmatrix},\quad g_- = \begin{pmatrix}I & O\\ -g_4^{-1}g_3 & g_4^{-1}\end{pmatrix}.
    \]
\end{lem}
\begin{proof}
    Notice that
    \[
    \begin{pmatrix}
        g_1 & g_2\\g_3 & g_4
    \end{pmatrix}
    \begin{pmatrix}
        I & O\\ -g_4^{-1}g_3 & g_4^{-1}
    \end{pmatrix} = 
    \begin{pmatrix}
        g_1 - g_2g_4^{-1}g_3 & g_2g_4^{-1}\\ O & I
    \end{pmatrix}.
    \]
    It suffices to prove that $(g_1^{-1})^{\mathrm{T}} = g_1 - g_2g_4^{-1}g_3$ and $(g_1^{-1})^{\mathrm{T}}g_3^{\mathrm{T}} = -g_2g_4^{-1}$. Indeed, since $g^{\mathrm{T}}g = I$, we have
    \[
    g_1^{\mathrm{T}}g_2 + g_3^{\mathrm{T}}g_4 = O,
    \]
    therefore
    \[
    (g_1^{-1})^{\mathrm{T}}g_3^{\mathrm{T}} = - (g_1^{-1})^{\mathrm{T}}(g_1^{\mathrm{T}}g_2g_4^{-1}) = -g_2g_4^{-1}.
    \]
    Since $gg^{\mathrm{T}} = I$, we have
    \[
    g_1g_1^{\mathrm{T}} + g_2g_2^{\mathrm{T}} = I,\ g_3g_1^{\mathrm{T}} + g_4g_2^{\mathrm{T}} = O,
    \]
    therefore
    \[
    (g_1^{-1})^{\mathrm{T}} = g_1 + g_2g_2^{\mathrm{T}}(g_1^{-1})^{\mathrm{T}} = g_1 - g_2(g_4^{-1}g_3g_1^{\mathrm{T}})(g_1^{-1})^{\mathrm{T}} = g_1 - g_2g_4^{-1}g_3.
    \]
\end{proof}
\begin{lem}\label{lem:4:2:2:2}
    Define
    \[
    \sigma_r(A) = \max_i \sum_{j=1}^n|a_{ij}|,\quad \sigma_c(A) = \max_j \sum_{i=1}^n|a_{ij}|
    \]
    for a matrix $A = (a_{ij})\in Mat_n(\mathbb{R})$. If there exist elements $A,B\in Mat_n(\mathbb{R})$ such that $\sigma_r(A)\leq a$ and $\sigma_r(B)\leq b$, then $\sigma_r(AB)\leq ab$. A similar conclusion holds for $\sigma_c$.
\end{lem}
\begin{proof}
    We present the proof for $\sigma_r$ only. For any $1\leq i\leq n$,
    \[
    \sum_{j=1}^n\left|\sum_{k=1}^n a_{ik}b_{kj}\right|\leq \sum_{j,k=1}^n |a_{ik}||b_{kj}| = \sum_{k=1}^n|a_{ik}|\left(\sum_{j=1}^n|b_{kj}|\right)\leq \sum_{k=1}^n|a_{ik}|\cdot b\leq ab.
    \]
    Thus, $\sigma_r(AB)\leq ab$.
\end{proof}
\begin{lem}\label{lem:4:2:2:3}
    Consider a matrix $A = (a_{ij})\in Mat_n(\mathbb{R})$, where $a_{ii}\geq a$ and $\sum_{j\neq i}|a_{ij}|\leq a'$ for all $i=1,\dots,n$, and $a>a'$ are real numbers. Then $A$ is invertible, with $\sigma_r(A^{-1})\leq \frac{1}{a-a'}$. A similar conclusion holds for $\sigma_c$.
\end{lem}
\begin{proof}
    We provide the proof for $\sigma_r$ only. It is well-known that diagonally dominant matrices are invertible. For such a matrix $A$, we have
    \[
    A^{-1} = A_1A_2^{-1},
    \]
    where
    \[
    A_1 = \mathrm{diag}(a_{ii}^{-1}),\ A_2 = (a_{ij}/a_{ii})_{i,j = 1}^n.
    \]
    The entries of $A_1$ are bounded by $a^{-1}$, thus $\sigma_r(A_1)\leq a^{-1}$. Moreover, note that
    \[
    A_2^{-1} = \sum_{k=0}^\infty (I - A_2)^k,
    \]
    while the lemma assumption implies that $\sigma_r(I - A_2)\leq a'/a$. By Lemma \ref{lem:4:2:2:2}, $\sigma_r(I - A_2)^k\leq (a'/a)^k$ for any $k\in\mathbb{N}$. Therefore,
    \[
    \sigma_r(A_2^{-1})\leq \sum_{k=0}^\infty(a'/a)^k = 1/(1-a'/a).
    \]
    Using Lemma \ref{lem:4:2:2:2} again, we conclude that
    \[
    \sigma_r(A^{-1})\leq \sigma_r(A_1)\sigma_r(A_2^{-1})\leq\frac{1}{a-a'}.
    \]
\end{proof}
\begin{proof}[Proof of Theorem \ref{thm:4:2:2}]
    Applying the $SL(n,\mathbb{R})$ action, we can assume that $Y = I$, $X$ is diagonal, and $\mathcal{I} = \{k+1,\dots,n\}$, where $1\leq k<n$. Then $g_Z:= g = (g_{ij}) \in SO(n)$. Lemma \ref{lem:4:2:2:1} implies a decomposition $g = g_+g_-^{-1}$. Since $g\in SO(n)$, i.e., $g.I = I$, we have $g_-.I = g_+.I$. Denote the diagonal matrices $X = X_0$ and $g.Z = Z_0$, then $(g_+^{-1})^{\mathrm{T}}\in GL^+(n,\mathbb{R})$ takes $X$ to $(g_+^{-1})^{\mathrm{T}}.X_0$, and takes $Z = ((g_-^{-1})^{\mathrm{T}}g_+^{\mathrm{T}}).Z_0$ to $(g_-^{-1})^{\mathrm{T}}.Z_0$.

    The theorem reduces to
    \[
    Bis((g_+^{-1})^{\mathrm{T}}.X_0,(g_+^{-1})^{\mathrm{T}}.I)\cap Bis((g_-^{-1})^{\mathrm{T}}.Z_0,(g_-^{-1})^{\mathrm{T}}.I) = \varnothing,
    \]
    or equivalently,
    \begin{equation}\label{disj}
    (g_+.(X_0^{-1} - I))^\perp\cap (g_-.(Z_0^{-1} - I))^\perp = \varnothing,\tag{*}
    \end{equation}
    under the assumption \eqref{equ:4:3}.

    Let $X_0 = \mathrm{diag}(x_i)$ and $Z_0 = \mathrm{diag}(z_i)$. Then,
    \[
    X_0^{-1} - I = \mathrm{diag}(x_i^{-1}-1),\quad Z_0^{-1} - I = \mathrm{diag}(z_i^{-1} - 1).
    \]
    Since $s(I,X),s(I,Z)\geq L$, Lemma \ref{lem:4:2:2:0} implies that for any $i\leq k$ and $j>k$,
    \[
    \frac{|x_j^{-1}-1|}{|x_i^{-1}-1|}\geq t\cdot \left(\frac{L-1}{n-1}\right)^t,\ \frac{|z_i^{-1}-1|}{|z_j^{-1}-1|}\geq t\cdot \left(\frac{L-1}{n-1}\right)^t.
    \]
    Thus, there exist positive constants $c_x$ and $c_z$ such that for any $i\leq k$ and $j>k$,
    \begin{equation}\label{equ:4:1}
    \begin{split}
        c_x(x_j^{-1} - 1)\geq t\cdot \left(\frac{L-1}{n-1}\right)^t, \quad -1\leq c_x(x_i^{-1} - 1)<0.\\
        c_z(z_i^{-1}-1)\geq t\cdot \left(\frac{L-1}{n-1}\right)^t, \quad -1\leq c_z(z_j^{-1}-1)<0.
    \end{split}
    \end{equation}
    
    If we let
    \[
    h = (h_{ij}) = \begin{pmatrix}
        (g_1^{-1})^{\mathrm{T}} & -(g_1^{-1})^{\mathrm{T}}g_3^{\mathrm{T}} \\ -g_4^{-1}g_3 & g_4^{-1}
    \end{pmatrix},
    \]
    then $h$ is decomposed as $h = h_a^{-1}h_b$, where
    \[
    h_a = \begin{pmatrix}
        g_1^{\mathrm{T}} & O \\ O & g_4
    \end{pmatrix},\ h_b = \begin{pmatrix}
        I & -g_3^{\mathrm{T}}\\ -g_3 & I
    \end{pmatrix}.
    \]
    The assumption of the theorem implies that the angle between $\mathbf{e}_i$ and the $i$-th column vector of $g$ is at most $\theta$, i.e., the diagonal elements of $h_a$ are no less than $\cos \theta$. For $i\leq k$,
    \[
    \sum_{j\neq i, j\leq k}|g_{ij}|\leq \sqrt{(k-1)\sum_{j\neq i,j\leq k}g_{ij}^2}\leq \sqrt{(k-1)}\sin\theta\leq \sqrt{(n-2)}\sin\theta.
    \]
    Similarly, for $i>k$, $\sum_{j\neq i, j> k}|g_{ij}|\leq \sqrt{(n-2)}\sin\theta$. Hence, by Lemma \ref{lem:4:2:2:3}:
    \[
    \sigma_r(h_a^{-1}), \sigma_c(h_a^{-1})\leq \frac{1}{\cos\theta - \sqrt{n-2}\sin \theta}.
    \]
    Moreover, the assumption of the theorem implies that $\sigma_r(h_b),\sigma_c(h_b)\leq 1+\sqrt{n-2}\sin\theta$. Applying Lemma \ref{lem:4:2:2:2}, we deduce that
    \[
    \sigma_r(h), \sigma_c(h)\leq \frac{1+\sqrt{n-2}\sin\theta}{\cos\theta - \sqrt{n-2}\sin \theta}.
    \]

    We establish the condition \eqref{disj} by proving the positive definiteness of the linear combination $c_x\cdot g_+.(X_0^{-1} - I) + c_z\cdot g_-.(Z_0^{-1} - I)$. Let $c_x\cdot g_+.(X_0^{-1} - I) = (\xi_{ij})$ and $c_z\cdot g_-.(Z_0^{-1} - I) = (\zeta_{ij})$. For $i\leq k$, we have the following inequalities:
    \begin{align*}
        & \xi_{ii} = \sum_{l\leq k}h_{li}^2(x_l^{-1}-1)\geq -\sum_{l\leq k}h_{li}^2,\\
        & \sum_{j\neq i}|\xi_{ij}|\leq \sum_{j\neq i,\ l\leq k}|h_{li}||h_{lj}||x_l^{-1} - 1|\leq \sum_{j\neq i,\ l\leq k}|h_{li}||h_{lj}|,\\
        & \zeta_{ii} = (z_i^{-1} - 1) + \sum_{l>k}h_{li}^2(z_l^{-1} - 1)\geq t((L-1)/(n-1))^t -\sum_{l>k}h_{li}^2,\\
        & \sum_{j\neq i}|\zeta_{ij}|\leq \sum_{j\neq i,\ l>k}|h_{li}||h_{lj}||z_l^{-1} - 1|\leq \sum_{j\neq i,\ l>k}|h_{li}||h_{lj}|.
    \end{align*}
    Hence,
    \begin{align*}
        & \xi_{ii}+\zeta_{ii} \geq t((L-1)/(n-1))^t - \sum_{l=1}^n h_{li}^2 \geq \left(\frac{1+\sqrt{n-2}\sin\theta}{\cos\theta - \sqrt{n-2}\sin\theta}\right)^2 - \sum_{l=1}^n h_{li}^2\\
        & = \sigma_r(h)\sigma_c(h) - \sum_{l=1}^n h_{li}^2\geq \sum_{l=1}^n\sigma_r(h)|h_{li}| - \sum_{l=1}^n h_{li}^2 \geq \sum_{l,j} |h_{lj}||h_{li}| - \sum_{l=1}^n h_{li}^2\\
        & = \sum_{j\neq i, 1\leq l\leq n} |h_{li}||h_{lj}| \geq \sum_{j\neq i}|\xi_{ij}+\zeta_{ij}|.
    \end{align*}
    For $i>k$, the inequality $\xi_{ii}+\zeta_{ii}\geq \sum_{j\neq i}|\xi_{ij}+\zeta_{ij}|$ holds analogously. This implies that $c_x\cdot g_+.(X_0^{-1} - I) + c_z\cdot g_-.(Z_0^{-1} - I)$ is diagonally dominant and hence positive definite. According to Lemma \ref{lem:1}, $Bis(X,Y)$ and $Bis(Y,Z)$ are disjoint.
\end{proof}
\begin{rmk}
    When $n=2$, maximal totally geodesic flat submanifolds in $\mathcal{P}(2)$ are geodesic lines. The two chambers in Definition \ref{defn:4:4} correspond geodesic rays, and the parameter $t$ can be set to $1$. Moreover, for any $X,Y,Z\in\mathcal{P}(2)$ such that $g_X.X$ and $g_Z.Z$ (as defined in Theorem \ref{thm:4:2:2}) lie in opposite chambers, the following relations hold:
    \[
    \theta = \frac{\pi-\alpha}{2},\ L = 2\cosh l-1 = \cosh^2(l/2) + 3\sinh^2(l/2),
    \]
    where $\alpha = \angle XYZ$ and $l = \min (d(X,Y),d(Y,Z))$. Therefore, the condition in Theorem \ref{thm:4:2:2} reduces to:
    \[
    \sin(\alpha/2)\sqrt{\cosh^2(l/2) + 3\sinh^2(l/2)}\geq 1,
    \]
    or a stronger condition, $\sin(\alpha/2)\cosh(l/2)\geq 1$, which agrees with the corresponding conclusion for hyperbolic spaces \cite{kapovich2023geometric,kapovich2018geometric}.
\end{rmk}
\vspace{12pt}
\section{Algorithm for computing the posets of partial Dirichlet-Selberg domains}\label{sec:s7}
This section aims to generalize step (2) of Poincar\'e's algorithm for $SO^+(n,1)$, as described in Subsection \ref{subsec:2:2}, to the Lie group $SL(n,\mathbb{R})$. Namely, we seek an algorithm that computes the poset structure of a finitely-sided $\mathcal{P}(n)$-polyhedron $P = \bigcap_{i=1}^k H_i$, given the equations for the half spaces $H_i$, $i=1,\dots,k$ in $\mathcal{P}(n)$. 

The original sub-algorithm, proposed by Epstein and Petronio\cite{epstein1994exposition}, utilizes the fact that a facet of an $\mathbf{H}^n$-polyhedron is an $\mathbf{H}^{n-1}$-polyhedron. This has no analogy for facets of an $SL(n,\mathbb{R})/SO(n)$-polyhedron. Instead, we describe an algorithm determining if the intersection of a collection of hyperplanes in $\mathcal{P}(n)$ is non-empty, given in Subsection \ref{sec:s7:2}.

Following the work of Epstein and Petronio, we adopt the \textbf{Blum-Shub-Smale (BSS) computational model}\cite{blum1989theory} for our algorithm. In the BSS model, arbitrarily many real numbers can be stored, and rational functions over real numbers can be computed in a single step.

We assert that a BSS algorithm exists that computes the poset structure of finitely-sided convex polyhedra in $\mathcal{P}(n)$:
\begin{thm}\label{thm:s5}
    There is a BSS algorithm with an input consisting of a point $X\in \mathcal{P}(n)$ and a finite list of elements $A'_1,\dots,A'_{k'}\in Sym_n(\mathbb{R})$ which yields an output describing the poset structure for the face set $\mathcal{F}(P)$, where
    \[
    P = \bigcap_{i=1}^{k'}\{Y|\mathrm{tr}(A_i'Y)\geq 0\}.
    \]
    Specifically, the output consists of the data $\mathcal{A}$, $L^{face}$, $L^{pos}$, and $L^{samp}$, where
    \begin{enumerate}
        \item $\mathcal{A} = \{A_1,\dots, A_k\}$ is a subset of the input set $\{A'_1,\dots,A'_{k'}\}$.
        \item $L^{face}$ is a two-dimensional array comprised of numbers from the set $\{1,\dots, k'\}$, describing the set $\{F_1,\dots, F_m\}$ of faces of $P$. Specifically, $L^{face}$ is a 2D array $\{L^{face}_1,\dots, L^{face}_m\}$, where $m = |\mathcal{F}(P)|$, and such that
        \[
        span(F_j) = \bigcap_{i\in L^{face}_j} A_i^\perp,\ j=1,\dots,m.
        \]
        \item $L^{pos}$ is a two-dimensional array comprised of numbers from the set $\{1,\dots, m\}$, describing the inclusion relation among the faces of $P$, namely
        \[
        L^{pos}_j = \{1\leq l\leq m|F_l\subsetneq F_j\},\ j=1,\dots,m.
        \]
        \item $L^{samp}$ is an array of elements in $\mathcal{P}(n)$ serving to describe sample points associated with the faces of $P$:
        \[
        L^{samp}_j\in F_j,\ j=1,\dots, m.
        \]
    \end{enumerate}
\end{thm}
We will describe the algorithm claimed by Theorem \ref{thm:s5} in the subsequent sections.

\subsection{Sample points for planes of \texorpdfstring{$\mathcal{P}(n)$}{Lg}}\label{sec:s7:2}
In this section, we describe an essential step of the algorithm claimed in Theorem \ref{thm:s5}. This sub-algorithm is designed to check the emptiness of the intersection of the given hyperplanes and to yield a sample point in this intersection. 
\begin{lem}\label{lem:s5:1}
    There is a numerical algorithm with an input consisting of matrices $A_1,\dots, A_l\in Sym_n(\mathbb{R})$, yielding the following outcome:
    \begin{itemize}
        \item If the intersection $\bigcap_{i=1}^l A_i^\perp = \varnothing$, the algorithm outputs \textbf{false}.
        \item If $\bigcap_{i=1}^l A_i^\perp$ is non-empty, the algorithm outputs \textbf{true} and provides a sample point in $\bigcap_{i=1}^l A_i^\perp$.
    \end{itemize}
\end{lem}
We utilize the following lemma to prove Lemma \ref{lem:s5:1}:
\begin{lem}\label{lem:algor}
    Suppose that $B_1,\dots,B_l\in Sym_n(\mathbb{R})$ are linearly independent matrices, and that $span(B_1,\dots, B_l)$ contains an invertible element. Then $span(B_1,\dots, B_l)$ contains a positive definite element if and only if
\begin{equation}\label{equ:crit}
    \sum x^i_0 B_i\succ0
\end{equation}
holds for a real and isolated critical point $(x^1_0,\dots,x^l_0)$ of the homogeneous polynomial $P(x^1,\dots,x^l) = \det(\sum x^i B_i)$ restricted to the unit sphere $\mathbf{S}^{l-1}$.
\end{lem}
\begin{proof}
    The ``if'' part is self-evident. To prove the ``only if'' part, we assume that $X' = \sum x'^i B_i$ is a positive definite element in $span(B_1,\dots, B_l)$, where $(x'^i): = \mathbf{x'} \neq \mathbf{0}\in\mathbb{R}^l$.
    
    We first show the existence of a critical point of $P|_{\mathbf{S}^{l-1}}$ satisfying \eqref{equ:crit}. Without loss of generality, we assume that $\mathbf{x'}$ is a unit vector. Let $\Sigma$ be the connected component of $\mathbf{S}^{l-1}\backslash \{P(x^1,\dots,x^l)=0\}$ containing $\mathbf{x'}$. Since $X' = \sum x'^iB_i$ is positive definite, $\sum x^i B_i$ is also positive definite for all $\mathbf{x} = (x^i)\in \Sigma$. Furthermore, as $P|_{\partial \Sigma} = 0$, $\Sigma$ contains a point $\mathbf{x_0} = (x^i_0)$ which is a local maximum point of $P|_{\mathbf{S}^{l-1}}$. Consequently, $\mathbf{x_0}$ is a critical point of $P|_{\mathbf{S}^{l-1}}$ with $\sum x^i_0 B_i$ being positive-definite.

    We proceed to show that the critical point $\mathbf{x_0} = (x_0^1,\dots,x_0^l)$ is isolated. Suppose to the contrary that the critical set of $P|_{\mathbf{S}^{l-1}}$ contains a subset $S\subset\mathbf{S}^{l-1}$, which is an algebraic variety with $\dim(S)\geq 1$ and $\mathbf{x_0}\in S$. By replacing $\mathbf{x_0}$ with another point in $S$, we assume that $\mathbf{x_0}$ is a regular point in $S$. Consequently, $\mathbf{x_0}$ is contained in a smooth curve $\mathbf{x}: (-\epsilon,\epsilon)\to \mathbf{S}^{l-1}$, where $\mathbf{x}(t)$ is a critical point of $P|_{\mathbf{S}^{l-1}}$ for each $t\in (-\epsilon,\epsilon)$. 
    
    The smooth function $\mathbf{x}$ admits an expansion:
    \[\mathbf{x}(t) = \mathbf{x_0} + t \mathbf{y_0} + t^2 \mathbf{z_0}+O(t^3),
    \]
    where $\mathbf{y_0}\neq \mathbf{0}$. Since the curve $\mathbf{x}$ lies in the unit sphere,
    \[
    \sum_{i=1}^l x_0^iy_0^i = 0,\quad \sum_{i=1}^l x_0^iz_0^i + \frac{1}{2}\sum_{i=1}^l (y_0^i)^2 = 0,
    \]
    implying that both $\mathbf{y_0}$ and $\mathbf{z_0}+\frac{||\mathbf{y_0}||^2}{2}\mathbf{x_0}$ lie in $T_\mathbf{x_0} \mathbf{S}^{l-1}$. Since $\mathbf{x_0}$ is a critical point of $P|_{\mathbf{S}^{l-1}}$, the derivatives of $P$ at $\mathbf{x_0}$ along both $\mathbf{y_0}$ and $\mathbf{z_0}+\frac{||\mathbf{y_0}||^2}{2}\mathbf{x_0}$ vanish. By letting $X(t) = \sum x^i(t)B_i$, $X_0 = \sum x_0^iB_i$, $Y_0 = \sum y_0^i B_i$, and $Z_0 = \sum z_0^i B_i$, the vanishing of these directional derivatives is formulated as
    \[
    \mathrm{tr}(X_0^{-1}Y_0) = 0,\quad \mathrm{tr}(X_0^{-1}Z_0) = -\frac{||\mathbf{y_0}||^2}{2}\mathrm{tr}(X_0^{-1}X_0) = -\frac{n}{2}||\mathbf{y_0}||^2.
    \]

    On the other hand, since the points $\mathbf{x}(t) = (x^1(t),\dots,x^l(t))$ are critical points for $-\epsilon<t<\epsilon$, Sard's Theorem implies that $\det(X(t)) = P(\mathbf{x}(t)) \equiv P(\mathbf{x_0}) = \det(X_0)$, leading to:
    \begin{equation*}
    \sum_{i=1}^n\lambda_i = 0,\quad \sum_{1\leq i<j\leq n}\lambda_i\lambda_j + \sum_{i=1}^n \mu_i = 0,
    \end{equation*}
    where $\lambda_i$ and $\mu_i$, $i=1,\dots, n$ are eigenvalues of $X_0^{-1}Y_0$ and $X_0^{-1}Z_0$, respectively. Since $X_0,Y_0$ and $Z_0$ are real symmetric matrices and $X_0$ is positive definite, $\lambda_i$ and $\mu_i$ are real numbers. Combining the equations above, we obtain that
    \begin{equation*}
         0\leq \sum_{i=1}^n \lambda_i^2 = (\sum_{i=1}^n \lambda_i)^2-2(\sum_{i<j}\lambda_i\lambda_j) = 2\sum_{i=1}^n\mu_i = 2\mathrm{tr}(X_0^{-1}Z_0) = -n ||\mathbf{y_0}||^2<0,
    \end{equation*}
    which is a contradiction.
\end{proof}
\begin{proof}[Proof of Lemma \ref{lem:s5:1}]
    Suppose that $A_1,\dots,A_l\in Sym_n(\mathbb{R})$ is the input, and $\{B_1,\dots,B_{l'}\}$ is a basis for the orthogonal complement of $\mathrm{span}(A_1,\dots,A_l)$ in $Sym_n(\mathbb{R})$. Then,
    \[
    \bigcap_{i=1}^{l} A_i^\perp = span(B_1,\dots,B_{l'})\cap \mathcal{P}(n).
    \]

    If $P(x^1,\dots,x^{l'}) = \det(\sum x^iB_i)\equiv 0$, neither matrix in $span(B_1,\dots,B_{l'})$ is strictly definite, implying that $\bigcap_{i=1}^l A_i^\perp$ is empty.
    
    Otherwise, $P(x^1,\dots,x^{l'})$ is a homogeneous polynomial of degree $n$ in variables $x^1,\dots, x^{l'}$. Since $\mathbf{S}^{l'-1}$ is compact, the restriction of the polynomial $P(x_1,\dots,x_{l'})|_{\mathbf{S}^{l'-1}}$ has finitely many isolated critical points, and there exist well-known numerical BSS algorithms to find them, e.g., \cite{burgisser2011problem}. Let $\mathbf{x}_1,\dots,\mathbf{x}_m$ denote these isolated critical points, where $\mathbf{x}_j = (x_j^1,\dots,x_j^{l'})$. If $\sum x_j^i B_i$ is positive definite for a certain $j\in\{1,\dots,m\}$, Lemma \ref{lem:algor} implies that $\bigcap_{i=1}^l A_i^\perp$ is non-empty with a sample point $\sum x_j^i B_i$. Conversely, if $\sum x_j^i B_i$ is not strictly definite for all $j$, Lemma \ref{lem:algor} implies that $\bigcap_{i=1}^l A_i^\perp$ is empty. The algorithm we described terminates within a finite number of steps.
\end{proof}
\subsection{Situations of face and half-space pairs}
Similarly to the case of hyperbolic spaces (cf. \cite{epstein1994exposition}), the algorithm claimed in Theorem \ref{thm:s5} will involve a step determining the ``situation'' of pairs $(F,H)$. Here, $F$ is a face of a given convex polyhedron $P$, and $H$ is a given half-space in $\mathcal{P}(n)$. We begin by defining the relative positions of such pairs:
\begin{lem}\label{lem:5:1}
Let $P$ be a polyhedron in $\mathcal{P}(n)$, and $H$ be a half space in $\mathcal{P}(n)$. For any face $F\in \mathcal{F}(P)$, one of the following relative positions holds for the pair $(F,H)$:
    \begin{enumerate}
        \item The face $F$ lies on the boundary of $H$, i.e., $F\subset \partial H$.
        \item The face $F$ lies in the interior of $H$, i.e., $F\subset \mathrm{int}(H)$.
        \item The face $F$ lies in $H$ and meets its boundary, i.e., $F\subset H$, $F\cap \partial H \neq \varnothing$, and $F\cap \mathrm{int}(H)\neq \varnothing$.
        \item The face $F$ lies in $H^{\mathsf{c}}$, i.e., $F\cap H = \varnothing$.
        \item The face $F$ lies in $(\mathrm{int}(H))^{\mathsf{c}}$ and meets its boundary, i.e., $F\cap \mathrm{int}(H) = \varnothing$, $F\cap \partial H \neq \varnothing$, and $F\cap H^{\mathsf{c}}\neq \varnothing$.
        \item The face  $F$ crosses $\partial H$, i.e., $F\cap \mathrm{int}(H)\neq \varnothing$ and $F\cap H^{\mathsf{c}}\neq \varnothing$.
    \end{enumerate}
    Here $H^{\mathsf{c}}$ refers to the complement of $H$.
\end{lem}
The lemma is self-evident.
\begin{defn}
    Let $H$ be a half-space and $P$ be a convex polyhedron in $\mathcal{P}(n)$. For $i=1,\dots, 6$, we denote $\mathcal{F}^{(i)}_H(P)$ as the set of faces $F\in \mathcal{F}(P)$ such that $(F,H)$ belongs to relative position ($i$).
\end{defn}
As described in the following lemma, the relative position of a pair $(F,H)$ is determined by the relative positions of pairs $(F',H)$, where $F'$ are all proper faces of $F$.
\begin{lem}\label{lem:5:2}
Consider a half-space $H$ and a polyhedron $P$ in $\mathcal{P}(n)$. Let $F\in \mathcal{F}(P)$, such that $\partial F\neq \varnothing$, $F\neq H$, and $F \neq \overline{H^{\mathsf{c}}}$. Then, the relative position of $(F,H)$ is determined as follows:
    \begin{enumerate}
        \item If $F$ has a proper face $F'\in \mathcal{F}^{(6)}_H(P)$, then $F\in \mathcal{F}^{(6)}_H(P)$.
        \item If $F$ has a proper face $F'_1\in \mathcal{F}^{(2)}_H(P)\cup \mathcal{F}^{(3)}_H(P)$ and another proper face $F'_2\in \mathcal{F}^{(4)}_H(P)\cup \mathcal{F}^{(5)}_H(P)$, then $F\in \mathcal{F}^{(6)}_H(P)$.
        \item If the previous two cases do not apply and $F$ has a proper face $F'\in \mathcal{F}^{(1)}_H(P)\cup \mathcal{F}^{(3)}_H(P)\cup \mathcal{F}^{(5)}_H(P)$, then $F\in \mathcal{F}^{(1)}_H(P)\cup\mathcal{F}^{(3)}_H(P)\cup \mathcal{F}^{(5)}_H(P)$.
        \item If $F'\in \mathcal{F}^{(2)}_H(P)$ for all proper faces $F'$ of $F$, then $F\in \mathcal{F}^{(6)}_H(P)$ if $F\cap \partial H \neq \varnothing$, and $F\in \mathcal{F}^{(2)}_H(P)$ if $F\cap \partial H = \varnothing$.
        \item If $F'\in \mathcal{F}^{(4)}_H(P)$ for all proper faces $F'$ of $F$, then $F\in \mathcal{F}^{(6)}_H(P)$ if $F\cap \partial H \neq \varnothing$, and $F\in \mathcal{F}^{(4)}_H(P)$ if $F\cap \partial H = \varnothing$.
    \end{enumerate}
\end{lem}
\begin{proof}
\textbf{Case (1).} If a proper face $F'\in \mathcal{F}^{(6)}_H(P)$, then $F\cap \mathrm{int}(H)$ contains $F'\cap \mathrm{int}(H)$, which is non-empty, and $F\cap H^{\mathsf{c}}$ contains $F'\cap H^{\mathsf{c}}$, also non-empty. Therefore, $F\in \mathcal{F}^{(6)}_H(P)$.

    \textbf{Case (2).} If a proper face $F'_1\in \mathcal{F}^{(2)}_H(P)\cup \mathcal{F}^{(3)}_H(P)$, then $F\cap \mathrm{int}(H)$ contains $F'_1\cap \mathrm{int}(H)$, which is non-empty. Additionally, if another proper face $F'_2\in \mathcal{F}^{(4)}_H(P)\cup \mathcal{F}^{(5)}_H(P)$, then $F\cap H^{\mathsf{c}}$ contains $F'_2\cap H^{\mathsf{c}}$, which is also non-empty. Therefore, $F\in \mathcal{F}^{(6)}_H(P)$.

    \textbf{Case (3).} Suppose that neither case (1) nor case (2) occurs. If a proper face $F'\in \mathcal{F}^{(1)}_H(P)\cup \mathcal{F}^{(3)}_H(P)\cup \mathcal{F}^{(5)}_H(P)$, we claim that $F\notin \mathcal{F}^{(6)}_H(P)$.

    If $F'\in \mathcal{F}^{(3)}_H(P)$, then $\partial H$ meets $F'$ at $\partial F'$, thus $\partial H\cap F'$ is a proper face of $F'$ in $\mathcal{F}^{(1)}_H(P)$. Similarly, if $F'\in \mathcal{F}^{(5)}_H(P)$, then $\partial H\cap F'$ is also in $\mathcal{F}^{(1)}_H(P)$. This implies that $\partial F\cap \partial H\neq \varnothing$ holds. Since $F\neq H$ and $F\neq \overline{H^{\mathsf{c}}}$, neither $F\supset \mathrm{int}(H)$ nor $F\supset H^{\mathsf{c}}$ is satisfied. Therefore, if $F\in \mathcal{F}^{(6)}_H(P)$, then $\partial F\cap \mathrm{int}(H) \neq\varnothing$ and $\partial F\cap H^{\mathsf{c}} \neq\varnothing$. However, this condition has been excluded by the two cases above.

    Since $F\cap \partial H$ contains $F'\cap \partial H$, which is non-empty, $F\in \mathcal{F}^{(1)}_H(P)\cup\mathcal{F}^{(3)}_H(P)\cup \mathcal{F}^{(5)}_H(P)$.

    \textbf{Case (4).} If $F'\in \mathcal{F}^{(2)}_H(P)$ for all proper faces of $F$, we have that $\partial F\cap \partial H = \bigcup_{F'_i\in \mathcal{F}(F)} (F_i'\cap \partial H) = \varnothing$. If $F\in \mathcal{F}^{(3)}_H(P)$, then the conditions $F\subset H$ and $F\cap \partial H\neq\varnothing$ imply that $\partial H$ meets $F$ on the boundary of $F$, which is a contradiction. 
    
    For any proper face $F'\subset F$, $F\cap \mathrm{int}(H)$ contains $F'\cap \mathrm{int}(H)$, which is non-empty. Knowing that $F\notin \mathcal{F}^{(3)}_H(P)$, we have either $F\in \mathcal{F}^{(2)}_H(P)$ or $F\in \mathcal{F}^{(6)}_H(P)$, depending on whether $F\cap \partial H$ is empty. 
    
    Similarly, the claim holds if we replace $\mathcal{F}^{(2)}_H(P)$ with $\mathcal{F}^{(4)}_H(P)$.
\end{proof}

\subsection{Output data for the new polyhedron}\label{sec:s7:4}
Suppose that we have determined the relative position of $(F,H)$ for all $F\in\mathcal{F}(P)$ for a given polyhedron $P$ and a half-space $H$ in $\mathcal{P}(n)$. We will then compute the output data required by Theorem \ref{thm:s5} for the intersection $P\cap H$, namely the lists $L^{face}$, $L^{pos}$ and $L^{samp}$ describing the face set $\mathcal{F}(P\cap H)$, the poset structure on $\mathcal{F}(P\cap H)$, and sample points of the faces in $\mathcal{F}(P\cap H)$, respectively. The set of faces $\mathcal{F}(P\cap H)$ is characterized by the lemma below:
\begin{lem}\label{lem:5:3}
Let $P$ be a polyhedron and $H$ be a half-space in $\mathcal{P}(n)$. Then $\mathcal{F}(P\cap H)$ consists of:
    \begin{itemize}
        \item Faces $F\in \mathcal{F}^{(1)}_H(P)\cup \mathcal{F}^{(2)}_H(P)\cup \mathcal{F}^{(3)}_H(P)$, and
        \item Intersections $F\cap H$ and $F\cap \partial H$, where $F\in \mathcal{F}^{(6)}_H(P)$.
    \end{itemize}
\end{lem}

\begin{proof}
The polyhedron $P$ in $\mathcal{P}(n)$ is described as the intersection $\bigcap_{i=1}^{k-1} H_i$, where $H_i = \{\mathrm{tr}(Y\cdot A_i)\geq 0\}$ for $i=1,\dots, k-1$, each $H_i$ is a half-space in $\mathcal{P}(n)$. Let $H = H_k = \{\mathrm{tr}(Y\cdot A_k)\geq 0\}$. A face $F'$ of $P\cap H = \bigcap_{i=1}^k H_i$ can be expressed as
    \[
    F' = \{Y\in\mathcal{P}(n)|\mathrm{tr}(Y\cdot A_i) = 0,\ \forall i\in\mathcal{I},\ \mathrm{tr}(Y\cdot A_j) \geq 0,\forall j\in\mathcal{J}\},
    \]
    where $\mathcal{I}$ and $\mathcal{J}$ are disjoint subsets of $\{1,\dots,k\}$. Here, we assume that the matrices $A_i$ for $i\in\mathcal{I}$ are linearly independent, and the set of inequalities $\mathrm{tr}(Y\cdot A_j) \geq 0$ is irredundant. We consider three cases: when $k\notin \mathcal{I}\cup\mathcal{J}$, when $k\in \mathcal{I}$, and when $k\in\mathcal{J}$.

    If $k\notin \mathcal{I}\cup\mathcal{J}$, then $F'$ is a face of $P$. Since $F'\subset P\cap H \subset H$, $F'
    \in \mathcal{F}^{(1)}_H(P)\cup \mathcal{F}^{(2)}_H(P)\cup \mathcal{F}^{(3)}_H(P)$.

    If $k\in \mathcal{I}$, the set
    \[
    F = \{Y\in\mathcal{P}(n)|\mathrm{tr}(Y\cdot A_i) = 0,\forall i\in\mathcal{I}\backslash\{k\},\ \mathrm{tr}(Y\cdot A_j) \geq 0,\forall j\in\mathcal{J}\}
    \]
    is a face of $P$, and $F' = F\cap \partial H$. Since $F\cap \partial H$ is non-empty, $F\notin \mathcal{F}^{(2)}_H(P)\cup \mathcal{F}^{(4)}_H(P)$. Since the equations defining $F'$ are irredundant, $F' = F\cap \partial H$ is not a face of $P$, implying that $F\in \mathcal{F}^{(6)}_H(P)$.

    If $k\in \mathcal{J}$, the set
    \[
    F = \{Y\in\mathcal{P}(n)|\mathrm{tr}(Y\cdot A_i) = 0,\forall i\in\mathcal{I},\ \mathrm{tr}(Y\cdot A_j) \geq 0,\forall j\in\mathcal{J}\backslash\{k\}\}
    \]
    is a face of $P$, and $F' = F\cap H$. Since the inequalities are irredundant, $F'\subsetneq F$, thus $F\cap H^{\mathsf{c}}$ is non-empty. Moreover, $F\cap H\neq F\cap \partial H$. These together imply that $F\in \mathcal{F}^{(6)}_H(P)$.

    On the other hand, if $F\in \mathcal{F}^{(1)}_H(P)\cup \mathcal{F}^{(2)}_H(P)\cup \mathcal{F}^{(3)}_H(P)$, then $F\subset P\cap H$, implying that $F\in \mathcal{F}(P\cap H)$. If $F\in \mathcal{F}^{(6)}_H(P)$, both $F\cap H$ and $F\cap \partial H$ are subsets of $P\cap H$, implying that $F\cap H$ and $F\cap \partial H\in \mathcal{F}(P\cap H)$.
\end{proof}

The poset structure of $\mathcal{F}(P\cap H)$ is described in the following lemma:
\begin{lem}\label{lem:5:4}
Let $P$ be a polyhedron and $H$ be a half-space in $\mathcal{P}(n)$. For a given face of $P\cap H$, we categorize its proper faces:
    \begin{itemize}
        \item For $F\in \mathcal{F}^{(1)}_H(P)\cup \mathcal{F}^{(2)}_H(P)\cup \mathcal{F}^{(3)}_H(P)$, the set of its proper faces in $\mathcal{F}(P\cap H)$ remains unchanged compared to those in $\mathcal{F}(P)$.
        \item For $F\in \mathcal{F}^{(6)}_H(P)$, the proper faces of $F\cap H$ include:
        \begin{itemize}
            \item Proper faces $F'$ of $F$ where $F'\in \mathcal{F}^{(1)}_H(P)\cup \mathcal{F}^{(2)}_H(P)\cup \mathcal{F}^{(3)}_H(P)$,
            \item Intersections $F'\cap H$ and $F'\cap\partial H$, where $F'\in  \mathcal{F}^{(6)}_H(P)$ is a proper face of $F$, and
            \item The intersection $F\cap \partial H$.
        \end{itemize}
        \item For $F\in \mathcal{F}^{(6)}_H(P)$, the proper faces of $F\cap \partial H$ include:
        \begin{itemize}
            \item Proper faces $F'$ of $F$ where $F'\in \mathcal{F}^{(1)}_H(P)$, and
            \item Intersections $F'\cap\partial H$, where $F'\in  \mathcal{F}^{(6)}_H(P)$ is a proper face of $F$.
        \end{itemize}
    \end{itemize}
\end{lem}

\begin{proof}
A face of $F\cap P$ falls into one of the three cases listed in Lemma \ref{lem:5:3}. On the one hand, it is evident that the faces of $P\cap H$ enumerated in Lemma \ref{lem:5:4} are proper faces of $F$, $F\cap H$, or $F\cap \partial H$, respectively. On the other hand, we aim to show that these are indeed all possible proper faces of $F$, $F\cap H$, or $F\cap \partial H$, respectively.

    \textbf{Proper faces of $F\in \mathcal{F}^{(1)}_H(P)\cup \mathcal{F}^{(2)}_H(P)\cup \mathcal{F}^{(3)}_H(P)$.} If $F'$ is a proper face of $F$ in $\mathcal{F}(P)$, it follows that $F'$ is also in $\mathcal{F}(P\cap H)$.

    \textbf{Proper faces of $F\cap H$, $F\in \mathcal{F}^{(6)}_H(P)$.} 
    Proper faces $F''$ of $F\cap H$ are faces of $P\cap H$. We consider three cases according to Lemma \ref{lem:5:3}:

    (1) If $F''$ is a face of $P$, then $F''\in \mathcal{F}^{(1)}_H(P)\cup \mathcal{F}^{(2)}_H(P)\cup \mathcal{F}^{(3)}_H(P)$. Since $F''\subsetneq F\cap H\subsetneq F$, $F''$ is a proper face of $F$ in $\mathcal{F}(P)$.

    (2) Suppose that $F'' = F'\cap H$, where $F'\in \mathcal{F}^{(6)}_H(P)$. Since $(F'\cap F)\cap H\subset F'\cap H = (F'\cap H)\cap (F\cap H) = (F'\cap F)\cap H$, replacing $F'$ with $F'\cap F$ maintains $F'\cap H$. Hence, we can assume that $F'$ is a face of $F$. Furthermore, $F'\cap H\subsetneq F\cap H$, implying that $F'$ is a proper face of $F$.
    
    (3) Suppose that $F'' = F'\cap \partial H$, where $F'\in \mathcal{F}^{(6)}_H(P)$. Analogously, $(F'\cap F)\cap \partial H\subset F'\cap \partial H = (F'\cap \partial H)\cap (F\cap H) = (F'\cap F)\cap \partial H$, thus replacing $F'$ with $F'\cap F$ maintains $F'\cap \partial H$. Hence, we can assume that $F'$ is a face (not necessarily proper) of $F$.

    \textbf{Proper faces of $F\cap \partial H$, $F\in \mathcal{F}^{(6)}_H(P)$.} Proper faces $F''$ of $F\cap \partial H$ are faces of $P\cap H$ with $F''\subset \partial H$. Thus, we consider two cases:

    (1) If $F''$ is a face of $P$, it follows that $F''\in \mathcal{F}^{(1)}_H(P)$. Since $F''\subsetneq F\cap \partial H\subsetneq F$, $F''$ is a proper face of $F$ in $\mathcal{F}(P)$.

    (2) Suppose that $F'' = F'\cap \partial H$, where $F'\in \mathcal{F}^{(6)}_H(P)$. Since $(F'\cap F)\cap \partial H\subset F'\cap \partial H = (F'\cap \partial H)\cap (F\cap \partial H) = (F'\cap F)\cap \partial H$, replacing $F'$ with $F'\cap F$ maintains $F'\cap \partial H$. Hence, we can assume that $F'$ is a face of $F$. Furthermore, $F'\cap \partial H\subsetneq F\cap \partial H$, implying that $F'$ is a proper face of $F$.
\end{proof}
Lastly, we describe a sub-algorithm that obtains sample points for faces of $P\cap H$:
\begin{lem}\label{lem:5:5}
There exists a BSS algorithm with inputs consisting of the lists $L^{face}$, $L^{pos}$ and $L^{samp}$ for a polyhedron $P$ in $\mathcal{P}(n)$, along with the equation for a half-space $H$ in $\mathcal{P}(n)$, and yields sample points for the faces of $P\cap H$. 
\end{lem}
\begin{proof}
By Lemma \ref{lem:5:3}, $\mathcal{F}(P\cap H)$ consists of faces $F_j\in\mathcal{F}^{(1)}_H(P)\cup\mathcal{F}^{(2)}_H(P)\cup\mathcal{F}^{(3)}_H(P)$ and intersections $F_j\cap H$ and $F_j\cap \partial H$ for $F_j\in\mathcal{F}^{(6)}_H(P)$. A sample point of the face $F_j\in\mathcal{F}^{(1)}_H(P)\cup\mathcal{F}^{(2)}_H(P)\cup\mathcal{F}^{(3)}_H(P)$ is given in the list $L^{samp}$. Thus, our task reduces to computing sample points of $F_j\cap H$ and $F_j\cap \partial H$, where $F_j\in\mathcal{F}^{(6)}_H(P)$.

    We begin by obtaining a sample point of $F' = F_j\cap \partial H$ by induction on the dimension of $F'$. If $F'$ is a minimal face, we compute its sample point using the algorithm described in Lemma \ref{lem:s5:1}. If $F'$ is not minimal, we assume by induction that we have obtained all sample points for the proper faces of $F'$. If $|\mathcal{F}(F')|\geq 2$, let $X_0'$ denote the barycenter of the sample points of the proper faces of $F'$. Note that $X_0'\in F'$ due to the convexity of $F'$, and $X_0'\notin \partial F'$ due to the disjointness of the interiors of its proper faces. Thus, the barycenter $X_0'$ lies in the interior of $F'$.

    If $\mathcal{F}(F') = \{F''\}$, then $F''$ is minimal, and $F'$ is a half-space in $span(F')$. Let $X_0''\in F''$ be the sample point of $F''$ and take an orthogonal basis $B_1,\dots, B_m$ of
    \[
    T_{X_0''}span(F') = \{B\in Sym_n(\mathbb{R})|B\in span_{Sym_n(\mathbb{R})}(F'),\ \mathrm{tr}((X_0'')^{-1}B) = 0\}.
    \]
    For at least one choice $i\in\{1,\dots, m\}$ and a sufficiently small $\epsilon>0$, either $X_0''+ \epsilon B_i\in \mathrm{int}(F')$ or $X_0''- \epsilon B_i\in \mathrm{int}(F')$. This point serves as a sample point of $F'$.

    Knowing a sample point $X_0'$ of $F' = F_j\cap\partial H$, we obtain a sample point of $F = F_j\cap H$ as follows. We take an orthogonal basis $C_1,\dots, C_{m'}$ of $T_{X_0'}span(F)$, which is given analogously to the previous case. For at least one choice $i\in\{1,\dots, m'\}$ and a sufficiently small $\epsilon>0$, either $X_0'+ \epsilon C_i\in \mathrm{int}(F)$ or $X_0'- \epsilon C_i\in \mathrm{int}(F)$. This point serves as a sample point of $F$.
\end{proof}
\subsection{Description of the algorithm}
We describe in this subsection the algorithm proposed in Theorem \ref{thm:s5}. The algorithm utilizes the lemmas established in Sections \ref{sec:s7:2} through \ref{sec:s7:4}.

\textbf{Algorithm for computing the poset structure of polyhedra in $\mathcal{P}(n)$.}

Consider a point $X\in\mathcal{P}(n)$ and a list $\mathcal{A}'$ of matrices $A_i'$, where $i=1,\dots,k'$. Define the half-spaces $H_i = \{\mathrm{tr}(A_i'\cdot Y)\geq 0\}$, and denote $P_{l} = \bigcap_{i=1}^l H_i$. We will use induction to show how to compute the poset data of $P_{l}$ for $l=1,\dots,k'$.

\textbf{Step (1).} We begin with $l=0$. Since the polyhedron $P_0$ is the entire space $\mathcal{P}(n)$, we initialize
\[
L^{face} = \{\varnothing\},\ L^{pos} = \{\varnothing\},\ L^{samp} = \{X\},
\]
and $\mathcal{A} = \varnothing$.

\textbf{Step (2).} We increase $l$ by $1$. By the induction assumption, we possess the poset data of $P_{l-1}$. That is, we have a set $\mathcal{A}$ of $n\times n$ symmetric matrices, such that $P_{l-1} = \bigcap_{A\in\mathcal{A}}\{\mathrm{tr}(A\cdot Y)\geq 0\}$, as well as lists denoted by $L^{face}$, $L^{pos}$, and $L^{samp}$ as required in Theorem \ref{thm:s5}.

\textbf{Step (3).} We will describe the computation of the poset data of $P_{l} = P_{l-1}\cap H_{l}$ from that of $P_{l-1}$. To begin with, we remove the first element of the list $\mathcal{A}'$, denoted by $A_l$, and append it to $\mathcal{A}$.

\textbf{Step (4).} We create a temporary list $L^{temp} = \{0,\dots,0\}$ of length equal to $|\mathcal{F}(P_{l-1})|$.

\textbf{Step (5).} We will replace the element $L^{temp}_j$ with a number from $\{1,\dots,6\}$ indicating the relative position of $(F_j,H_l)$, where $F_j\in \mathcal{F}(P_{l-1})$. To achieve this, we first determine the relative positions of the minimal faces in $\mathcal{F}(P_{l-1})$, i.e., the faces $F_j$ such that $L^{pos}_j = \varnothing$. Such a face $F_j$ is a plane $\bigcap_{i\in L_j^{face}}A_i^\perp$ in $\mathcal{P}(n)$. 

If $A_l\in span_{i\in L_j^{face}}(A_i)$, then $F_j\in \mathcal{F}^{(1)}_{H_l}(P_{l-1})$, and we set $L^{temp}_j = 1$. Otherwise, we apply the algorithm described by Lemma \ref{lem:5:2} to determine if $F_j\cap H_l = \varnothing$. If it is not empty, $F_j\in \mathcal{F}^{(6)}_{H_l}(P_{l-1})$, and we set $L^{temp}_j = 6$. If $F_j\cap H_l$ is empty, we compute $\mathrm{tr}(A_lX_j)$, where $X_j = L^{samp}_j$ is the sample point of $F_j$. If $\mathrm{tr}(A_lX_j)>0$, we set $L^{temp}_j = 2$; if $\mathrm{tr}(A_lX_j)<0$, we set $L^{temp}_j = 4$.

\textbf{Step (6).} Based on Lemma \ref{lem:5:2}, we determine the relative position of $F_j$ when the relative positions of all the proper faces of $F_j$ are determined. If $L^{temp}_{j'} = 6$ for any $j'\in L^{pos}_j$, we set $L^{temp}_{j} = 6$. If $L^{temp}_{j'_1} \in \{2,3\}$ and $L^{temp}_{j'_2} \in \{4,5\}$ for any $j'_1,j'_2\in L^{pos}_j$, we also set $L^{temp}_{j} = 6$. 

If neither of the cases above applies to $F_j$ and $L^{temp}_{j'} \in\{1,3,5\}$ for any $j'\in L^{pos}_j$, we check if $A_l\in span_{i\in L_j^{face}}(A_i)$. If so, we set $L^{temp}_{j} = 1$. Otherwise, if $\mathrm{tr}(A_lX_j)>0$, we set $L^{temp}_j = 3$; if $\mathrm{tr}(A_lX_j)<0$, we set $L^{temp}_j = 5$. 

If none of the cases above apply to $F_j$, either $L^{temp}_{j'} = 2$ for all $j'\in L^{pos}_j$, or $L^{temp}_{j'} = 4$ for all $j'\in L^{pos}_j$. Suppose that the former holds. We check if $span(F)\cap \partial H_l = \varnothing$ by applying Lemma \ref{lem:s5:1}. If it is empty, set $L^{temp}_j = 2$. Otherwise, let $X_0\in span(F)\cap \partial H_l$ be the sample point we derived from Lemma \ref{lem:s5:1}. If $\mathrm{tr}(X_0A_i)<0$ for any $1\leq i\leq l-1$, set $L^{temp}_j = 2$; otherwise, set $L^{temp}_j = 6$.

\textbf{Step (7).} We derive the data $L^{face}$ for $P_l$ according to Lemma \ref{lem:5:3}. For any number $j$, if $L^{temp}_j \in \{1,2,3\}$, $F_j$ remains in $\mathcal{F}(P_l)$, and no action is taken for such $j$.

If $L^{temp}_j \in \{4,5\}$, $F_j$ no longer exists as a face of $P_l$, thus we remove the elements $L^{face}_j$, $L^{temp}_j$, $L^{samp}_j$, and $L^{pos}_j$. Moreover, we remove any number $j$ that occurs in $L^{pos}$ and decrease by $1$ any number greater than $j$.

If $L^{temp}_j = 6$, both $F_j$ and $F_j\cap H$ are faces of $P_l$. Since $span(F_j\cap H) = span(F_j)$, we keep the element $L^{face}_{j}$ representing the new face $F_j\cap H$ instead of $F_j$. Furthermore, we append an element
$L^{face}_j\cup \{l\}$ to $L^{face}$ representing $F_j\cap \partial H$. Let this be the $\hat{j}$-th element of $L^{face}$.

\textbf{Step (8).} We derive the data $L^{pos}$ for $P_l$ according to Lemma \ref{lem:5:4}. The case $L^{temp}_j \in \{4,5\}$ will not occur. If $L^{temp}_j \in \{1,2,3\}$, no action is taken since the proper faces of $F_j$ in $\mathcal{F}^{(4)}_H(P_{l-1})\cup \mathcal{F}^{(5)}_H(P_{l-1})$ have been removed from $L^{pos}_j$ in step (6).

If $L^{temp}_j = 6$, then $F_j\cap \partial H$ is a proper face of $F_j\cap H$, and we append the element $\hat{j}$ to $L^{pos}_j$. Additionally, for each $l\in L^{pos}_j$, we check the value $L^{temp}_l$. The case $L^{temp}_l\in \{4,5\}$ will not occur, and no action is taken if $L^{temp}_l\in \{1,2,3\}$. If $L^{temp}_l = 6$, $F_l\cap \partial H$ is a proper face of both $F_j\cap H$ and $F_j\cap \partial H$, thus we append $\hat{l}$ to both $L^{pos}_j$ and $L^{pos}_{\hat{j}}$.

\textbf{Step (9).} We derive the data $L^{samp}$ for $P_l$ as described in Lemma \ref{lem:5:5}.

\textbf{Step (10).} We check if all numbers in $\{1,\dots,l\}$ appear in $L^{face}$ as facets. If a number $i\in\{1,\dots,l\}$ does not appear, we remove $A_i$ from the list $\{A_1,\dots,A_l\}$, decrease by $1$ any numbers greater than $i$ appearing in $L^{face}$, and decrease $l$ by $1$. 

\textbf{Step (11).} Repeat steps (2) through (10) if $\mathcal{A}'$ is non-empty. If $\mathcal{A}'$ is empty, the algorithm terminates, and the data $\mathcal{A}$, $L^{face}$, $L^{pos}$ and $L^{samp}$ are the required output of the theorem.
\vspace{12pt}
\section{Finite and infinite sided Dirichlet-Selberg domains}\label{sec:s4}
The objective of this section is to classify discrete subgroups $\Gamma<SL(3,\mathbb{R})$ into the following types:
\begin{itemize}
    \item The subgroup $\Gamma$ belongs to the \textbf{finitely-sided type}, if for every $X\in\mathcal{P}(3)$, the Dirichlet-Selberg domain $DS(X,\Gamma)$ is finitely-sided.
    \item The subgroup $\Gamma$ belongs to the \textbf{infinitely-sided type}, if for a generic choice of $X\in\mathcal{P}(3)$, the Dirichlet-Selberg domain $DS(X,\Gamma)$ is infinitely-sided.
\end{itemize}
Our study will focus on discrete Abelian subgroups of $SL(3,\mathbb{R})$ that consist of matrices with exclusively positive eigenvalues. Such matrices have favorable properties within $SL(3,\mathbb{R})$. In particular, if all eigenvalues of $g\in SL(3,\mathbb{R})$ are positive, then $g^k \in SL(3,\mathbb{R})$ for any $k\in\mathbb{R}$. 

For every $X\in\mathcal{P}(3)$, discrete subgroup $\Gamma<SL(3,\mathbb{R})$, and $g\in SL(3,\mathbb{R})$, the Dirichlet-Selberg domains $DS(X,\Gamma)$ and $DS(g.X,g^{-1}\Gamma g)$ are isometric. Therefore, our initial focus will be classifying the conjugacy classes of Abelian subgroups of $SL(3,\mathbb{R})$ with only positive eigenvalues.
\begin{prop}\label{2-gen}
    Let $\Gamma$ be a discrete Abelian subgroup of $SL(3,\mathbb{R})$ where all eigenvalues of each $\gamma\in \Gamma$ are positive real numbers. Then, $\Gamma$ is conjugate to a subgroup of $SL(3,\mathbb{R})$ generated by either of the following:

    (i) For cyclic $\Gamma$, the generators are displayed below:
    \begin{table}[H]
        \centering
        \begin{tabular}{c||c|c|c|c}
        \hline
        Type & (1) & (2) & (3) & (4)\\
        \hline
        Generator & $\begin{pmatrix} 1& 1 & 0\\ 0 & 1 & 0\\ 0 & 0 & 1\end{pmatrix}$ & $\begin{pmatrix} 1& 1 & 0\\ 0 & 1 & 1\\ 0 & 0 & 1\end{pmatrix}$ & \begin{tabular}[t]{@{}c@{}}$\begin{pmatrix} e^r& 0 & 0\\ 0 & e^s & 0\\ 0 & 0 & e^t\end{pmatrix}$\\($r+s+t = 0$;\\ $(r,s,t)\neq (0,0,0)$)\end{tabular} & \begin{tabular}[t]{@{}c@{}}$\begin{pmatrix} e^t& 1 & 0\\ 0 & e^t & 0\\ 0 & 0 & e^{-2t}\end{pmatrix}$\\($t\neq 0$)\end{tabular} \\
        \hline
    \end{tabular}
    \caption{Generators of cyclic discrete Abelian subgroups of $SL(3,\mathbb{R})$.}
    \label{tab:my_label}
    \end{table}
    (ii) For $2$-generated $\Gamma$, the generators are displayed below:
    \begin{table}[H]
        \centering
        \begin{tabular}{c||c|c|c|c|c}
        \hline
        Type & (1) & (1') & (2) & (3) & (4)\\
        \hline
        Generators & \begin{tabular}[t]{@{}c@{}}$\begin{pmatrix} 1& 1 & 0\\ 0 & 1 & 0\\ 0 & 0 & 1\end{pmatrix}$ \vspace{2mm}\\ $\begin{pmatrix} 1& 0 & 1\\ 0 & 1 & 0\\ 0 & 0 & 1\end{pmatrix}$\end{tabular} & \begin{tabular}[t]{@{}c@{}}$\begin{pmatrix} 1& 0 & 0\\ 0 & 1 & 1\\ 0 & 0 & 1\end{pmatrix}$ \vspace{2mm}\\ $\begin{pmatrix} 1& 0 & 1\\ 0 & 1 & 0\\ 0 & 0 & 1\end{pmatrix}$\end{tabular} & \begin{tabular}[t]{@{}c@{}}$\begin{pmatrix} 1& 1 & 0\\ 0 & 1 & 1\\ 0 & 0 & 1\end{pmatrix}$ \vspace{2mm}\\ $\begin{pmatrix} 1& a & b\\ 0 & 1 & a\\ 0 & 0 & 1\end{pmatrix}$\\ ($b\neq a(a-1)/2$)\end{tabular} & \begin{tabular}[t]{@{}c@{}}$\begin{pmatrix} e^r& 0 & 0\\ 0 & e^s & 0\\ 0 & 0 & e^t\end{pmatrix}$ \vspace{2mm}\\ $\begin{pmatrix} e^{r'}& 0 & 0\\ 0 & e^{s'} & 0\\ 0 & 0 & e^{t'}\end{pmatrix}$\\ ($r+s+t = $\\$r'+s'+t' = 0$)\end{tabular} & \begin{tabular}[t]{@{}c@{}}$\begin{pmatrix}e^t& 1 & 0\\ 0 & e^t & 0\\ 0 & 0 & e^{-2t}\end{pmatrix}$ \vspace{2mm}\\ $\begin{pmatrix}e^s& a & 0\\ 0 & e^s & 0\\ 0 & 0 & e^{-2s}\end{pmatrix}$ \\ ($(s,t)\neq (0,0)$)\end{tabular}\\
        \hline
        \end{tabular}
        \caption{Generators of $2$-generated discrete Abelian subgroups of $SL(3,\mathbb{R})$.}
        \label{tab:my_label2}
    \end{table}
\end{prop}
\begin{proof}
    First, we suppose that $\Gamma$ is cyclic, generated by $\gamma$. As all eigenvalues of $\gamma$ are positive, the Jordan canonical form of $\gamma$ falls into one of the types listed in Table \ref{tab:my_label}.

    Next, we suppose that $\Gamma$ is $2$-generated, which breaks into the following cases:
    \begin{enumerate}
        \item The subgroup $\Gamma$ is unipotent. For every non-identity element in $\Gamma$, the algebraic and geometric multiplicities of the eigenvalue $\lambda = 1$ are $3$ and $2$, respectively.
        \item The subgroup $\Gamma$ is unipotent. For every non-identity element in $\Gamma$, the algebraic multiplicity of the eigenvalue $\lambda = 1$ is $3$, while for at least one element in $\Gamma$, the geometric multiplicity of $\lambda = 1$ is $1$. 
        \item All elements in $\Gamma$ are diagonalizable by similarity.
        \item The group $\Gamma$ contains both a non-unipotent and a non-diagonalizable element.
    \end{enumerate}
    
    \textbf{Case (1).} We assume without loss of generality that one of the generators of $\Gamma$ is
    \[
    \gamma_1 = \begin{pmatrix} 1& 1 & 0\\ 0 & 1 & 0\\ 0 & 0 & 1\end{pmatrix},
    \]
    and denote the other generator by $\gamma_2$. The conditions $\gamma_1\gamma_2 = \gamma_2\gamma_1$ and $\det(\lambda I-\gamma_2) = (\lambda-1)^3$ imply that
    \[
    \gamma_2 = \begin{pmatrix} 1& a & b\\ 0 & 1 & 0\\ 0 & c & 1\end{pmatrix},\ a,b,c\in\mathbb{R}.
    \]
    Since the geometric multiplicity of the eigenvalue $1$ of $\gamma_2$ is $2$, we have $rank(\gamma_2 - I) = 1$, implying that $b=0$ or $c=0$. If $c=0$, then $b\neq 0$, as $\Gamma$ is discrete and $2$-generated. Additionally, we can assume that $b>0$, after replacing $\gamma_2$ with $\gamma_2^{-1}$ if necessary. Let
    \[
    g = \begin{pmatrix}
        b^{1/3} & 0 & 0\\ 0 & b^{1/3} & 0\\ 0 & -ab^{-2/3} & b^{-2/3}
    \end{pmatrix},
    \]
    then
    \[
    g^{-1}\gamma_1g = \gamma_1,\ g^{-1}\gamma_2g = \begin{pmatrix}
        1 & 0 & 1\\ 0 & 1 & 0\\ 0 & 0 & 1
    \end{pmatrix},
    \]
    which correspond to type (1) shown in Table \ref{tab:my_label2}.
    
    If $b=0$, then $c\neq 0$. We can similarly take the generators $\gamma_1$ and $\gamma_2$ to generators of type (1') via a similarity transformation.

    \textbf{Case (2).} Assuming that $\gamma_1\in \Gamma$ is unipotent with its eigenvalue $\lambda = 1$ having geometric multiplicity $1$. Without loss of generality, we have
    \[
    \gamma_1 = \begin{pmatrix}
        1 & 1 & 0\\ 0 & 1 & 1\\ 0 & 0 & 1
    \end{pmatrix}.
    \]
    Let $\gamma_2$ be another generator of $\Gamma$. Since $\gamma_1 \gamma_2 = \gamma_2\gamma_1$ and $\det(\gamma_2) = 1$, we deduce that
    \[
    \gamma_2 = \begin{pmatrix}
        1 & a  & b\\ 0 & 1 & a\\ 0 & 0 & 1
    \end{pmatrix},
    \]
    where $a,b\in\mathbb{R}$. As $\Gamma$ is $2$-generated and discrete, $\gamma_2\neq \gamma_1^a$, implying that $b\neq a(a-1)/2$. Thus, the generators $\gamma_1$ and $\gamma_2$ correspond to type (2) shown in Table \ref{tab:my_label2}.

    \textbf{Case (3).} If $\gamma_1\gamma_2 = \gamma_2\gamma_1$ while both $\gamma_1$ and $\gamma_2$ are diagonalizable, they are simultaneously diagonalizable. Hence, $\gamma_1$ and $\gamma_2$ correspond to type (3) shown in Table \ref{tab:my_label2}.

    \textbf{Case (4).} Let $\gamma_1$ be a generator of $\Gamma$ that has eigenvalues other than $1$. Without loss of generality, we can assume $\gamma_1$ to be a Jordan matrix after applying a similarity transformation.
    
    If $\gamma_1$ is diagonal, its diagonal elements cannot be pairwise distinct; otherwise, the other generator $\gamma_2$ will be diagonal, leading to a contradiction. Hence, we assume that $\gamma_1 = diag(e^s,e^s,e^{-2s})$, where $s\neq 0$. Since $\gamma_1 \gamma_2 = \gamma_2\gamma_1$, we deduce that
    \[
    \gamma_2 = \begin{pmatrix}
        * & * & 0\\ * & * & 0\\ 0 & 0 & *
    \end{pmatrix},
    \]
    and $\gamma_2$ is not diagonalizable. By a similarity transformation of the first two rows and columns, this becomes
    \[
    \gamma_2 = \begin{pmatrix}
        e^t & 1 & 0\\ 0 & e^t & 0\\ 0 & 0 & e^{-2t}
    \end{pmatrix},
    \]
    for a real number $t$.

    If $\gamma_1$ is not diagonalizable, then
    \[
    \gamma_1 = \begin{pmatrix}
        e^t & 1 & 0\\ 0 & e^t & 0\\ 0 & 0 & e^{-2t}
    \end{pmatrix},
    \]
    where $t\neq 0$. The condition $\gamma_1 \gamma_2 = \gamma_2\gamma_1$ implies that
    \[
    \gamma_2 = \begin{pmatrix}
        e^s & a & 0\\ 0 & e^s & 0\\ 0 & 0 & e^{-2s}
    \end{pmatrix}
    \]
    for some real numbers $s$ and $a$. In both cases, $\gamma_1$ and $\gamma_2$ correspond to type (4) shown in Table \ref{tab:my_label2}.
    
    In any of the four cases, let $\gamma_3\in SL(3,\mathbb{R})$ have only positive eigenvalues and commute with both $\gamma_1$ and $\gamma_2$. It is easy to show that $\gamma_3 = \gamma_1^{k_1}\gamma_2^{k_2}$ for $k_1,k_2\in\mathbb{R}$, implying that $\langle \gamma_1,\gamma_2,\gamma_3\rangle$ is either $2$-generated or non-discrete. Hence, Abelian discrete subgroups of $SL(3,\mathbb{R})$ with positive eigenvalues are at most $2$-generated.
\end{proof}

For clarity and organization, we further categorize cyclic groups of type (3) into two sub-types:
\begin{enumerate}
    \item[(3)] The generator $\gamma = diag(e^r,e^s,e^t)$, where all of $r,s,t$ are nonzero.
    \item[(3')] The generator $\gamma = diag(e^s,e^{-s},1)$, with $s\neq 0$.
\end{enumerate}
Using the classification of discrete Abelian subgroups of $SL(3,\mathbb{R})$ with positive eigenvalues, we present the main result of this section:
\begin{thm}\label{main:s4}
    Let $\Gamma$ be a discrete and free Abelian subgroup of $SL(3,\mathbb{R})$, generated by matrices with exclusively positive eigenvalues.
    \begin{itemize}
        \item If $\Gamma$ is a cyclic group of type (1), (3), or (4), or if it is a $2$-generated group of type (1) or (3), the Dirichlet-Selberg domain $DS(X,\Gamma)$ is finitely-sided for all $X\in\mathcal{P}(3)$.
        \item If $\Gamma$ is a cyclic group of type (2) or (3'), or if it is a $2$-generated group of type (1'), (2) or (4), the Dirichlet-Selberg domain $DS(X,\Gamma)$ is infinitely-sided for all $X$ in a dense and Zariski open subset of $\mathcal{P}(3)$.
    \end{itemize}
\end{thm}
\subsection{An equivalent condition}
Let $\Gamma$ be a discrete subgroup of $SL(3,\mathbb{R})$ and let $X\in\mathcal{P}(3)$. A facet of the Dirichlet-Selberg domain $D = DS(X,\Gamma)$ lies in a bisector $Bis(X,\gamma.X)$, where $\gamma\in\Gamma$. We denote such a facet by $F_{\gamma}$. The following lemma characterizes the existence of such facets:
\begin{lem}\label{lem:s4:1}
    Let $\Gamma$ be a discrete subgroup of $SL(n,\mathbb{R})$. Suppose that there exists a smooth function $g:\mathbb{R}^m\to SL(n,\mathbb{R})$ such that $\Gamma = g(\Lambda)$, where $\Lambda$ is a discrete subset of $\mathbb{R}^m$, $\mathbf{0}\in\Lambda$, and $g(\mathbf{0}) = e$. For $A,X\in \mathcal{P}(n)$, define a function $s_{X,A}^g: \mathbb{R}^m\to \mathbb{R},\ s_{X,A}^g(\mathbf{k}) = s(g(\mathbf{k}).X,A)$.

    Then for any $\mathbf{k}_0\in\Lambda\backslash \{\mathbf{0}\}$, the facet $F_{g(\mathbf{k}_0)}$ of $DS(X,\Gamma)$ exists if and only if there exists a matrix $A\in \mathcal{P}(n)$ such that $\mathbf{0}$ and $\mathbf{k}_0$ are the only minimum points of $s_{X,A}^g|_{\Lambda}$.
\end{lem}
\begin{proof}
    The existence of the facet $F_{g(\mathbf{k}_0)}$ is equivalent to the non-emptiness of
    \[
    \mathrm{int}(F_{g(\mathbf{k}_0)}) = Bis(X,g(\mathbf{k}_0).X)\cap\left(\bigcap_{\mathbf{k}\in\Lambda-\{ \mathbf{0},\mathbf{k}_0\}}\{Y|s(I,Y)<s(g(\mathbf{k}).I,Y)\}\right).
    \]
    Moreover, a point $A\in \mathcal{P}(n)$ lies in this intersection if and only if
    \[
    s^g_{X,A}(\mathbf{k}_0) = s^g_{X,A}(\mathbf{0}),\ s^g_{X,A}(\mathbf{k}) = s^g_{X,A}(\mathbf{0}),\ \forall \mathbf{k}\in\Lambda-\{ \mathbf{0},\mathbf{k}_0\},
    \]
    meaning that $\mathbf{0}$ and $\mathbf{k}_0$ are the only points where $s_{X,A}^g|_{\Lambda}$ attains its minima.
\end{proof}
\begin{rmk}
    Let $X\in\mathcal{P}(n)$, $\Lambda\subset\mathbb{R}^m$ be a discrete subset, $g:\mathbb{R}^m\to Sym_n(\mathbb{R})$ a smooth function, and $\Gamma =  g(\Lambda)$ a discrete subgroup of $SL(n,\mathbb{R})$. Lemma \ref{lem:s4:1} implies the following: 
    \begin{itemize}
        \item If for all but finitely many points $\mathbf{k}\in\Lambda$ and for every $A\in\mathcal{P}(n)$, the function $s^g_{X,A}|_\Lambda$ attains the minimum at points other than $\mathbf{k}$ and $\mathbf{0}$, then $F_\gamma$ is a facet of $DS(X,\Gamma)$ for only finitely many elements $\gamma = g(\mathbf{k})\in\Gamma$. Thus, the Dirichlet-Selberg domain $DS(X,\Gamma)$ is finitely-sided.
        \item If there are infinitely many points $\mathbf{k}\in\Lambda$ such that $\mathbf{k}$ and $\mathbf{0}$ are the only two minimum points of $s^g_{X,A}|_\Lambda$ for a certain $A\in\mathcal{P}(n)$, then $F_\gamma$ is a facet of $DS(X,\Gamma)$ for infinitely many elements $\gamma = g(\mathbf{k})\in\Gamma$. Thus, the Dirichlet-Selberg domain $DS(X,\Gamma)$ is infinitely-sided.
    \end{itemize}
\end{rmk}
Below is a generalization of Lemma \ref{lem:s4:1}:
\begin{cor}\label{lem:s4:11}
    Let $\Gamma<SL(n,\mathbb{R})$ be a discrete subgroup. Suppose that there is a smooth function $g:\mathbb{R}^m\to SL(n,\mathbb{R})$ such that $\Gamma = g(\Lambda)$, where $\Lambda\subset \mathbb{R}^m$ is a discrete subset, $\mathbf{0}\in \Lambda$, and $g(\mathbf{0}) = e$. Define the functions $s_{X,A}^g$ analogously to Lemma \ref{lem:s4:1}.

    Suppose that there exists a matrix $A\in \mathcal{P}(n)$ and a finite subset $\Lambda_0\subset \Lambda$ satisfying the following conditions:
    \begin{enumerate}
        \item The point $\mathbf{0}\in \Lambda_0$.
        \item There exists a nonzero point $\mathbf{k_0}\in \Lambda_0$ such that $s_{X,A}^g(\mathbf{k_0}) = s_{X,A}^g(\mathbf{0})$.
        \item For any $\mathbf{k}\in \Lambda_0$, $s_{X,A}^g(\mathbf{k})\leq s_{X,A}^g(\mathbf{0})$; for any $\mathbf{k}\in \Lambda\backslash\Lambda_0$, $s_{X,A}^g(\mathbf{k})> s_{X,A}^g(\mathbf{0})$.
    \end{enumerate}
    Then the Dirichlet-Selberg domain $DS(X,\Gamma)$ has a facet $F_{g(\mathbf{k})}$ for at least one element $\mathbf{k}\in \Lambda_0\backslash\{\mathbf{0}\}$.
\end{cor}
\begin{proof}
    Assume that $\Lambda_0 = \{\mathbf{0},\mathbf{k}_0,\mathbf{k}_1,\dots,\mathbf{k}_r\}$, where 
    \[
    s_{X,A}^g(\mathbf{k}_0)\geq s_{X,A}^g(\mathbf{k}_1)\geq\dots\geq s_{X,A}^g(\mathbf{k}_r).
    \]
    Define
    \[
    \Lambda'_i: = (\Lambda\backslash \Lambda_0)\cup \{\mathbf{0},\mathbf{k}_0,\dots,\mathbf{k}_i\},\ i=0,\dots, r,
    \]
    thus $\Lambda'_r = \Lambda$. We will prove the following assertion by induction on $i$:
    \begin{enumerate}
        \item[(*)] The Dirichlet-Selberg domain $DS(X,g(\Lambda_i'))$ contains a facet $F_{g(\mathbf{k}_j)}$ for a certain $j\in\{0,\dots, i\}$.
    \end{enumerate}
    
    By Lemma \ref{lem:s4:1}, $DS(X,g(\Lambda'_0))$ contains the facet $F_{g(\mathbf{k}_0)}$. This serves as the base case for the assertion (*).

    Assume that the claim (*) holds for $(i-1)$. To prove the assertion for $i$, note that
    \[
    DS(X,g(\Lambda_i')) = DS(X,g(\Lambda_{i-1}'))\cap H_i,\ H_i = \{Y|s(I,Y)\leq s(g(\mathbf{k}_i).I,Y)\}.
    \]
    Let $F_{g(\mathbf{k}_{j'})}\in\mathcal{F}(DS(X,g(\Lambda_{i-1}')))$ be the facet in the induction assumption, where $j'\in \{0,1,\dots,i-1\}$. If the position of $(DS(X,g(\Lambda_{i-1}')), H_i)$ is not case (6) in Lemma \ref{lem:5:1}, then $F_{g(\mathbf{k}_{j'})}$ remains a facet of $DS(X,g(\Lambda_i'))$. If it is case (6) in Lemma \ref{lem:5:1}, then $\partial H_i \cap DS(X,g(\Lambda_{i-1}'))$ is the facet $F_{g(\mathbf{k}_{i})}$ of $DS(X,g(\Lambda_i'))$. This confirms the claim (*).

    Particularly, when $i=r$, (*) is equivalent to the lemma's statement.
\end{proof}
The proof of Theorem \ref{main:s4} consists of a series of assertions that will be described in the subsequent sections. For clarity, we shall consistently denote the $(i,j)$ entry of $X^{-1}$ and $A$ by $x^{ij}$ and $a_{ij}$, respectively.
\subsection{Discrete subgroups of \texorpdfstring{$SL(3,\mathbb{R})$}{Lg}: finitely-sided cases}
In this section, we will examine the cases when the Dirichlet-Selberg domain $DS(X,\Gamma)$ is finitely-sided for all $X\in \mathcal{P}(3)$, as asserted in Theorem \ref{main:s4}.

\begin{proof}[Proof of Theorem \ref{main:s4}, cyclic group of type (1)]
    We interpret the group $\Gamma$ generated by
    \[
    \gamma = \begin{pmatrix}
        1 & 1 & 0\\ 0 & 1 & 0\\ 0 & 0 & 1
    \end{pmatrix}
    \]
    as the image of $\mathbb{Z}$ under the $1$-variable function
    \[
    g(k) = \begin{pmatrix}
        1 & k & 0\\ 0 & 1 & 0\\ 0 & 0 & 1
    \end{pmatrix},\ \forall k\in\mathbb{R}.
    \]
    In this case, the function $s^g_{X,A}$ in Lemma \ref{lem:s4:1} becomes
    \begin{multline*}
        s_{X,A}^g(k) = s(g(k).X,A) = \mathrm{tr}(((\gamma^{\mathrm{T}})^k.X.\gamma^k)^{-1}.A) \\
        = x^{11}a_{22}k^2 + 2(x^{11}a_{12}+x^{12}a_{22}+x^{13}a_{23})k + s(X,A),
    \end{multline*}
    which is a quadratic polynomial in $k$. Since both $X$ and $A$ are positive definite, the leading term of $s_{X,A}^g$ has a positive coefficient. If $s_{X,A}^g$ evaluates the same at $0$ and $k_0$, then $s_{X,A}^g(k)<s_{X,A}^g(0)$ for any $k$ between $0$ and $k_0$.

    Therefore, if $s_{X,A}^g|_\mathbb{Z}$ attains its minimum at both $0$ and $k_0$, the integer $k_0$ can only be $1$ or $-1$. Lemma \ref{lem:s4:1} with $\Lambda = \mathbb{Z}$ implies that the Dirichlet-Selberg domain $DS(X,\Gamma)$ is $2$-sided for any $X\in\mathcal{P}(3)$.
\end{proof}
\begin{proof}[Proof of Theorem \ref{main:s4}, cyclic group of type (3)]
    We interpret the cyclic subgroup $\Gamma$ of type (3) as a one-parameter family given by
    \[
    g(k) = \gamma^k = diag(e^{rk},e^{sk},e^{tk}),\ k\in\mathbb{Z},
    \]
    where $r+s+t =0$ and $r,s,t\neq 0$. Without loss of generality, we assume that $r\geq s>0>t$. For a given $X$ and $A\in\mathcal{P}(3)$, the function $s^g_{X,A}$ becomes
    \[
    s^g_{X,A}(k) = x^{11}a_{11}e^{-2rk} + x^{22}a_{22}e^{-2sk} +x^{33}a_{33}e^{-2tk} + 2x^{23}a_{23}e^{rk}+2x^{13}a_{13}e^{sk}+2x^{12}a_{12}e^{tk}.
    \]
    Since $x^{ii}$, $a_{ii}>0$ for $i=1,2,3$, there is a unique $k_c\in\mathbb{R}$ such that
    \[
    \sqrt{x^{11}a_{11}}e^{-rk_c} + \sqrt{x^{22}a_{22}}e^{-sk_c} = \sqrt{x^{33}a_{33}}e^{-tk_c}.
    \]
    Therefore, $s^g_{X,A}(k) = c\cdot f(k-k_c)$, where
    \[
    f(n) = e^{2(r+s)n} + 2p\alpha_{13}e^{sn}+2(1-p)\alpha_{23}e^{rn} + p^2e^{-2rn}+ (1-p)^2e^{-2sn}+2p(1-p)\alpha_{12}e^{-(r+s)n},
    \]
    and
    \[
    c= x^{33}a_{33}e^{-2tk_c}>0,\ p= \frac{\sqrt{x^{11}a_{11}}e^{-rk_c}}{\sqrt{x^{33}a_{33}}e^{-tk_c}}\in (0,1),\ \alpha_{ij} = \frac{x^{ij}a_{ij}}{\sqrt{x^{ii}x^{jj}a_{ii}a_{jj}}}.
    \]
    For any $i\neq j$, $|\alpha_{ij}|<\xi: = \max_{i\neq j}\frac{|x^{ij}|}{\sqrt{x^{ii}x^{jj}}}$, $\xi<1$ and depends only on $X$. For any $0\leq p\leq 1$ and any $-\xi\leq \alpha_{ij}\leq \xi$, $\lim_{n\to \infty}f'(n) = \infty$ and $\lim_{n\to -\infty}f'(n) = -\infty$. Since $f'(n;p,\alpha_{ij})$ is continuous with respect to $p$ and $\alpha_{ij}$, there exists a number $N>0$ determined by $r,s,t$, and $\xi$, such that
    \[
    f'(n;p,\alpha_{ij})>0,\ \forall n>N;\ f'(n;p,\alpha_{ij})<0,\ \forall n<-N,
    \]
    for every $(p,\alpha_{12},\alpha_{13},\alpha_{23})$ in the compact region $[0,1]\times [-\xi,\xi]^3$. If $k=0$ and $k = k_0$ are the only minimum points of $s^g_{X,A}|_\mathbb{Z}$, then $(-k_c)$ and $(k_0-k_c)$ are the only minimum points of $f|_{\mathbb{Z} - k_c}$. Thus $|k_c|,|k_0 - k_c|<N+1$, implying that $|k_0|<2(N+1)$. Since there are finitely many choices of such $k_0$, Lemma \ref{lem:s4:1} with $\Lambda = \mathbb{Z}$ implies that the Dirichlet-Selberg domain $DS(X,\Gamma)$ is finitely-sided for any $X\in\mathcal{P}(3)$.
\end{proof}
\begin{proof}[Proof of Theorem \ref{main:s4}, cyclic group of type (4)]
    We interpret the cyclic subgroup $\Gamma$ of type (4) as a one-parameter family given by
    \[
    g(k) = \gamma^k = \begin{pmatrix}
        e^{sk} & ke^{s(k-1)} & 0\\ 0 & e^{sk} & 0\\ 0 & 0 & e^{-2sk}
    \end{pmatrix},\ k\in\mathbb{Z},
    \]
    where $s\neq 0$. Without loss of generality, we assume that $s>0$. The function $s^g_{X,A}$ becomes
    \[
    \begin{split}
        & s^{g}_{X,A}(k) = x^{33}a_{33}e^{4sk} + (2x^{13}a_{13}+2x^{23}a_{23}-2ke^{-s}x^{23}a_{13})e^{sk}\\
        & + (x^{11}a_{11}+x^{22}a_{22}+2x^{12}a_{12}-2ke^{-s}(x^{12}a_{11}+x^{22}a_{12})+k^2e^{-2s}x^{22}a_{11})e^{-2sk}.
    \end{split}
    \]
    A unique $k_c\in\mathbb{R}$ exists such that
    \[
    (\sqrt{x^{11}a_{11}} + \sqrt{x^{22}a_{22}}+e^{-s}\sqrt{x^{22}a_{11}})e^{-sk_c} = \sqrt{x^{33}a_{33}}e^{2sk_c}.
    \]
    Hence, $s^g_{X,A}(k) = c\cdot f(k-k_c)$, where
    \[
    \begin{split}
         & f(n) = e^{4sn} + (2\alpha_{13}p+2\alpha_{23}q-2\beta_3 (1-p-q)n)e^{sn}\\
         & + (p^2+q^2+2\alpha_{12}pq - 2(\beta_1p+\beta_2q)(1-p-q)n + (1-p-q)^2n^2)e^{-2sn},
    \end{split}
    \]
    and
    \[
    c = x^{33}a_{33}e^{4sk_c}>0,\ p = \frac{\sqrt{x^{11}a_{11}}e^{-sk_c}}{\sqrt{x^{33}a_{33}}e^{2sk_c}},\ q = \frac{\sqrt{x^{22}a_{22}}e^{-sk_c}}{\sqrt{x^{33}a_{33}}e^{2sk_c}},
    \]
    \[
    \alpha_{ij} = \frac{x^{ij}a_{ij}}{\sqrt{x^{ii}x^{jj}a_{ii}a_{jj}}},\ \beta_1 = \frac{x^{12}}{\sqrt{x^{11}x^{22}}},\ \beta_2 = \frac{a_{12}}{\sqrt{a_{11}a_{22}}},\ \beta_3 = \frac{x^{23}a_{13}}{\sqrt{x^{22}x^{33}a_{11}a_{33}}},
    \]
    where $p,q>0$, $p+q< 1$, $|\alpha_{ij}|,|\beta_1|,|\beta_3|<\xi: = \max_{i\neq j}\frac{|x^{ij}|}{\sqrt{x^{ii}x^{jj}}}$, and $|\beta_2|<1$. Analogously to the proof for cyclic groups of type (3), there exists a number $N>0$ determined by $s$ and $\xi$, such that
    \[
    f'(n;p,q,\alpha_{ij},\beta_i)>0,\ \forall n>N; f'(n;p,q,\alpha_{ij},\beta_i)<0,\ \forall n<-N,
    \]
    and for every $(p,q,\alpha_{ij},\beta_i)$ in the compact region $\{(p,q)|p,q\geq 0,p+q\leq 1\}\times [-\xi,\xi]^5\times [-1,1]$. Hence, if $k=0$ and $k=k_0$ are the only minimum points of $s^g_{X,A}|_{\mathbb{Z}}$, then $|k_0|<2(N+1)$ analogously. Lemma \ref{lem:s4:1} with $\Lambda = \mathbb{Z}$ implies that $DS(X,\Gamma)$ is finitely-sided for any $X\in\mathcal{P}(3)$.
    \end{proof}
We now consider $2$-generated subgroups. To utilize Lemma \ref{lem:s4:1}, we investigate the level curves of the function $s_{X,A}^g:\mathbb{R}^2\to\mathbb{R}$.
\begin{proof}[Proof of Theorem \ref{main:s4}, $2$-generated group of type (1)]
    We interpret the group $\Gamma$ as a two-parameter family
    \[
    g(k,l) = \gamma_1^k\gamma_2^l = \begin{pmatrix}
        1 & k & l\\ 0 & 1 & 0\\ 0 & 0 & 1
    \end{pmatrix},\ \forall (k,l)\in\mathbb{Z}^2.
    \]
    Thus, the function $s_{X,A}^g$ is expressed as
    \[
    s_{X,A}^g(k,l) = a_{11}(x^{22}(k-k_c)^2 + 2x^{23}(k-k_c)(l-l_c)+x^{33}(l-l_c)^2)+const,
    \]
    where
    \[
    k_c = \frac{a_{12}}{a_{11}}+\frac{x^{12}x^{33} - x^{13}x^{23}}{x^{22}x^{33} - (x^{23})^2},\ l_c = \frac{a_{13}}{a_{11}}+\frac{x^{13}x^{22} - x^{12}x^{23}}{x^{22}x^{33} - (x^{23})^2}.
    \]
    As $x^{22}x^{33}>(x^{23})^2$, the level curves of $s^g_{X,A}$ are ellipses centered at $(k_c,l_c)$. These ellipses share the same eccentricity, depending solely on $X$.

    If $(0,0)$ and $(k_0,l_0)$ are the only minimum points of $s_{X,A}^g|_{\mathbb{Z}^2}$, then there exists a level curve of $s_{X,A}^g$ that surrounds these two points and excludes all other points in $\mathbb{Z}^2$. Since any circle of diameter greater than $\sqrt{5}$ encompasses at least $3$ points in $\mathbb{Z}^2$, the minor axis length of the mentioned level curve is at most $\sqrt{5}$. Thus, the major axis length of the level curve is bounded by a constant depending solely on $X$. Consequently, there are only finitely many choices of $(k_0,l_0)$ such that $(0,0)$ and $(k_0,l_0)$ are the only minimum points of $s^g_{X,A}|_{\mathbb{Z}^2}$. By Lemma \ref{lem:s4:1} with $\Lambda = \mathbb{Z}^2$, the Dirichlet-Selberg domain $DS(X,\Gamma)$ is finitely-sided for any $X\in\mathcal{P}(3)$.
\end{proof}
\begin{proof}[Proof of Theorem \ref{main:s4}, $2$-generated group of type (3)]
    We interpret the group $\Gamma$ as a two-parameter family
    \[
    g(k,l,m) = diag(e^k,e^l,e^m),\ \forall (k,l,m)\in\Lambda,
    \]
    where the domain of $g$ is the $2$-plane
    \[
    \{(k,l,m)\in \mathbb{R}^3|k+l+m = 0\},
    \]
    and $\Lambda = \mathbb{Z}(r,s,t)\oplus \mathbb{Z}(r',s',t')$ is contained in this $2$-plane.

    The function $s_{X,A}^g$ is given by
    \begin{align*}
        & s_{X,A}^g(k,l,m)\\
        & = 
        (x^{11}a_{11})e^{2k}+(x^{22}a_{22})e^{2l}+(x^{33}a_{33})e^{2m} + (2x^{23}a_{23})e^{-k} + (2x^{13}a_{13})e^{-l} + (2x^{12}a_{12})e^{-m}\\
        & = c(e^{2(k-k_c)}+e^{2(l-l_c)}+e^{2(m-m_c)} + 2\alpha_{23}e^{-(k-k_c)} + 2\alpha_{13}e^{-(l-l_c)} + 2\alpha_{12}e^{-(m-m_c)}),
    \end{align*}
    where
    \[
    c = \sqrt[3]{x^{11}x^{22}x^{33}a_{11}a_{22}a_{33}},\ \alpha_{12} = \frac{x^{12}a_{12}}{\sqrt{x^{11}x^{22}a_{11}a_{22}}},\ \alpha_{13} = \frac{x^{13}a_{13}}{\sqrt{x^{11}x^{33}a_{11}a_{33}}},\ \alpha_{23} = \frac{x^{23}a_{23}}{\sqrt{x^{22}x^{33}a_{22}a_{33}}},
    \]
    and 
    \[
    k_c = -\frac{1}{3}\log\frac{x^{11}a_{11}}{\sqrt{x^{22}x^{33}a_{22}a_{33}}},\ l_c =  -\frac{1}{3}\log\frac{x^{22}a_{22}}{\sqrt{x^{11}x^{33}a_{11}a_{33}}},\ m_c =  -\frac{1}{3}\log\frac{x^{33}a_{33}}{\sqrt{x^{11}x^{22}a_{11}a_{22}}}.
    \]
    Here, the point $(k_c,l_c,m_c)$ lies on the plane $\{k+l+m = 0\}$. For any $i\neq j$, the coefficient $|\alpha_{ij}|<\xi := \max_{i\neq j}\frac{|x^{ij}|}{\sqrt{x^{ii}x^{jj}}}$, where $\xi<1$ and depends only on $X$. Thus,
    \[
    f_-(k-k_c,l-l_c,m-m_c)\leq s_{X,A}^g(k,l,m)/c\leq f_+(k-k_c,l-l_c,m-m_c),\ \forall k+l+m = 0,
    \]
    where
    \[
    f_\pm(k-k_c,l-l_c,m-m_c): = e^{2(k-k_c)}+e^{2(l-l_c)}+e^{2(m-m_c)}\pm 2\xi(e^{-(k-k_c)}+e^{-(l-l_c)}+e^{-(m-m_c)}).
    \]
    Let $d = d(k,l,m)$ represent $1/\sqrt{6}$ times the Euclidean distance between $(k_c,l_c,m_c)$ and $(k,l,m)$. When $d$ is fixed, it is evident that $f_+(k-k_c,l-l_c,m-m_c)$ reaches its maximum at $(k-k_c,l-l_c,m-m_c) = (2d,-d,-d)$, and $f_-(k-k_c,l-l_c,m-m_c)$ reaches its minimum at $(k-k_c,l-l_c,m-m_c) = (-2d,d,d)$. Therefore,
    \[
    2(1-\xi)e^d-4\xi e^{-d/2}+e^{-2d} := f_-(d) \leq s_{X,A}^g(k,l,m)/c\leq f_+(d) :=  2e^{2d} + 4\xi e^{d/2}+ 2(1+\xi)e^{-d}.
    \]
    
    If a level curve of $s_{X,A}^g(k,l,m)$ surrounds only two points $(0,0,0)$ and $(k_0,l_0,m_0)$ among the points in $\Lambda$, the inscribed radius of the level curve is less than a certain constant $\rho>0$ determined by $\Lambda$. Since $\xi<1$, it follows that $\lim_{d\to\infty}f_-(d) = \infty$. Hence, there exists a constant $R<\infty$ depending solely on $\xi$ and $\rho$, such that $f_+(\rho) = f_-(R)$. Consequently, the value taken on this level curve is less than $f_+(\rho) = f_-(R)$, implying that the diameter of the level curve is less than $2R$. Therefore, there are only finitely many choices of $(k_0,l_0,m_0)\in\Lambda$, such that $(0,0,0)$ and $(k_0,l_0,m_0)$ are the only minimum points of $s^g_{X,A}|_{\Lambda}$. By Lemma \ref{lem:s4:1}, the Dirichlet-Selberg domain $DS(X,\Gamma)$ is finitely-sided for any $X\in\mathcal{P}(3)$.
\end{proof}
\subsection{Discrete subgroups of \texorpdfstring{$SL(3,\mathbb{R})$}{Lg}: infinitely-sided cases}
We proceed to the cases when the Dirichlet-Selberg domain $DS(X,\Gamma)$ is infinitely-sided for a generic choice of $X\in \mathcal{P}(3)$ as asserted in Theorem \ref{main:s4}.

\begin{proof}[Proof of Theorem \ref{main:s4}, cyclic group of type (2)]
    We interpret the cyclic group $\Gamma$ of type (2) as a one-parameter family
    \[
    g(k): = \gamma^k = \begin{pmatrix}
        1 & k & k(k+1)/2\\ 0 & 1 & k \\ 0 & 0 & 1 
    \end{pmatrix},\ k\in\mathbb{Z}.
    \]
    Thus for any $A = (a_{ij})$ and $X = (x^{ij})^{-1}\in\mathcal{P}(3)$, the function $s^g_{X,A}$ is expressed as
    \begin{align*}
        & s^g_{X,A}(k) = s(\gamma^k.X,A) = \mathrm{tr}(((\gamma^k)^{\mathrm{T}}X\gamma^k)^{-1}A)\\
        & = (x^{33}a_{11}/4)k^4 + \left(-x^{33}a_{12}+\left(x^{33}/2-x^{23}\right)a_{11}\right)k^3\\
        & + \left(x^{33} a_{13}+x^{33} a_{22}+(3 x^{23}-x^{33})a_{12}+\left(x^{33}/4-x^{23}+x^{13}+x^{22}\right)a_{11}\right) k^2\\
        & + \left(-2x^{33}a_{23}+(x^{33}-2x^{23})a_{13}-2x^{23}a_{22}+(x^{23}-2x^{13}-2x^{22})a_{12}+(x^{13}-2x^{12})a_{11}\right)k\\
        & + (x^{33}a_{33}+x^{22}a_{22}+x^{11}a_{11}+2x^{23}a_{23}+2x^{13}a_{13}+2x^{12}a_{12}),
    \end{align*}
    a quartic polynomial in $k$. We will demonstrate that for any $X\in\mathcal{P}(3)$ and any $k_0\in\mathbb{Z}$, there exists a positive definite matrix $A$ such that
    \[
    s^g_{X,A}(k) = k^2(k-k_0)^2 + const,
    \]
    representing a quartic function whose global minimum points are $k=0$ and $k = k_0$. By comparing coefficients of the $k^4$ and $k^3$ terms, we derive
    \[
    x^{33}a_{11}/4 = 1,\quad -x^{33}a_{12}+\left(x^{33}/2-x^{23}\right)a_{11} = -2{k_0},
    \]
    implying that
    \[
    a_{11} = 4/x^{33}>0,\ a_{12} = (2({k_0}+1)x^{33} - 4x^{23})/((x^{33})^2).
    \]
    Furthermore, we let $a_{22}$ be sufficiently large, so that $a_{11}a_{22}>a_{12}^2$. By comparing the coefficients of the $k^2$ and $k^1$ terms, we deduce:
    \begin{align*}
        & x^{33} a_{13}+x^{33} a_{22}+(3 x^{23}-x^{33})a_{12}+\left(x^{33}/4-x^{23}+x^{13}+x^{22}\right)a_{11} = {k_0}^2,\\
        & -2x^{33}a_{23}+(x^{33}-2x^{23})a_{13}-2x^{23}a_{22}+(x^{23}-2x^{13}-2x^{22})a_{12}+(x^{13}-2x^{12})a_{11} = 0,
    \end{align*}
    which is a linear equation system in the unknowns $a_{13}$ and $a_{23}$. This linear equation system has an invertible coefficient matrix and thus has a unique solution for $a_{13}$ and $a_{23}$.
    
    Lastly, we determine $a_{33}$ by setting $\det(A) = 1$. These steps yield a matrix $A\in F_{\gamma^k}$. By Lemma \ref{lem:s4:1} with $\Lambda = \mathbb{Z}$, the Dirichlet-Selberg domain $DS(X,\Gamma)$ is infinitely sided for any $X\in\mathcal{P}(3)$.
\end{proof}
\begin{proof}[Proof of Theorem \ref{main:s4}, cyclic group of type (3')]
    We interpret the cyclic group $\Gamma$ of type (3') as a one-parameter family
    \[
    g(k): = \gamma^k = diag(e^{sk},e^{-sk},1),\ k\in\mathbb{Z}.
    \]
    Thus, the function $s^g_{X,A}$ is expressed as
    \[
    s_{X,A}^g(k) = s(\gamma^k.X,A) = a_{22}x^{22}e^{2s k}+2a_{23}x^{23}e^{sk}+2a_{13}x^{13}e^{-sk}+a_{11}x^{11}e^{-2sk}+const.
    \]
    Additionally, we assume that $x^{23}, x^{13}\neq 0$. We will demonstrate that for any $k_0\in\mathbb{Z}$, there exists a positive definite matrix $A$ such that
    \[
    s^g_{X,A}(k) = e^{2s k}-2(e^{sk_0}+1)e^{sk}-2e^{sk_0}(e^{sk_0}+1)e^{-sk}+e^{2sk_0}e^{-2sk} + const,
    \]
    which represents a function with minimum points at $k= 0$ and $k = k_0$. A suitable solution is given by:
    \[
    a_{11} = e^{2tk_0}/x^{11},\quad a_{12} = 0,\quad a_{22} = 1/x^{22},\quad a_{23} = -(e^{tk_0}+1)/x^{23},\quad a_{13} = -e^{tk_0}(e^{tk_0}+1)/x^{13},
    \]
    and $a_{33}$ is determined by $\det(A) = 1$. Analogously to the preceding case, the existence of such a solution for $A$ implies that $DS(X,\Gamma)$ is infinitely sided whenever the center $X$ does not belong to the proper Zariski closed subset $\{X = (x^{ij})^{-1}\in\mathcal{P}(3)|x^{13}x^{23} = 0\}$.
\end{proof}
We proceed to consider the cases of $2$-generated subgroups:
\begin{proof}[Proof of Theorem \ref{main:s4}, $2$-generated group of type (1')]
    We interpret the group $\Gamma$ as a two-parameter family, given by
    \[
    g(k,l) = \gamma_1^k\gamma_2^l = \begin{pmatrix}
        1 & 0 & l\\ 0 & 1 & k\\0 & 0 & 1
    \end{pmatrix},\ \forall (k,l)\in\mathbb{Z}^2,
    \]
    and the function $s_{X,A}^g$ is expressed as
    \[
    s_{X,A}^g(k,l) = x^{33}(a_{22}(k-k_c)^2 + 2a_{12}(k-k_c)(l-l_c)+a_{11}(l-l_c)^2)+const,
    \]
    where
    \[
    k_c = \frac{x^{23}}{x^{33}}+\frac{a_{11}a_{23} - a_{12}a_{13}}{a_{11}a_{22}-a_{12}^2},\ l_c = \frac{x^{13}}{x^{33}}+\frac{a_{22}a_{13} - a_{12}a_{23}}{a_{11}a_{22}-a_{12}^2}.
    \]
    We claim that for any coprime pair $(k_0,l_0)\in\mathbb{Z}^2$, there exists a matrix $A\in\mathcal{P}(3)$ such that $s^g_{X,A}|_{\mathbb{Z}^2}$ achieves its global minimum at $(k,l) = (0,0)$ and $(k,l) = (k_0,l_0)$. Specifically, we claim that the matrix $A$ can be chosen so that
    \begin{equation}\label{equ:ellip}
    s^g_{X,A}(k,l) = \epsilon^2(k_0(k-k_0/2)+l_0(l-l_0/2))^2 +(l_0(k-k_0/2)-k_0(l-l_0/2))^2+const
    \end{equation}
    for arbitrarily small $\epsilon>0$. In other words, a particular level set of $s_{X,A}^g$ is an ellipse, with its major axis being the line segment between $(0,0)$ and $(k_0,l_0)$, and its minor axis length being $\epsilon$ times the length of the major axis.

    A comparison of the coefficients of the $k^2$, $kl$ and $l^2$ terms yields that:
    \[
    a_{11} = \epsilon^2 l_0^2 + k_0^2,\ a_{12} = -(1-\epsilon^2)k_0l_0,\ a_{22} = \epsilon^2 k_0^2 + l_0^2.
    \]
    Therefore, $a_{11}a_{22} - a_{12}^2 = 2\epsilon^2(k_0^2+l_0^2)^2>0$. A comparison of the coefficients of the $k^1$ and $l^1$ terms implies that:
    \[
    k_0/2 = k_c = \frac{x^{23}}{x^{33}}+\frac{a_{11}a_{23} - a_{12}a_{13}}{a_{11}a_{22}-a_{12}^2},\ l_0/2 = l_c = \frac{x^{13}}{x^{33}}+\frac{a_{22}a_{13} - a_{12}a_{23}}{a_{11}a_{22}-a_{12}^2}.
    \]
    This can be interpreted as a $2$-variable linear equation system in the unknowns $a_{13}$ and $a_{23}$. The coefficient matrix of the equation system is
    \[
    \begin{pmatrix}a_{11} & -a_{12}\\ -a_{12} & a_{22}\end{pmatrix},
    \]
    which is invertible. Thus, a unique solution for $a_{13}$ and $a_{23}$ is determined by $k_0$, $l_0$, $\epsilon$ and $X$. Finally, we determine $a_{33}$ by setting $\det(A) = 1$.
    
    The matrix $A$ we constructed is positive definite. Furthermore, a certain level curve of $s^g_{X,A}$ is an ellipse whose major axis is the line segment between $(0,0)$ and $(k_0,l_0)$. Assuming that $k_0$ and $l_0$ are coprime, the ellipse will exclude all other points in $\mathbb{Z}^2$ if the minor axis is sufficiently short, i.e., the number $\epsilon$ is sufficiently small. By Lemma \ref{lem:s4:1}, the Dirichlet-Selberg domain $DS(X,\Gamma)$ is infinitely-sided for any $X\in\mathcal{P}(3)$.
\end{proof}
\begin{proof}[Proof of Theorem \ref{main:s4}, $2$-generated group of type (2)]
    We interpret the $2$-generated subgroup $\Gamma$ of type (2) for given constants $a$ and $b$ as a two-parameter family:
    \[
    g(k,l) = \begin{pmatrix}
        1 & -k & k^2-l\\ 0 & 1 & -k\\ 0 & 0 & 1
    \end{pmatrix},\ \forall (k,l)\in \Lambda = \Lambda(a,b),
    \]
    where
    \[
    \Lambda(a,b) = \left\{(k,l)\left|k=x+a y, l = \frac{1}{2} \left(a^2( y^2-y)+2 a x y+2 b y+x^2-x\right),\ (x,y)\in \mathbb{Z}\right.\right\}
    \]
    is a discrete subset of $\mathbb{R}^2$. The function $s^g_{X,A}$ is expressed as
    \[
    \begin{split}
        & s^g_{X,A}(k,l) = (a_{11} {x^{22}}+2 a_{12} {x^{23}}+a_{22} {x^{33}})(k-k_c)^2\\
        & +2(a_{11} {x^{23}}+a_{12} {x^{33}})(k-k_c) (l-l_c) +(a_{11} {x^{33}})(l-l_c)^2+const,
    \end{split}
    \]
    where $k_c$ and $l_c$ are given by:
    \begin{align*}
        & k_c = -\frac{\splitfrac{a_{11}^2( {x^{12}} {x^{33}}- {x^{13}} {x^{23}})+a_{11} a_{12}( {x^{22}} {x^{33}}- {(x^{23})}^2)}{+(a_{11} a_{22}-a_{12}^2) {x^{23}} {x^{33}}+(a_{11} a_{23} - a_{12} a_{13}) {(x^{33})}^2}}{a_{11}^2 ({x^{22}} {x^{33}}- {(x^{23})}^2)+(a_{11} a_{22}-a_{12}^2) {(x^{33})}^2},\\
        & l_c = -\frac{\splitfrac{a_{11}^2( {x^{13}} {x^{22}}- {x^{12}} {x^{23}})-a_{11} a_{12}( {x^{12}} {x^{33}}-{x^{13}} {x^{23}})+(a_{11} a_{13}-a_{12}^2)({x^{22}} {x^{33}}-{(x^{23})}^2)}{+(a_{13} a_{22}-a_{12} a_{23}) {(x^{33})}^2-(a_{11} a_{23}-a_{12} a_{13}) {x^{23}} {x^{33}} +(a_{11} a_{22} - a_{12}^2) ({x^{13}} {x^{33}}-{(x^{23})}^2)}}{a_{11}^2 ({x^{22}} {x^{33}}- {(x^{23})}^2)+(a_{11} a_{22}-a_{12}^2) {(x^{33})}^2}.
    \end{align*}
    We assert that for any sufficiently small $\delta>0$, there exists $\epsilon = \epsilon(X,\delta)>0$, such that $\epsilon = O(\delta^2)$ as $\delta\to 0$. Furthermore, for any $(k_0,l_0)\in \Lambda$ with $|k_0/l_0| = \delta$, there exists a positive definite matrix $A$ such that the equation \eqref{equ:ellip} holds. In other words, the following properties hold:
    \begin{itemize}
        \item A certain level curve of $s^g_{X,A}$ is an ellipse whose major axis is between $(0,0)$ and $(k_0,l_0)$.
        \item The minor axis length is $\epsilon$ times the length of the major axis.
    \end{itemize}
    By comparing the coefficients of the $k^2$, $kl$ and $l^2$ terms, we obtain:
    \[    a_{11}x^{22}+2a_{12}x^{23}+a_{22}x^{33} = \epsilon^2k_0^2+l_0^2,\ a_{11}x^{33} = \epsilon^2l_0^2+k_0^2,\ a_{11}x^{23}+a_{12}x^{33} = (\epsilon^2-1)k_0l_0.
    \]
    This equation system has a unique solution, namely
    \begin{align*}
        & a_{11} = \frac{k_0^2+\epsilon ^2 l_0^2}{{x^{33}}},a_{12} =  -\frac{(k_0^2+\epsilon ^2l_0^2) {x^{23}}+k_0 l_0(1-\epsilon ^2){x^{33}}}{{x^{33}}^2},\\
        & a_{22} = \frac{(l_0^2+\epsilon ^2k_0^2) {x^{33}}^2+2 k_0 l_0(1-\epsilon^2) {x^{23}} {x^{33}}-(k_0^2+\epsilon ^2l_0^2) ({x^{22}} {x^{33}}-2 {x^{23}}^2)}{{x^{33}}^3}.
    \end{align*}
    In order to satisfy the positive definite condition $a_{11}a_{22}>a_{12}^2$, the following inequality must hold:
    \[
    -\frac{l_0^4(x^{22}x^{33} - {x^{23}}^2)}{{x^{33}}^4}\epsilon^4 + \frac{(k_0^2+l_0^2)^2{x^{33}}^2-2k_0^2l_0^2(x^{22}x^{33} - {x^{23}}^2)}{{x^{33}}^4} \epsilon^2 - \frac{k_0^4(x^{22}x^{33} - {x^{23}}^2)}{{x^{33}}^4} > 0.
    \]
    As $k_0/l_0\to 0$, the roots of the quartic function on the left-hand side are $\epsilon = \pm \epsilon_+$ and $\epsilon = \pm\epsilon_-$, where $\epsilon_\pm = \epsilon_\pm(k_0/l_0)$ have the following series expansions:
    \[
    \epsilon_+ = \frac{x^{33}}{\sqrt{x^{22}x^{33} - {x^{23}}^2}}+O\left((k_0/l_0)^2\right),\ \epsilon_- =  \frac{\sqrt{x^{22}x^{33} - {x^{23}}^2}}{x^{33}}(k_0/l_0)^2 + O\left((k_0/l_0)^4\right),
    \]
    with all coefficients determined by $X$. We let $\delta = |k_0/l_0|$ and set $\epsilon = \epsilon(\delta)$ such that $\epsilon(\delta)>\epsilon_-(\delta)$ and $\epsilon(\delta)\sim\epsilon_-(\delta)$ as $\delta\to 0$. This choice ensures the positive definite condition $a_{11}a_{22}>a_{12}^2$.

    By comparing the coefficients of the $k^1$ and $l^1$ terms, we deduce that $k_c = k_0/2$ and $l_c = l_0/2$. Substituting the solution for $a_{11}$, $a_{12}$ and $a_{22}$ above, the denominators of the expressions for both $k_c$ and $l_c$ are
    \[
    a_{11}^2 ({x^{22}} {x^{33}}- {(x^{23})}^2)+(a_{11} a_{22}-a_{12}^2) {(x^{33})}^2 = -\epsilon^2(k_0^2+l_0^2)^2\neq 0.
    \]
    Thus, the equations $k_c = k_0/2$ and $l_c = l_0/2$ form a linear equation system with unknowns $a_{13}$ and $a_{23}$, with an invertible coefficient matrix. Consequently, the equation system has a unique solution for $a_{13}$ and $a_{23}$. Finally, we determine $a_{33}$ by setting $\det(A) = 1$.

    We continue the proof by utilizing Corollary \ref{lem:s4:11}. Our proof consists of two cases, depending on whether the entry $a$ of the generator $\gamma_2$ is rational.

    \textbf{Case (1): $a\in\mathbb{Q}$.} Assume that $a = p/q$, where $(p,q)$ are coprime. The first components of points in $\Lambda$ take values in $(1/q)\mathbb{Z}$, and we have
    \[
    \Lambda\cap \{(k,l)|k = 1/q\} = \{(1/q, l_n)|l_n = (a(a-1)-2b)qn+l_0,\ n\in\mathbb{Z}\},
    \]
    where $l_0$ is a constant depending on $a$ and $b$. Let $\delta_n = (1/q)/l_n$, then $\delta_n = O(n^{-1})$. Our construction of the matrix $A$ yields a level curve of $s^g_{X,A}$, whose major axis is the line segment between $(0,0)$ and $(1/q,l_n)$, and the minor axis length is $\epsilon_n$ times the length of the major axis. Here, $\epsilon_n = \epsilon(\delta_n) = O(n^{-2})$. Consequently, the level curve intersects with the line $\{l = 0\}$ at $k=0$ and
    \[
    k = \frac{\epsilon_n^2/q(1/q^2 + l_n^2)}{\epsilon_n^2/q^2 + l_n^2} = O(n^{-4}).
    \]
    For sufficiently large $n$, the level curve we constructed does not surround any points other than $(0,0)$ and $(1/q,l_n)$ in $\Lambda$. By Lemma \ref{lem:s4:1}, the Dirichlet-Selberg domain $DS(X,\Gamma)$ is infinitely-sided.

    \textbf{Case (2): $a\notin\mathbb{Q}$.} Choose any $(k_1,l_1)$. Our construction yields a point $A_1\in\mathcal{P}(3)$ and a level curve of $s^g_{X,A_1}$ surrounding $(0,0)$ and $(k_1,l_1)$. We choose the points $(k_i,l_i)$ inductively as follows: suppose we have chosen points $A_j\in\mathcal{P}(3)$ and $(k_j,l_j)\in\Lambda$, where $j=1,\dots, (i-1)$. Let $\Lambda_j$ be the set of points in $\Lambda$ surrounded by the level curve of $s^g_{X,A_j}$ through $(0,0)$ and $(k_j,l_j)$, then the union $\bigcup_{j=1}^{i-1}\Lambda_j$ is a finite set. Since $a\notin\mathbb{Q}$, we can choose $(k_i,l_i)\in\Lambda$ such that $k_i$ is sufficiently small and $l_i$ is sufficiently large, ensuring that the level curve of $s^g_{X,A_i}$ in our construction excludes all points in $\bigcup_{j=1}^{i-1}\Lambda_j\backslash \{(0,0)\}$. By Corollary \ref{lem:s4:11}, $DS(X,\Gamma)$ has a facet $F_{g(k_i',l_i')}$, where $(k_i',l_i')\in \Lambda_i\backslash\{(0,0)\}$ for all $i\in\mathbb{N}$. Our construction implies that these points are pairwise distinct. Thus the Dirichlet-Selberg domain $DS(X,\Gamma)$ is infinitely-sided.
\end{proof}
\begin{proof}[Proof of Theorem \ref{main:s4}, $2$-generated group of type (4)]
A $2$-generated subgroup $\Gamma$ of type (4) is interpreted as a two-parameter family
\[
g(k,l) = \begin{pmatrix}
    e^{-k} & -l e^{-k} & 0\\ 0 & e^{-k} & 0\\ 0 & 0 & e^{2k}
    \end{pmatrix},\ \forall (k,l)\in\Lambda = (t,1)\mathbb{Z}\oplus (s,a)\mathbb{Z}\subset \mathbb{R}^2,
\]
where $(s,t)\neq (0,0)$ and $a\in\mathbb{R}$. The function $s^g_{X,A}$ is expressed as:
\begin{align*}
    & s^g_{X,A}(k,l) = e^{2k}(a_{11}x^{11}+2a_{12}x^{12}+a_{22}x^{22}+2l(a_{11}x^{12}+a_{12}x^{22})+l^2a_{11}x^{22})\\
    & +2e^{-k}(a_{13}x^{13}+a_{23}x^{23}+la_{13}x^{23})+e^{-4k}a_{33}x^{33}.
\end{align*}

We assert that if $X = (x^{ij})^{-1}$ satisfies $x^{23}\neq 0$, then for any $(k_0,l_0)\in\Lambda$ where $k_0\neq 0$, there exists a point $A\in\mathcal{P}(3)$, such that:
\begin{itemize}
    \item A level curve of $s_{X,A}^g$ is connected and passes through $(0,0)$ and $(k_0,l_0)$.
    \item The level curve lies between the lines $k=0$ and $k=k_0$, and is tangent to these lines at $(0,0)$ and $(k_0,l_0)$, respectively.
\end{itemize}

Indeed, the level curve $s_{X,A}^g = c$ is the union of graphs of the following functions:
\[
l = l_\pm(e^{-k};c) = l_0(e^{-k};c) \pm \sqrt{l_1(e^{-k};c)},
\]
where
\begin{align*}
    & l_0(e^{-k};c) = -\left(\frac{x^{12}}{x^{22}}+\frac{a_{12}}{a_{11}}\right)-\frac{a_{13}x^{23}}{a_{11}x^{22}}e^{-3k},\\
    & l_1(e^{-k};c) = 2\frac{a_{11}a_{13}(x^{12}x^{23}-x^{13}x^{22})+x^{22}x^{23}(a_{12}a_{13}-a_{11}a_{23})}{a_{11}^2{x^{22}}^2}e^{-3k}\\
    & -\frac{a_{11}a_{33}x^{22}x^{33}-a_{13}^2{x^{23}}^2}{a_{11}^2{x^{22}}^2}e^{-6k}+\frac{c}{a_{11}x_{22}}e^{-2k}-\left(\frac{x^{11}x^{22}-{x^{12}}^2}{{x^{22}}^2}+\frac{a_{11}a_{22}-a_{12}^2}{a_{11}^2}\right).
\end{align*}
The function $l_1$ is a polynomial in $t = e^{-k}$ of degree $6$, with a negative leading coefficient. If $t = e^{-k} = 1$ and $t = e^{-k} = e^{-k_0}$ are the only positive zeroes of $l_1(t)$, then $l_1(e^{-k})\geq 0$ if and only if $0\leq k\leq k_0$, implying the connectedness of the level curve. Thus, it is sufficient to select numbers 
$a_{ij}$, such that the matrix $A = (a_{ij})\in\mathcal{P}(3)$, and
\begin{itemize}
    \item The values of the function $l = l_0(e^{-k};c)$ at $k=0$ and $k=k_0$ are $0$ and $l_0$, respectively.
    \item The only positive zeroes of the function $l_1(t;c)$ are $t = 1$ and $t = e^{-k_0}$. 
\end{itemize}

We set $a_{11} = 1$. The first requirement yields a linear equation system in the unknowns $a_{12}$ and $a_{13}$, with a unique set of solutions
\[
    a_{12} = -\frac{x^{12}}{x^{22}}-\frac{l_0 e^{3k_0}}{e^{3k_0}-1},\quad a_{13} = -\frac{l_0 e^{3k_0}x^{22}}{(e^{3k_0}-1)x^{23}}.
\]
With $a_{11}$, $a_{12}$ and $a_{13}$ given above, the second requirement yields a linear equation system in the unknowns $a_{23}$ and $a_{33}$, resulting in a set of solutions in terms of $k_0$, $l_0$, $X$, $c$ and $a_{22}$:
\begin{align*}
    & a_{23} = \left(-a_{22} e^{6 k_0} {x^{22}}^2 {x^{23}}+a_{22} {x^{22}}^2 {x^{23}}+c e^{4 k_0} {x^{22}} {x^{23}}-c {x^{22}} {x^{23}}+e^{6 k_0} l_0^2 {x^{22}}^2 {x^{23}}\right.\\
    & +2 e^{3 k_0} l_0 {x^{12}} {x^{22}} {x^{23}}+2 e^{6 k_0} l_0 {x^{12}} {x^{22}} {x^{23}}-2 e^{3 k_0} l_0 {x^{13}} {x^{22}}^2-e^{6 k_0} {x^{11}} {x^{22}} {x^{23}}+2 e^{6 k_0} {x^{12}}^2 {x^{23}}+\\
    & \left.{x^{11}} {x^{22}} {x^{23}} -2 {x^{12}}^2 {x^{23}}\right)/\left(2 \left(e^{3 k_0}-1\right) {x^{22}} {x^{23}}^2\right),\\
    & a_{33} = -\left(e^{3 k_0} \left(-a_{22} e^{3 k_0} {x^{22}}^2+a_{22} {x^{22}}^2+c e^{k_0} {x^{22}}-c {x^{22}}+e^{3 k_0} l_0^2 {x^{22}}^2+2 e^{3 k_0} l_0 {x^{12}} {x^{22}}\right.\right.\\
    & \left.\left.-e^{3 k_0} {x^{11}} {x^{22}}+2 e^{3 k_0} {x^{12}}^2+{x^{11}} {x^{22}}-2 {x^{12}}^2\right)\right)/\left(\left(e^{3 k_0}-1\right) {x^{22}} x^{33}\right).
\end{align*}
With $a_{23}$ and $a_{33}$ given above, the determinant $\det(a_{ij})$ forms a quadratic polynomial in $c$, with coefficients depending on $k_0,l_0,X$, and $a_{22}$. The coefficient of the $c^2$ term is
\[
-\frac{\left(1 + e^{k_0}\right)^2\left(1 + e^{2k_0}\right)^2}{4 \left(e^{k_0}+e^{2 k_0}+1\right)^2 {x^{23}}^2}<0.
\]
We select $c = c(k_0,l_0,X,a_{22})$ to be the maximum point of this quadratic function. Assuming this, $\det(a_{ij})$ becomes a quadratic polynomial in the variable $a_{22}$, whose coefficients are in terms of $k_0,l_0$, and $X$. The coefficient of the $a_{22}^2$ term is
\[
    \frac{e^{4 k_0}\left(e^{2 k_0} {x^{23}}^2 + \left(1 + e^{k_0}\right)^2\left(1 + e^{2k_0}\right){x^{22}} {x^{33}}\right)}{\left(1 + e^{k_0}\right)^2\left(1 + e^{2k_0}\right)^2{x^{33}}^2}>0.
\]
Hence, when $a_{22}$ is sufficiently large, both $\det(A)>0$ and $a_{11}a_{22} - a_{12}^2>0$ hold. This results in a positive definite matrix $A = (a_{ij})$; up to a positive scaling, it yields $\det(A) = 1$. 

To show that $t = 1$ and $t = e^{-k_0}$ are the only positive zeroes of $l_1$, we note that the derivative $l_1'(t)$ contains terms of $t^5$, $t^2$, and $t^1$ only, implying that $l_1'(t)$ has a zero $t=0$. The other zeroes of $l_1'(t)$ satisfy $3\alpha t^4 = 3\beta t + c_0$, where
\[
\alpha = \frac{a_{11}a_{33}x^{22}x^{33}-a_{13}^2{x^{23}}^2}{a_{11}^2{x^{22}}^2},\ \beta = \frac{a_{11}a_{13}(x^{12}x^{23}-x^{13}x^{22})+x^{22}x^{23}(a_{12}a_{13}-a_{11}a_{23})}{a_{11}^2{x^{22}}^2},\ c_0 =  \frac{c}{a_{11}x^{22}}.
\]
We have $\alpha>0$. Moreover, 
$(0,0)$ lies on the level curve, thus $c=\mathrm{tr}(X^{-1}\cdot A)>0$, implying that $c_0>0$. Consequently, $l_1'(t)$ has at most one positive zero, and $l_1(t)$ has at most two positive zeroes, which must be $t = 1$ and $t = e^{-k_0}$.

The remainder of the proof is divided in two cases, based on whether $a:=t/s$ is rational. If $a = p/q\in\mathbb{Q}$, then the first components of points in $\Lambda = (t,1)\mathbb{Z}\oplus (s,a)\mathbb{Z}$ take discrete values $(s/q)\mathbb{Z}$. If $a\notin \mathbb{Q}$, there exists a point $(k_n,l_n)\in\Lambda$ where $k_n$ is arbitrarily large and $l_n$ is arbitrarily close to $0$. Analogously to the previous case, the Dirichlet-Selberg domain $DS(X,\Gamma)$ is infinitely-sided whenever the center $X$ does not belong to the proper Zariski closed subset $\{X = (x^{ij})^{-1}\in\mathcal{P}(3)|x^{23} = 0\}$.
\end{proof}
\section{Schottky subgroups of \texorpdfstring{$SL(2k,\mathbb{R})$}{Lg}}\label{sec:s8}
Using Dirichlet-Selberg domains in $\mathcal{P}(n)$, we extend the notion of \textbf{Schottky groups}\cite{maskit1967characterization} to subgroups of $SL(n,\mathbb{R})$:
\begin{defn}
    A discrete subgroup $\Gamma<SL(n,\mathbb{R})$ is called a Schottky group of rank $k$ if there exists a point $X\in\mathcal{P}(n)$, such that the Dirichlet-Selberg domain $DS(X,\Gamma)$ is $2k$-sided and is ridge free. 
\end{defn}
Analogously to the case of $SO^+(n,1)$, one can prove that Schottky groups in $SL(n,\mathbb{R})$ are free subgroups of $SL(n,\mathbb{R})$. 

In this section, we show that Schottky groups in $SL(n,\mathbb{R})$ exist in the generic case when $n$ is even and only exist in a degenerated case when $n$ is odd.
\subsection{Schottky groups in \texorpdfstring{$SL(n,\mathbb{R})$}{Lg}: \texorpdfstring{$n$}{Lg} is even}
\begin{defn}
    For any $A\in SL(n,\mathbb{R})$ with only positive eigenvalues, one defines the \textbf{attracting and repulsing subspaces} of $\mathbb{R}\mathbf{P}^{n-1}$ as follows:
    \[
    C_A^+ = span_{\lambda_i>1}(\mathbf{v}_i)/\mathbb{R}^\times,\quad C_A^- = span_{0<\lambda_j<1}(\mathbf{v}_j)/\mathbb{R}^\times,
    \]
    where $\mathbf{v}_i$ denotes the eigenvector of $A^{\mathrm{T}}$ associated with the eigenvalue $\lambda_i$, $i=1,\dots,n$.
\end{defn}
\begin{thm}\label{schot}(cf. \cite{tits1972free})
    Suppose that $A_1,\dots,A_k\in SL(2n,\mathbb{R})$ are such that the attracting and repulsing spaces $C_{A_i}^\pm$, $i=1,\dots, k$, are all $(n-1)$-dimensional and pairwise disjoint. Then there exists an integer $M>0$ such that the group $\Gamma = \langle A_1^M,\dots, A_k^M\rangle$ is a Schottky group of rank $k$.
\end{thm}
\begin{proof}
    Denote the eigenvalues of $A_i$ by
    \[
    \lambda_{i,1}\geq\dots\geq \lambda_{i,n}>1>\lambda_{i,n+1}\geq\dots \geq \lambda_{i,2n}>0,\ i=1,\dots,k.
    \]
    We claim that there exists an integer $M$ satisfying the following conditions:
    \begin{itemize}
        \item For any real numbers $m_i^\pm\geq M$, $i=1,\dots, k$, the $2k$ bisectors $Bis(I, A_i^{m_i^+}.I)$, $Bis(I, A_i^{-m_i^-}.I)$ are pairwise disjoint.
        \item For each bisector $\sigma$ among the $2k$ ones, the center $I$ of the Dirichlet-Selberg domain and the other $(2k-1)$ bisectors lie in the same connected component of $\sigma^{\mathsf{c}} = \mathcal{P}(2n)\backslash \sigma$.
    \end{itemize}
    To prove our first claim, we define
    \[
    f_{m,i}^+(\mathbf{x}) = \lambda_{i,n+1}^m\left(\dfrac{||(A_i^{-m})^{\mathrm{T}}\mathbf{x}||^2}{||\mathbf{x}||^2}-1\right),\ f_{m,i}^-(\mathbf{x}) = \lambda_{i,n}^{-m}\left(\dfrac{||(A_i^{m})^{\mathrm{T}}\mathbf{x}||^2}{||\mathbf{x}||^2}-1\right),\ \forall \mathbf{x}\in\mathbb{R}\mathbf{P}^{2n-1},
    \]
    which are smooth functions on $\mathbb{R}\mathbf{P}^{2n-1}$. For any $\mathbf{x}\in C_{A_i}^\pm$, 
    \[
    \lim_{m\to \infty}f^\pm_{m,i}(\mathbf{x}) = 0,
    \]
    while for any $\mathbf{x}\notin C_{A_i}^\pm$,
    \[
    \lim_{m\to \infty}f^\pm_{m,i}(\mathbf{x}) = \infty.
    \]
    Since $\mathbb{R}\mathbf{P}^{2n-1}$ is compact and the $(n-1)$-dimensional submanifolds $C_{A_i}^\pm$ are pairwise disjoint, there exists a positive number $M$, such that for any $m_i^\pm\geq M$, the sum of any two among the $2k$ functions $f^\pm_{m_i^\pm,i}$, $i=1,\dots,k$ is positive. That is, the sum of any two among the $2k$ symmetric matrices
    \[
    \lambda_{i,n+1}^{m_i^+}((A_i^{m_i^+}.I)^{-1} - I),\ \lambda_{i,n}^{-m_i^-}((A_i^{-m_i^-}.I)^{-1} - I),\ i=1,\dots,k
    \]
    is a positive definite matrix. By Lemma \ref{lem:1}, the bisectors $Bis(A_i^{\pm m_i^\pm}.I,I)$ are pairwise disjoint for any numbers $m_i^\pm\geq M$.

    To prove our second claim, we assume the opposite: there are bisectors $\sigma_1$ and $\sigma_2$ among the $2k$ bisectors $Bis(A_i^{\pm m_i^\pm}.I,I)$, $m_i^\pm\geq M$, such that $\sigma_2$ and the center $I$ lie in different components of $\sigma_1^{\mathsf{c}}$. Without loss of generality, we suppose that $\sigma_1 = Bis(A_1^{m_1}.I,I)$ and $\sigma_2 = Bis(A_2^{m_2}.I,I)$. Let $X$ be an arbitrary point in $Bis(A_2^{m_2}.I,I)$; the assumption implies that $I$ and $X$ lie in different components of $\sigma_1^{\mathsf{c}}$. Since
    \[
    \lim_{m\to\infty}s(A_1^{m}.I,I) = \lim_{m\to\infty}s(A_1^{m}.I,X) = \infty,
    \]
    the points $I$ and $X$ lie in the same component of $Bis(A_1^{m}.I,I)^{\mathsf{c}}$ for $m$ large enough. Thus, there exists a real number $m_1'>m_1\geq M$ such that
    \[
    X\in Bis(A_1^{m_1'}.I,I).
    \]
    However, $X\in Bis(A_2^{m_2}.I,I)$, implying that $Bis(A_2^{m_2}.I,I)$ and $Bis(A_1^{m_1'}.I,I)$ intersect at $X$, which contradicts our first claim. This completes the proof of our second claim.

    In conclusion, there exists a number $M>0$ such that the Dirichlet-Selberg domain $DS(I,\Gamma)$ of $\Gamma = \langle A_1^M,\dots, A_k^M\rangle$ is bounded by the $2k$ bisectors $Bis(A_i^{\pm M}.I,I)$ and is ridge-free. Consequently, $\Gamma$ is a Schottky group of rank $k$.
\end{proof}
\subsection{Schottky groups in \texorpdfstring{$SL(n,\mathbb{R})$}{Lg}: \texorpdfstring{$n$}{Lg} is odd}
Note that for odd $n$, Schottky groups in $SL(n,\mathbb{R})$ can be obtained from these in $SL(n-1,\mathbb{R})$ via the inclusion map $SL(n-1,\mathbb{R})\hookrightarrow SL(n,\mathbb{R})$, $g\mapsto diag(g,1)$: 
\begin{exm}
    Consider the matrices
    \[
    A = diag(A_0,1) = \begin{pmatrix}
        3 & 0 & 0\\0 & 1/3 & 0\\ 0 & 0 & 1
    \end{pmatrix}, B = diag(B_0,1) =  \begin{pmatrix}
        5/3 & 4/3 & 0\\4/3 & 5/3 & 0\\ 0 & 0 & 1
    \end{pmatrix}.
    \]
    It is evident that each pair from the four matrices $((A^\pm.I)^{-1}-I) = diag((A_0^\pm.I)^{-1} - I,0)$ and $((B^\pm.I)^{-1}-I) = diag((B_0^\pm.I)^{-1} - I,0)$ has a positive semi-definite linear combination. By Lemma \ref{lem:1}, the bisectors $Bis(A^\pm.I,I)$ and $Bis(B^\pm.I,I)$ are pairwise disjoint. Moreover, for any bisector $\sigma$ among these four, the point $I$ and the other three bisectors lie in the same connected component of $\sigma^{\mathsf{c}}$. Thus, $\Gamma$ is a Schottky group. In fact, $\Gamma$ projects to the Schottky group $\langle A_0,B_0\rangle$ in $SL(2,\mathbb{R})$.
\end{exm}

In contrast, Bobb and Riestenberg proved that $SL(3,\mathbb{R})$ does not contain Schottky groups in the generic case. We show that for odd number $n$, $SL(n,\mathbb{R})$ does not contain Schottky groups under the following non-degeneracy assumption.
\begin{defn}
    We call the discrete subgroup $\Gamma<SL(n,\mathbb{R})$ a non-degenerate Schottky group of rank $k$, if there is a point $X\in\mathcal{P}(n)$ satisfying the following criteria:
    \begin{itemize}
        \item The Dirichlet-Selberg domain $DS(X,\Gamma)$ is $2k$-sided and ridge free. 
        \item For any facet $Bis(A_i.X,X)$ of $DS(X,\Gamma)$ where $A_i\in\Gamma$ and $i=1,\dots,k$, and for any eigenvalue $\lambda_{i,j}$ of $A_i$ where $j=1,\dots,n$, the absolute value $|\lambda_{i,j}|\neq 1$. 
    \end{itemize}
\end{defn}
\begin{thm}
    For any odd number $n$, there are no non-degenerate Schottky groups in $SL(n,\mathbb{R})$.
\end{thm}
\begin{proof}
    Assume to the contrary that $\Gamma = \langle A_1,\dots, A_k\rangle< SL(n,\mathbb{R})$ is a non-degenerate Schottky group of rank $k$, i.e., there exists a point $X\in\mathcal{P}(n)$ such that the Dirichlet-Selberg domain $DS(X,\Gamma)$ is ridge-free, with $2k$ facets $Bis(A_i.X,X)$ and $Bis(A_i^{-1}.X,X)$, $i=1,\dots,k$. Without loss of generality, we can assume that $X = I$ after an isometry of the Dirichlet-Selberg domain $DS(X,\Gamma)$; the Dirichlet-Selberg domain after the isometry corresponds to a subgroup conjugate to $\Gamma$.

    We extend the notions of attracting and repulsing subspaces:
    \[
    C_{A_i,\mathbb{C}}^+ = span_{\mathbb{C},|\lambda_j|>1}(\mathbf{v}_j)/\mathbb{C}^\times,\ C_{A_i,\mathbb{C}}^- = span_{\mathbb{C},|\lambda_j|<1}(\mathbf{v}_j)/\mathbb{C}^\times,
    \]
    where $\mathbf{v}_j\in \mathbb{C}^n$ is the eigenvector of $A_i^{\mathrm{T}}$ associated with the eigenvalue $\lambda_j$. Then, $C_{A_i,\mathbb{C}}^+$ and $C_{A_i,\mathbb{C}}^-$ are proper subspaces of $\mathbb{C}\mathbf{P}^{n-1}$. Since $\Gamma$ is assumed to be non-degenerate, $\dim_\mathbb{C}(C_{A_i,\mathbb{C}}^+) + \dim_\mathbb{C}(C_{A_i,\mathbb{C}}^-) = n-2$, $i=1,\dots,k$. That is, either $\dim_\mathbb{C}(C_{A_i,\mathbb{C}}^+)\geq (n-1)/2$ or $\dim_\mathbb{C}(C_{A_i,\mathbb{C}}^-)\geq (n-1)/2$. Without loss of generality, we can assume that $\dim_\mathbb{C}(C_{A_1,\mathbb{C}}^+)$, $ \dim_\mathbb{C}(C_{A_2,\mathbb{C}}^+)\geq (n-1)/2$. Consequently, the intersection
    \[
        C_{A_1,\mathbb{C}}^+\cap C_{A_2,\mathbb{C}}^+\neq \varnothing.
    \]

    On the one hand, for any $m\in\mathbb{N}$, the bisectors $Bis(A_1^m.I, I)$ and $Bis(A_2^m.I, I)$ are disjoint. For $m=1$, this follows from the definition of Schottky groups. For $m\geq 2$, we note that the bisectors $Bis(A_i^m.I,I)$, $i=1,\dots,k$ do not intersect the Dirichlet-Selberg domain $DS(I,\Gamma)$. In other words, $Bis(A_i^m.I,I)$ lies in $DS(I,\Gamma)^{\mathsf{c}}$. Therefore, the bisector $Bis(A_i^m.I,I)$ and the point $A_i^m.I$ lie in the same connected component of $DS(I,\Gamma)^{\mathsf{c}}$. Since $A_i^m.I$ lies in the component
    \[
    \mathrm{int}(H_i^+) = \{Y|s(A_i.I,Y)> s(I,Y)\},
    \]
    the bisectors $Bis(A_1^m.I, I)$ and $Bis(A_2^m.I, I)$ belong to different components of $DS(I,\Gamma)^{\mathsf{c}}$. Thus, they are disjoint.

    On the other hand, we will derive a contradiction by showing that the bisectors $Bis(A_1^m.I, I)$ and $Bis(A_2^m.I, I)$ intersect for sufficiently large $m\in \mathbb{N}$. Take vectors
    \[
    \mathbf{v}\in C_{A_1,\mathbb{C}}^+\cap C_{A_2,\mathbb{C}}^+,\ \mathbf{w}\in (C_{A_1,\mathbb{C}}^+\cup C_{A_2,\mathbb{C}}^+)^{\mathsf{c}}. 
    \]
    Similarly to the proof of Theorem \ref{schot}, we establish that
    \[
    \mathbf{w}^*((A_1^m.I)^{-1}-I)\mathbf{w}>0
    \]
    and
    \[
    \mathbf{w}^*((A_2^m.I)^{-1}-I)\mathbf{w}>0,
    \]
    for sufficiently large $m$. Furthermore, 
    \[
    \mathbf{v}^*(A_1^m.I)^{-1}\mathbf{v} = ||(A_1^{-m})^{\mathrm{T}}\mathbf{v}||^2 = ||\varphi^m(\mathbf{v})||^2\leq ||\varphi^m||^2\cdot ||\mathbf{v}||^2,
    \]
    where $\varphi$ represents the restriction of the linear transformation $(A_1^{-1})^{\mathrm{T}}$ to the $A_1^{\mathrm{T}}$-invariant subspace $\mathbb{C}\cdot C^+_{A_1,\mathbb{C}}$ of $\mathbb{C}^n$. Gelfand's theorem implies that
    \[
    \lim_{m\to\infty} ||\varphi^m||^{1/m} = \rho(\varphi) = \max_{|\lambda_j|>1}|\lambda_j|^{-1} < 1,
    \]
    where $\rho(\varphi)$ denotes the spectral radius of $\varphi$. It follows that $\lim_{m\to\infty}||\varphi^m|| = 0$, and consequently
    \[
        \lim_{m\to\infty}\mathbf{v}^*(A_1^m)^{-1}\mathbf{v} = 0.
    \]
    Similarly,
    \[
        \lim_{m\to\infty}\mathbf{v}^*(A_2^m)^{-1}\mathbf{v} = 0.
    \]
    Thus, inequalities
    \[
    \mathbf{v}^*((A_1^m.I)^{-1}-I)\mathbf{v}<0
    \]
    and
    \[
    \mathbf{v}^*((A_2^m.I)^{-1}-I)\mathbf{v}<0
    \]
    hold for sufficiently large $m$. 
    
    The inequalities above imply that the pencil
    \[
        ((A_1^m.I)^{-1}-I, (A_2^m.I)^{-1}-I)
    \]
    is indefinite for sufficiently large $m$. According to Lemma \ref{lem:1}, the bisectors $Bis(A_1^m.I, I)$ and $Bis(A_2^m.I, I)$ intersect for sufficiently large $m$, leading to a contradiction.
\end{proof}
\vspace{12pt}
\section{Future prospects}\label{sec:99}
In Section \ref{sec:s5}, we proved that an invariant angle function cannot be defined for certain pairs of hyperplanes. However, using the Riemannian angle - which depends on the base point - to formulate the ridge-cycle condition is less efficient. This prompts the question:
\begin{ques}
    When an angle function cannot be defined, how can the ridge-cycle condition in Poincar\'e's theorem be formulated without resorting to Riemannian angles\cite{kapovich2023geometric}?
\end{ques}
Even in the case of $\mathcal{P}(3)$, there exists a Dirichlet-Selberg domain $D_S(X,\Gamma)$ for which the angle function at any ridge cannot be defined:
\begin{prop}\label{prop:3:2}
Consider the group
\begin{equation*}
    \Gamma = \left\{\left.\begin{pmatrix}1 & a & b\\ 0 & 1 & 0\\ 0 & 0 & 1\end{pmatrix}\right|a,b\in\mathbb{Z}\right\},
\end{equation*}
which is a $2$-generated discrete abelian subgroup of $SL(3,\mathbb{R})$, of type (1) in Proposition \ref{2-gen}. Then for all $X\in \mathcal{P}(3)$, the angle function at any ridge of $D_S(X,\Gamma)$ is not defined.
\end{prop}

\begin{proof}
Let $X = (x_{ij})$, $X^{-1} = (x^{ij})$, and denote the two generators of $\Gamma$ 
$$h_1 = \begin{pmatrix}1 & 1 & 0 \\ 0 & 1 & 0\\ 0 & 0 & 1\end{pmatrix},\quad h_2 = \begin{pmatrix}1 & 0 & 1 \\ 0 & 1 & 0\\ 0 & 0 & 1\end{pmatrix}.$$ 
Direct computation shows that
\[
    \begin{split}
        (h_1^ah_2^b.X)^{-1} - X^{-1} = (a^2 x^{22}+2a bx^{23}+b^2x^{33}-2a x^{12} - 2b x^{13})\mathbf{e}_1\otimes \mathbf{e}_1\\
        - (ax^{22} + b x^{23})(\mathbf{e}_1\otimes \mathbf{e}_2 + \mathbf{e}_2\otimes \mathbf{e}_1) - (ax^{23} + b x^{33})(\mathbf{e}_1\otimes \mathbf{e}_3 + \mathbf{e}_3\otimes \mathbf{e}_1),
    \end{split}
\]
for any $a,b\in\mathbb{Z}$. Any ridge of $D_S(X,\Gamma)$ correspond to a pair of bisectors $Bis(h_1^ah_2^b.X, X)$ and $Bis(h_1^{a'}h_2^{b'}.X, X)$; since these bisectors intersect, the corresponding matrix pencil does not contain matrices of rank $1$ that are always semi-definite. Therefore, the pencil contain matrices of rank $2$ only, has no generalized eigenvalues, and is normalized to a singular Lancaster-Rodman block $(A^\mathrm{s},B^\mathrm{s})$. According to Theorem \ref{thm:main:2}, $Bis(h_1^ah_2^b.X, X)$ and $Bis(h_1^{a'}h_2^{b'}.X, X)$ is not in the domain of an invariant angle function. Thus, the angle function at any ridge of $D_S(X,\Gamma)$ is undefined.
\end{proof}
Regarding the generalization of step (6) of Poincar\'e's algorithm, Kapovich conjectures that finitely-sided Dirichlet-Selberg domains $DS(X,\Gamma_l)$ exhibit the completeness property, similarly to Dirichlet domains in hyperbolic spaces:
\begin{ques}[\cite{kapovich2023geometric}]
    Let $D = DS(X,\Gamma_l)$ be a finitely-sided Dirichlet-Selberg domain in $\mathcal{P}(n)$ satisfying the tiling condition. Is the quotient space $M = D/\sim$ complete?
\end{ques}
Our forthcoming research will primarily focus on this conjecture for finite-volume Dirichlet-Selberg domains.

It is imperative to strengthen the condition in Theorem \ref{thm:4:2:2} and provide a sufficient condition for ensuring that the bisectors $Bis(X,Y)$ and $Bis(Y,Z)$ are at least $1$ unit apart. This condition is linked to a quasi-geodesic property involved in the Kapovich-Leeb-Porti algorithm\cite{kapovich2014morse}.

We seek applications of Poincar\'e's algorithm to specific subgroups of $SL(n,\mathbb{R})$, including surface groups\cite{long2011zariski} and knot groups\cite{doi:10.1080/10586458.2011.565258} in $SL(3,\mathbb{Z})$.

Furthermore, we aim to generalize the algorithm and related constructions to other symmetric spaces of non-compact type, including the Siegel upper half space, the symmetric space for $SO^+(p,q)$\cite{helgason1979differential}, and Bruhat-Tits buildings\cite{garrett1997buildings}.
\vspace{12pt}
\appendix

\section{Motivation for the Angle-like Function}
This appendix describes how we discovered the equation \eqref{equ:main:2}, particularly the methodology employed to identify a suitable function:
\[
\theta:\ \Sigma_n^{(2)}\sqcup \Sigma_n^{(0)}\to [0,\pi],\ (A,B)\mapsto \theta(A^\perp,B^\perp)
\]
which fulfills the properties (1) through (4) in Definition \ref{defn:s5:1}. Here, $\Sigma_n^{(2)}$ represents the set consisting of pairs $(A,B)\in Sym_n(\mathbb{R})^2$ for which the corresponding co-oriented hyperplane pair $(A^\perp, B^\perp)$ is of type (2), and that the generalized eigenvalues of the pencil $(A,B)$ are pairwise distinct. Furthermore, $\Sigma_n^{(0)}$ represents the set consisting of pairs $(A,B)$ in $Sym_n(\mathbb{R})^2$, such that $A,B\neq O$ and $B = c\cdot A$ for some $c\in\mathbb{R}$; that is, the corresponding co-oriented hyperplanes $A^\perp$ and $B^\perp$ are either identical or oppositely co-oriented.

We begin by observing that a restriction of $\theta$ factors through $\mathbf{S}^1\times \mathbf{S}^1$. Here and after, we consistently regard $\mathbf{S}^1$ as $ \mathbb{R}/2\pi\mathbb{Z}$.
\begin{lem}
    For any pair $(A,B)\in \Sigma_n^{(2)}$, $\theta|_{(span(A,B)-\{O\})^2}$ factors through $\mathbf{S}^1\times \mathbf{S}^1$ via
    \[
    (span(A,B)-\{O\})^2\xrightarrow{\varphi\times\varphi} \mathbf{S}^1\times \mathbf{S}^1\xrightarrow{\angle} [0,\pi],
    \]
    where $\varphi$ is a map $\varphi_{(A,B)}: span(A,B)-\{O\}\to \mathbf{S}^1$, and $\angle: \mathbf{S}^1\times \mathbf{S}^1\to [0,\pi]$ is the Euclidean angle.
\end{lem}
\begin{proof}
    Consider an arbitrary pair $(A,B)\in \Sigma_n^{(2)}$. We define the map $\varphi_{(A,B)}$ as follows:
    \begin{itemize}
        \item For any $c>0$, we have $\varphi(c\cdot A) = 0$; for any $c<0$, $\varphi(c\cdot A) = \pi$.
        \item If $C$ is a positive linear combination of $B$ and $A$, or $B$ and $-A$, then $\varphi(C) = \theta(A^\perp,C^\perp)$.
        \item If $C$ is a positive linear combination of $-B$ and $A$ or $-B$ and $-A$, then $\varphi(C) = -\theta(A^\perp,C^\perp)$.
    \end{itemize}
    The definition of $\varphi$ ensures that $\angle(\varphi(C_1),\varphi(C_2)) = \theta(C_1^\perp,C_2^\perp)$ whenever $C_1 \in span(A)$ or $C_2 \in span(A)$.
    
    Next, we will show that $\angle(\varphi(C_1),\varphi(C_2)) = \theta(C_1^\perp,C_2^\perp)$ holds for any $C_1,C_2\in span(A,B) - span(A)$. After a positive scaling of $C_1$ and $C_2$, we can assume that $C_1 = \epsilon_1 B+ c_1 A$ and $C_2 = \epsilon_2 B+ c_2 A$, where $c_1,c_2\in\mathbb{R}$, and $\epsilon_1,\epsilon_2\in\{1,-1\}$. The proof is divided in three cases:

    \textbf{Case (1).} Assume that $\epsilon_1 = \epsilon_2 = 1$ and $c_1>c_2$. In this case, $C_1$ is a positive linear combination of $A$ and $C_2$. The property (4) of $\theta$ in Definition \ref{defn:s5:1} implies that
    \[
    \theta(A^\perp,C_2^\perp) - \theta(A^\perp,C_1^\perp) = \theta(C_1^\perp,C_2^\perp)>0,
    \]
    which leads to $0<\varphi(C_1)<\varphi(C_2)<\pi$. Hence, the angle is computed as
    \[
    \angle(\varphi(C_1),\varphi(C_2)) = \varphi(C_2) - \varphi(C_1) = \theta(A^\perp,C_2^\perp) - \theta(A^\perp,C_1^\perp) = \theta(C_1^\perp,C_2^\perp).
    \]
    
    \textbf{Case (2).} If $\epsilon_1 = \epsilon_2 = -1$, we can show that $\angle(\varphi(C_1),\varphi(C_2)) = \theta(C_1^\perp,C_2^\perp)$, following a reasoning analogous to the Case (1).

    \textbf{Case (3).} Assume that $\epsilon_1 = 1$ and $\epsilon_2 = -1$. When $c_1 + c_2 = 0$, it follows that $C_2 = -C_1$. The property (3) of $\theta$ in Definition \ref{defn:s5:1} implies that
    \[
        \varphi(C_2) = -\theta(A^\perp,-C_1^\perp) = -\pi + \theta(A^\perp,C_1^\perp) = \varphi(C_1) - \pi,
    \]
    leading to the conclusion that $\angle(\varphi(C_1),\varphi(C_2)) = \pi = \theta(C_1^\perp,C_2^\perp)$. 
    
    When $c_1 + c_2>0$, $A$ is a positive linear combination of $C_1$ and $C_2$. The property (4) of $\theta$ in Definition \ref{defn:s5:1} implies that
    \[
    \theta(C_1^\perp,C_2^\perp) = \theta(A^\perp,C_1^\perp) + \theta(A^\perp,C_2^\perp) = \varphi(C_1) - \varphi(C_2).
    \]
    Therefore, $\varphi(C_1)$ and $\varphi(C_2)\in \mathbf{S}^1$ satisfy that $0<\varphi(C_1) - \varphi(C_2)<\pi$, leading to the conclusion:
    \[
    \theta(C_1^\perp,C_2^\perp) = \varphi(C_1) - \varphi(C_2) = \angle (\varphi(C_1), \varphi(C_2)).
    \]

    When $c_1 + c_2<0$, we can show that $\angle(\varphi(C_1),\varphi(C_2)) = \theta(C_1^\perp,C_2^\perp)$, analogously to the preceding argument.
\end{proof}
To construct an invariant angle function $\theta$ on $\Sigma_n^{(2)}\sqcup \Sigma_n^{(0)}$, we aim at constructing a family of maps $\varphi_{(A,B)}: span(A,B)-\{O\}\to \mathbf{S}^1$ for $(A,B)\in \Sigma_n^{(2)}$ that satisfies particular properties. We note that the definition of invariant angle function requires a connection between pairs in $\Sigma_n^{(2)}$. Specifically, the property (2) in Definition \ref{defn:s5:1} suggests a connection between the pairs $(A,B)$ and $(g.A,g.B)$ for any $g\in SL(n,\mathbb{R})$. Moreover, the pairs $(A,B)$ and $(A',B')$ are inherently related when $span(A',B') = span(A,B)$. This observation motivates us to introduce an equivalence relation on $\Sigma_n^{(2)}$, formalized as follows:
\begin{defn}
    Two pairs $(A,B)$, $(A',B')\in \Sigma_n^{(2)}$ are equivalent if
    \[
    (A',B') = (g.(p A+ q B),g.(r A+s B)),
    \]
    for an element $g\in SL(n,\mathbb{R})$ and numbers $p,q,r,s\in\mathbb{R}$ with $p s-q r\neq 0$. We denote this equivalence relation by $(A,B)\sim (A',B')$.
\end{defn}

The concept of the cross-ratio plays an important role in characterizing the equivalence classes on $\Sigma_n^{(2)}$, as will be introduced in Subsection \ref{Sec:A:1}.
\subsection{Further insights to the family of functions \texorpdfstring{$\varphi_{(A,B)}$}{Lg}}\label{Sec:A:1}

We refer to \cite{hartshorne1967foundations} for the following concepts and propositions related to the cross-ratio:
\begin{defn}
    The \textbf{cross-ratio} of four distinct points $p_i = [x_i:y_i]\in \mathbb{R}\mathbf{P}^1$, $i=1,2,3,4$, is defined as
    \[
    R_\times (p_1,p_2;p_3,p_4) = \frac{(x_3y_1-y_3x_1)(x_4y_2-y_4x_2)}{(x_3y_2-y_3x_2)(x_4y_1-y_4x_1)}.
    \]

    For points $\mathbf{x}_i\in \mathbb{R}^2-\{(0,0)\}$, the cross ratio $R_\times(\mathbf{x}_1,\mathbf{x}_2;\mathbf{x}_3,\mathbf{x}_4)$ denotes $R_\times([\mathbf{x}_1],[\mathbf{x}_2];[\mathbf{x}_3],[\mathbf{x}_4])$, where $[\mathbf{x}_i]\in\mathbb{R}\mathbf{P}^1$, for $i=1,2,3,4$.
\end{defn}
\begin{prop}\label{prop:a:1:1}
    Consider points $p_i\in \mathbb{R}\mathbf{P}^1$ for $i=1,2,3,4$. The cross-ratio $R_\times (p_1,p_2;p_3,p_4)$ is greater than $1$ if and only if the points $p_1$, $p_2$, $p_3$, and $p_4$ are arranged on $\mathbb{R}\mathbf{P}^1$ according to their cyclic order.
\end{prop}
\begin{prop}\label{prop:A:1:2}
    Identify points $p_i\in \mathbb{R}\mathbf{P}^1$ as one-dimensional linear subspaces of $\mathbb{R}^2$, where $i=1,2,3,4$. Denote by $\theta_i$ the angle between $p_i$ and $p_{i+1}$ for $i=1,2,3$. The cross-ratio $R_\times (p_1,p_2;p_3,p_4)$ is expressed as a function of $\theta_i$:
    \[
    R_\times (p_1,p_2;p_3,p_4) = R_\times(\theta_1,\theta_2,\theta_3): = \frac{\sin(\theta_1+\theta_2)\sin(\theta_2+\theta_3)}{\sin\theta_2\sin(\theta_1+\theta_2+\theta_3)}.
    \]
\end{prop}
\begin{prop}\label{prop:A:1:3}
    Consider two-dimensional vector spaces $V$ and $W$ and an invertible linear map $\varphi: V\to W$. Let $p_i\in V$ and $q_i\in W$ represent distinct points for $i=1,2,3,4$. Assuming that $q_i = \varphi(p_i)$ for $i=1,2,3$, then the equality
    \[
    R_\times (p_1,p_2;p_3,p_4) = R_\times (q_1,q_2;q_3,q_4)
    \]
    holds if and only if $q_4$ is a non-zero multiple of $\varphi(p_4)$.
\end{prop}
\begin{rmk}\label{rmk:A:1:1}
    Invertible linear maps on $\mathbb{R}^2$ correspond to M\"obius transformations on $\mathbb{R}\mathbf{P}^1$. Thus, Proposition \ref{prop:A:1:3} remains valid when the parts ``two-dimensional spaces'' and ``invertible linear map'' are replaced by ``projective lines'' and ``M\"obius transformation'' respectively.
\end{rmk}
The cross-ratios of consecutive points among $n\geq 4$ points in $\mathbb{R}\mathbf{P}^1$ depend on each other:
\begin{lem}\label{lem:A:1:1:1}
    Consider $n\geq 4$ distinct points $p_1,\dots, p_n\in \mathbb{R}\mathbf{P}^1$, denote $R_i = R_\times (p_i,p_{i+1};p_{i+2},p_{i+3})$ for $i=1,\dots,n$, with indices being taken modulo $n$. It follows that $R_{n-2}$, $R_{n-1}$, and $R_n$ are determined by $R_1$ through $R_{n-3}$.
\end{lem}
\begin{proof}
    A unique M\"obius transformation exists that takes $p_1$, $p_2$ and $p_3$ to $0$, $1$ and $\infty$, respectively. As noted in Remark \ref{rmk:A:1:1}, M\"obius transformations preserve the cross-ratios among points in $\mathbb{R}\mathbf{P}^1$. Consequently, we may assume that $(p_1,p_2,p_3) = (0,1,\infty)$ without loss of generality. Under this assumption, equations $R_i = R_\times (p_i,p_{i+1};p_{i+2},p_{i+3})$, $i=1,\dots,n-3$ uniquely determine the points $p_4$ through $p_n$. Hence, the cross-ratios $R_{n-2}$, $R_{n-1}$ and $R_n$ are also uniquely determined.
\end{proof}
We return to the equivalence classes in $\Sigma_n^{(2)}$.
\begin{defn}
    Define $M_n'$ as the $(n-3)$-dimensional space given by
    \[
    M_n' = \{(R_1,\dots,R_n)\in \mathbb{R}_{>1}^n| R_i = R_i(R_1,\dots,R_{n-3}),i=n-2,n-1,n\},
    \]
    where $R_i(R_1,\dots,R_{n-3})$, for $i=n-2,n-1,n$, are as described in Lemma \ref{lem:A:1:1:1}. Introduce the following equivalence relations on $M_n'$:
    \[
    (R_1,\dots, R_{n})\sim (R_{n},\dots,R_1),\ (R_1,\dots,R_{n})\sim (R_{i+1},\dots, R_{i+n}), \ i=1,\dots,n-1,
    \]
    with indices taken modulo $n$. Define $M_n = M_n'/\sim$ as the quotient space by these equivalence relations.

    For a point in $M_n$ represented by $(R_1,\dots, R_{n})\in M_n'$, denote by $\Sigma_n^{R_1,\dots,R_{n}}$ the set consisting of pairs $(A,B)\in \Sigma_n^{(2)}$, such that the generalized eigenvalues of $(A,B)$ are ordered as $\lambda_n>\dots>\lambda_1$, and their cross-ratios
    \[
    R_i': = R_\times (\lambda_i,\lambda_{i+1};\lambda_{i+2},\lambda_{i+3}),\ i=1,\dots,n,
    \]
    satisfy that $(R_1,\dots,R_n)\sim(R_1',\dots,R_n')$. For conciseness, we omit the last three components from the superscript and denote the set above by $\Sigma_n^{R_1,\dots,R_{n-3}}$.
\end{defn}
\begin{lem}
    Each equivalence class in $\Sigma_n^{(2)}$ is contained in $\Sigma_n^{R_1,\dots,R_{n-3}}$ for a specific element $(R_1,\dots,R_{n})\in M_n$.
\end{lem}
\begin{proof}
    For each $(A,B)\in \Sigma_n^{(2)}$ with eigenvalues $\lambda_n>\dots>\lambda_1$, Proposition \ref{prop:a:1:1} implies that the cross-ratios $R_i = R_\times (\lambda_i,\lambda_{i+1};\lambda_{i+2},\lambda_{i+3})>1$ for $i=1,\dots, n$.

    Consider a pair $(A,B)\in \Sigma_n^{R_1,\dots,R_{n-3}}$, where $(R_1,\dots,R_{n})\in M_n$. According to Proposition \ref{prop:2}, the generalized eigenvalues of $(g.A,g.B)$ are equal to those of $(A,B)$ for any $g\in SL(n,\mathbb{R})$. Thus, $(g.A,g.B)\in \Sigma_n^{R_1,\dots,R_{n-3}}$. Moreover, for any $p,q,r,s\in\mathbb{R}$ with $ps-qr\neq 0$, Lemma \ref{lem:mob} demonstrates that the generalized eigenvalues of $(pA+qB,rA+sB)$ are equal to $(p \lambda_i+q)/(r \lambda_i+s)$, where $i=1,\dots, n$, which are the images of $\lambda_1$, $\dots$, $\lambda_n$ under a M\"obius transformation. According to Remark \ref{rmk:A:1:1}, M\"obius transformations preserve the cross-ratios. Additionally, such transformations alter the order of $\lambda_i$, $i=1,\dots,n$ by a cyclic permutation, possibly followed by a reversal. Thus, they alter the order of $R_\times(\lambda_i,\lambda_{i+1};\lambda_{i+2},\lambda_{i+3})$, $i=1,\dots,n$ in the same manner. Such permutations preserve the equivalence classes in $M_n'$. Hence, $(pA+qB,rA+sB)\in \Sigma_n^{R_1,\dots,R_{n-3}}$ as well.
\end{proof}
\begin{rmk}
    For $(A,B)\in \Sigma_n^{R_1,\dots,R_{n-3}}$, the set of singular matrices in $span(A,B)$ consists of $n$ lines, specifically $span(A-\lambda_i B)$, where $\lambda_i$, $i=1,\dots, n$ are the generalized eigenvalues of $(A,B)$. Additionally, $R_i = R_\times(A-\lambda_i B,A-\lambda_{i+1} B;A-\lambda_{i+2} B,A-\lambda_{i+3} B)$ for $i=1,\dots,n$.
\end{rmk}
We make a further assumption on the family of maps $\varphi_{(A,B)}$. For any $(R_1,\dots,R_n)\in M_n$, we assume that the angles $\theta_i = \angle(\varphi_{(A,B)}(C_i),\varphi_{(A,B)}(C_{i+1}))$ remain constant despite different selections of $(A,B)\in \Sigma_n^{R_1,\dots,R_{n-3}}$. Here, $C_1,\dots, C_n$ are the generators of the $n$ lines of singular matrices in $span(A, B)$, with $R_i = R_\times(C_i, C_{i+1}; C_{i+2}, C_{i+3})$, $i=1,\ldots,n$. According to Proposition \ref{prop:A:1:2}, the angles $\theta_i$, $i=1,\dots,n$, satisfy the following condition:
\begin{equation}\label{equ:a1}
    \begin{split}
        & \sum_{i=1}^n\theta_i = \pi,\\
        & R_\times(\theta_i,\theta_{i+1},\theta_{i+2}) = R_i,\ \forall i=1,\dots, n.
    \end{split}
\end{equation}
Denote by $\Theta$ a proposed function that takes $(R_1,\dots,R_n)\in M_n'$ to $(\theta_1,\dots,\theta_n)\in (\mathbf{S}^1)^n$ satisfying \eqref{equ:a1}. Note that a cyclic permutation of the elements $C_1,\dots, C_n$ results in identical permutations of both $(R_1,\dots, R_{n})$ and $(\theta_1,\dots,\theta_n)$. Consequently, we establish another condition:
\begin{equation}\label{equ:a2}
    \begin{split}
        & \Theta(R_{1+j},\dots,R_{n+j}) = (\theta_{1+j},\dots,\theta_{n+j}),\ \forall j=0,\dots, n-1,\\
        & \Theta(R_n,\dots,R_1) = (\theta_2,\theta_1,\theta_n,\dots,\theta_3).
    \end{split}
\end{equation}
We describe the requirements for the family of maps $\varphi_{(A,B)}$ as below.    
\begin{prop}\label{prop:A:1:4}
    Suppose that there exists a map $\Theta: M_n'\to (\mathbf{S}^1)^n$ satisfying conditions \eqref{equ:a1} and \eqref{equ:a2}. Moreover, assume that for each tuple $(R_1,\dots,R_{n})\in M_n'$ and and the corresponding output $(\theta_1,\dots,\theta_n) = \Theta(R_1,\dots, R_n)$, there is a family of maps $\varphi_{(A,B)}:span(A,B)-\{O\}\to\mathbf{S}^1$ for $(A,B)\in\Sigma_n^{R_1,\dots,R_{n-3}}$ satisfying the criteria:
    \begin{enumerate}
        \item The map $\varphi_{(A,B)}$ is the composition of a linear map $\psi_{(A,B)}: span(A,B)\to \mathbb{R}^2$ and the canonical quotient map $\mathbb{R}^2-\{O\}\to (\mathbb{R}^2-\{O\})/\mathbb{R}^+ = \mathbf{S}^1$.
        \item Suppose that $C_1,\dots,C_n$ is any set of pairwise linearly independent singular elements in $span(A,B)$ satisfying that
        \[
        R_\times(C_i,C_{i+1};C_{i+2},C_{i+3}) = R_i,
        \] 
        for $i=1,\dots,n$, and $C_2,\dots,C_{n-1}$ represent positive linear combinations of $C_1$ and $C_{n}$, where the indices are taken modulo $n$. Then the relationship
        \begin{equation}\label{equ:A:2}
        \varphi_{(A,B)}(C_i) = \eta+ \sum_{j=1}^{i-1}\theta_j,\ i=1,\dots, n,
        \end{equation}
        holds for a specific $\eta = \eta(C_1,\dots, C_n) \in \mathbf{S}^1$.
    \end{enumerate}
Under these criteria, the function
\[
\theta(A^\perp,B^\perp):=\angle(\varphi_{(A,B)}(A),\varphi_{(A,B)}(B))
\]
is well-defined and serves as an invariant angle function.
\end{prop}
\begin{proof}
    We begin by proving the uniqueness, up to an additive constant, of the map $\varphi_{(A,B)}$. Assuming that $\varphi_{(A,B)}(C_1) = 0$, the criteria (1) and (2) imply the expressions
    \[
    \psi_{(A,B)}(C_1) = k_1(1,0),\ \psi_{(A,B)}(C_2) = k_2(\cos\theta_1,\sin\theta_1),\ \psi_{(A,B)}(C_n) = k_n(-\cos\theta_n,\sin\theta_n),
    \]
    where $k_1$, $k_2$, and $k_n$ are positive scalars. This determines a unique linear map $\psi_{(A,B)}$ up to a positive multiple, implying the uniqueness of $\varphi_{(A,B)}$ under the assumption $\varphi_{(A,B)}(C_1) = 0$. Consequently, $\varphi_{(A,B)}$ is unique up to a constant without this assumption. Thus, the function $\theta$ is uniquely determined and is independent of $\eta$.

    Properties (1), (3) and (4) in Definition \ref{defn:s5:1} are inherently fulfilled by the angular nature of $\theta$. Regarding the property (2), suppose that a choice of singular elements $C_1,\dots, C_n$ in $span(A,B)$ satisfies the criteria in the proposition. Thus, the elements $g.C_1,\dots, g.C_n$ satisfy these criteria for $span(g.A,g.B)$. The uniqueness of $\varphi_{(g.A,g.B)}$ implies that
    \[
    \varphi_{(g.A,g.B)}(g.C) = \varphi_{(A,B)}(C),
    \]
    for any $C\in span(A,B)-\{O\}$. Hence,
    \[
    \theta((g.A)^\perp,(g.B)^\perp) = \angle(\varphi_{(g.A,g.B)}(g.A),\varphi_{(g.A,g.B)}(g.B)) =\angle(\varphi_{(A,B)}(A),\varphi_{(A,B)}(B)) = \theta(A^\perp,B^\perp).
    \]
\end{proof}
\begin{rmk}
    The proof of the existence of the map $\varphi_{(A,B)}$, guaranteed by the equation \eqref{equ:a1}, is omitted here for conciseness.
\end{rmk}

The objective now shifts to identifying an appropriate map $\Theta: M_n'\to (\mathbf{S}^1)^n$. We will deal with this in the next section, focusing on cases with small values of $n$.

\subsection{Examples: cases \texorpdfstring{$n=3$}{Lg}, \texorpdfstring{$4$}{Lg}, and \texorpdfstring{$5$}{Lg}}\label{Sec:A:2}
In this section, we demonstrate the derivation of the formula \eqref{equ:main:2} when $n$ takes the values $3$, $4$, and $5$.
\begin{exm}\label{ex:A:1}
    Suppose that $n=3$. In this case, $M_3$ is reduced to a point. Consider any pair $(A,B)\in \Sigma_3^{(2)}$ whose generalized eigenvalues are arranged as $\lambda_1<\lambda_2<\lambda_3$. Both choices of singular elements $C_i$, $i=1,2,3$, satisfy the criterion (2) in Proposition \ref{prop:A:1:4}:
    \[
    C_1 = A-\lambda_1 B,\ C_2 = A-\lambda_2 B,\ C_3 = A-\lambda_3 B,
    \]
    or
    \[
    C_1 = A-\lambda_2 B,\ C_2 = A-\lambda_3 B,\ C_3 = -(A-\lambda_1 B).
    \]
    Thus the function $\varphi_{(A,B)}$ satisfies \eqref{equ:A:2} for either choice of $C_i$, yielding that
    \[
    (\eta,\eta+\theta_1,\eta+\theta_1+\theta_2) = (\eta'+\theta_1,\eta'+\theta_1+\theta_2,\eta'+\pi),
    \]
    where $\eta,\eta'$ are specific elements in $\mathbf{S}^1$. This condition implies that $\theta_1 = \theta_2 = \theta_3 = \pi/3$, a choice that also satisfies the equation \ref{equ:a2}. Consequently, the criterion (1) in Proposition \ref{prop:A:1:4} implies that 
    \[
    \psi_{(A,B)}(A - \lambda_i B) = k_i\left(\cos\left(\frac{(i-1)\pi}{3}\right),\sin\left(\frac{(i-1)\pi}{3}\right)\right),\quad i=1,2,3,
    \]
    where $k_i$, $i=1,2,3$ are positive numbers. To determine these numbers, we note that the matrices $C_i = A - \lambda_i B$, $i=1,2,3$, are linearly dependent, expressed as:
    \[
        \frac{1}{\lambda_3-\lambda_1}(C_3-C_1) = \frac{1}{\lambda_2-\lambda_1}(C_2-C_1).
    \]
    Therefore,
    \[
        \frac{1}{\lambda_3-\lambda_1}(\psi_{(A,B)}(C_3)-\psi_{(A,B)}(C_1)) = \frac{1}{\lambda_2-\lambda_1}(\psi_{(A,B)}(C_2)-\psi_{(A,B)}(C_1)).
    \]
    Combining these equations, we derive a solution, which is unique up to a positive multiple:
    \[
        k_1 = \frac{1}{\lambda_3-\lambda_2},k_2 = \frac{1}{\lambda_3-\lambda_1},k_3 = \frac{1}{\lambda_2-\lambda_1}.
    \]
    Therefore, we determine $\psi_{(A,B)}(A)$ and $\psi_{(A,B)}(B)$ up to a positive multiple, particularly determining the angle between them. Straightforward computations show that
    \begin{equation}\label{equ:A:a1}
    \begin{split}
        &\cos\theta (A^\perp,B^\perp) = \cos\angle(\psi_{(A,B)}(A),\psi_{(A,B)}(B))\\
        &= \frac{\lambda_1^2\lambda_2+\lambda_2^2\lambda_3+\lambda_3^2\lambda_1 + \lambda_1\lambda_2^2+\lambda_2\lambda_3^2+\lambda_3\lambda_1^2-6\lambda_1\lambda_2\lambda_3}{2\sqrt{\splitfrac{(\lambda_1^2+\lambda_2^2+\lambda_3^2-\lambda_1\lambda_2-\lambda_2\lambda_3-\lambda_3\lambda_1)}{(\lambda_1^2\lambda_2^2+\lambda_2^2\lambda_3^2+\lambda_3^2\lambda_1^2-\lambda_1^2\lambda_2\lambda_3-\lambda_2^2\lambda_3\lambda_1-\lambda_3^2\lambda_1\lambda_2)}}}.
    \end{split}
    \end{equation}
\end{exm}
\begin{exm}\label{ex:A:2}
    Consider the case where $n=4$. Let $p_1$, $p_2$, $p_3$, and $p_4\in\mathbb{R}\mathbf{P}^1$ and denote their cross-ratio by $R = R_\times (p_1,p_2;p_3,p_4)$. The cross-ratios of these points in different orders are as follows:
    \[
    R_\times (p_2,p_3;p_4,p_1) = \dfrac{R}{R-1},\ R_\times (p_3,p_4;p_1,p_2) = R,\ R_\times (p_4,p_1;p_2,p_3) = \dfrac{R}{R-1}.
    \]
    Therefore, $M_4 = \mathbb{R}_{>1}/\{R\sim \frac{R}{R-1}\}$. For any $R\in M_4$ and $(A,B)\in \Sigma_4^R$ with generalized eigenvalues $\lambda_1<\lambda_2<\lambda_3<\lambda_4$, both choices of the singular elements $C_i$, $i=1,2,3,4$, satisfy the criterion (2) in Proposition \ref{prop:A:1:4} with $R_\times(C_1,C_2;C_3,C_4) = R$:
    \[
    C_1 = A-\lambda_1 B,\ C_2 = A-\lambda_2 B,\ C_3 = A-\lambda_3 B,\ C_4 = A-\lambda_4 B,
    \]
    or
    \[
    C_1 = A-\lambda_3 B,\ C_2 = A-\lambda_4 B,\ C_3 = -(A-\lambda_1 B),\ C_4 = -(A-\lambda_2 B).
    \]
    Thus, the function $\varphi_{(A,B)}$ satisfies \eqref{equ:A:2} for both choices of $C_i$. Similarly to Example \ref{ex:A:1}, this condition implies that
    \[
    \theta_1 = \theta_3,\ \theta_2 = \theta_4.
    \]
    This condition also satisfies equation \eqref{equ:a2}. Moreover, Proposition \ref{prop:A:1:2} implies that
    \[
    R = R_\times(\theta_1,\theta_2,\theta_3) = \frac{\sin\pi/2\cdot \sin\pi/2}{\sin(\pi/2-\theta_1)\sin(\pi/2+\theta_1)} = \frac{1}{\cos^2\theta_1},
    \]
    hence we have:
    \[
    \theta_1 = \theta_3 = \arccos\sqrt{1/R},\ \theta_2 = \theta_4 = \arcsin\sqrt{1/R}.
    \]
    Consequently, the criterion (1) in Proposition \ref{prop:A:1:4} implies that $\psi_{(A,B)}: span(A,B)\to \mathbb{R}^2$ is characterized by:
    \[
    \psi_{(A,B)}(A-\lambda_1 B) = k_1(1,0),\ \psi_{(A,B)}(A-\lambda_2 B) = k_2(1,\sqrt{R-1}),
    \]
    \[
    \psi_{(A,B)}(A-\lambda_3 B) = k_3(0,1),\ \psi_{(A,B)}(A-\lambda_4 B) = k_4(-\sqrt{R-1},1),
    \]
    for some positive numbers $k_i$, $i=1,2,3,4$. By a straightforward computation similar to Example \ref{ex:A:1}, we determine these numbers and derive the following formula:
    \begin{equation}\label{equ:A:a2}
    \begin{split}
    &\cos\theta(A^\perp,B^\perp) = \cos\angle(\psi_{(A,B)}(A),\psi_{(A,B)}(B))\\
    & = \frac{\lambda_4\lambda_2-\lambda_3\lambda_1}{\sqrt{{(\lambda_4-\lambda_3+\lambda_2-\lambda_1)}{(\lambda_4\lambda_3\lambda_2-\lambda_4\lambda_3\lambda_1+\lambda_4\lambda_2\lambda_1-\lambda_3\lambda_2\lambda_1)}}}.
\end{split}
    \end{equation}
\end{exm}
\begin{exm}\label{ex:A:3}
    We consider the case where $n=5$. For points $p_i\in\mathbb{R}\mathbf{P}^1$, $i=1,\dots,5$, we define the cross-ratios as usual:
    \[
    R_i = R_\times (p_i,p_{i+1};p_{i+2},p_{i+3}).
    \]
    As asserted by Lemma \ref{lem:A:1:1:1}, cross-ratios $R_3$, $R_4$ and $R_5$ are dependent on $R_1$ and $R_2$. This is clear from the following relationships:
    \[
    R_3 = \frac{R_2}{R_1(R_2-1)},\ R_4 = \frac{1}{R_1+R_2 - R_1R_2},\ R_5 = \frac{R_1}{R_2(R_1-1)}.
    \]
    For any $(A,B)\in \Sigma_5^{R_1,R_2}$ with generalized eigenvalues $\lambda_1<\dots<\lambda_5$, the following choice of singular elements $C_i$, $i=1,\dots, 5$, satisfies the assumption in criterion (2) with cross-ratios $(R_{1+j},\dots,R_{5+j})\in M_5'$:
    \[
    C_1 = A-\lambda_{j+1}B,\ \dots,\ C_{5-j} = A-\lambda_5 B,\ C_{6-j} = -(A-\lambda_1 B),\ \dots,\ C_5 = -(A-\lambda_j B),
    \]
    where $j=0,\dots, 4$, and the indices are taken modulo $5$. 
    
    To find a function $\Theta$ satisfying equations \eqref{equ:a1} and \eqref{equ:a2}, we note that $\theta_3$, $\theta_4$ and $\theta_5$ are dependent on $\theta_1$ and $\theta_2$ according to Proposition \ref{prop:A:1:2}. Specifically, if we set $c_i = \cot \theta_i$, $i=1,\dots, 5$, then the second equation in \eqref{equ:a1} implies that
    \[
        R_i = \frac{(c_{i+1}+c_i)(c_{i+1}+c_{i+2})}{c_ic_{i+1}+c_ic_{i+2}+c_{i+1}c_{i+2}-1},\quad i=1,\dots,5.
    \]
    Consequently, we regard $c_3$, $c_4$ and $c_5$ as rational functions of $c_1$ and $c_2$, with parameters $R_1$ and $R_2$. After straightforward computation, the condition $\sum\theta_i = \pi$ appears superfluous. We define
    \[
        F(c_1,c_2;R_1,R_2) = c_1+c_2+\sum_{j=3}^5c_j(c_1,c_2;R_1,R_2),
    \]
    which is a rational function of $(c_1,c_2)\in\mathbb{R}^2$ with parameters $R_1$ and $R_2$. Further computation reveals that $F(c_1,c_2)$ has a unique minimum point at:
    \begin{equation}\label{equ:A:3}
        \begin{split}
            c_1 =  \frac{(R_1R_2-R_1-R_2)(R_1R_2-R_1+R_2+1)}{\sqrt{\splitfrac{(R_1-1)(R_2-1)(R_1R_2-R_1-R_2)}{\cdot(3R_1^2R_2^2-6R_1R_2^2-6R_1^2R_2+3R_1^2+3R_2^2+5R_1R_2-3R_1-3R_2)}}},\\
            c_2  = \frac{(R_1-1)(R_2^2+2R_1 R_2 + R_2-R_1R_2^2-R_1)}{\sqrt{\splitfrac{(R_1-1)(R_2-1)(R_1R_2-R_1-R_2)}{\cdot(3R_1^2R_2^2-6R_1R_2^2-6R_1^2R_2+3R_1^2+3R_2^2+5R_1R_2-3R_1-3R_2)}}}.
        \end{split}
    \end{equation}
    \begin{prop}
        Let $\Theta: M_5'\to (\mathbf{S}^1)^5$, $\Theta(R_1,\dots,R_5) = (\theta_1,\dots,\theta_5):=(\mathrm{arccot}c_1,\dots,\mathrm{arccot}c_5)$, where $c_1 = c_1(R_1,R_2)$ and $c_2 = c_2(R_1,R_2)$ are given by \eqref{equ:A:3}, and $c_j = c_j(c_1,c_2;R_1,R_2)$ for $j=3,4,5$ are described above. Then, $\Theta$ satisfies the conditions \eqref{equ:a1} and \eqref{equ:a2}.
    \end{prop}
    \begin{proof}
        It suffices to show that $\Theta(R_{1+j},R_{2+j}) = (\theta_{1+j},\dots,\theta_{5+j})$ for $j=1,\dots, 4$, where the indices are taken modulo $5$. Suppose that $\Theta(R_{1+j},R_{2+j}) = (\theta_{1+j}',\dots,\theta_{5+j}')$. By definition, both $(\theta_1,\dots,\theta_5)$ and $(\theta_{1}',\dots,\theta_{5}')$ are points on the variety
        \[
        \{(\theta_1,\dots,\theta_5)|R_\times(\theta_i,\theta_{i+1},\theta_{i+2}) = R_i,i=1,\dots, 5\},
        \]
        that minimize $\sum_{i=1}^5\cot \theta_i$. We have verified that such a minimum point is unique. Therefore, $\theta_i' = \theta_i$ for $i=1,\dots, 5$.
    \end{proof}
    From the function $\Theta: M_5\to (\mathbf{S}^1)^5$ constructed above, we derive the following formula by a straightforward computation, analogously to Examples \ref{ex:A:1} and \ref{ex:A:2}:
    \begin{equation}\label{equ:A:a3}
        \begin{split}
            & \cos \theta(A^\perp,B^\perp) = \cos \angle(\psi_{(A,B)}(A),\psi_{(A,B)}(B)) \\
            & = \frac{\sum(3\lambda_1^2\lambda_2\lambda_3\lambda_4+3\lambda_1\lambda_2\lambda_3\lambda_4^2-\lambda_1^2\lambda_3^2\lambda_4 - \lambda_1^2\lambda_3\lambda_4^2-\lambda_1\lambda_2^2\lambda_3\lambda_4-\lambda_1\lambda_2\lambda_3^2\lambda_4)-10\lambda_1\lambda_2\lambda_3\lambda_4\lambda_5}{2\sqrt{\splitfrac{(\sum(\lambda_1^2\lambda_2\lambda_3+\lambda_1\lambda_2\lambda_3^2+\lambda_1^2\lambda_3\lambda_4-\lambda_1^2\lambda_3^2-\lambda_1\lambda_2\lambda_3\lambda_4-\lambda_1\lambda_2^2\lambda_3))}{(\sum(\lambda_1^2\lambda_2^2\lambda_3\lambda_4+\lambda_1^2\lambda_2\lambda_3\lambda_4^2+\lambda_1\lambda_2\lambda_3^2\lambda_4^2-\lambda_1\lambda_2^2\lambda_3^2\lambda_4-\lambda_1^2\lambda_2^2\lambda_4^2-\lambda_1^2\lambda_2\lambda_3\lambda_4\lambda_5))}}}.
        \end{split}
    \end{equation}
\end{exm}
We can further simplify Equation \eqref{equ:A:a3}. Noticeably, the cyclic sums appearing in equation \eqref{equ:A:a3} represent cyclic sums of products of linear terms in $\lambda_1,\dots, \lambda_5$, namely,
\[
    \begin{split}
        & \sum_{cyc}(\lambda_1^2\lambda_2\lambda_3+\lambda_1\lambda_2\lambda_3^2+\lambda_1^2\lambda_3\lambda_4-\lambda_1^2\lambda_3^2-\lambda_1\lambda_2\lambda_3\lambda_4-\lambda_1\lambda_2^2\lambda_3)\\
        = & -\sum_{cyc} (\lambda_1-\lambda_2)(\lambda_2-\lambda_3)(\lambda_3-\lambda_4)(\lambda_4-\lambda_5),
    \end{split}
\]
\[
    \begin{split}
        & \sum_{cyc}(3\lambda_1^2\lambda_2\lambda_3\lambda_4+3\lambda_1\lambda_2\lambda_3\lambda_4^2-\lambda_1^2\lambda_3^2\lambda_4 - \lambda_1^2\lambda_3\lambda_4^2-\lambda_1\lambda_2^2\lambda_3\lambda_4-\lambda_1\lambda_2\lambda_3^2\lambda_4)-10\lambda_1\lambda_2\lambda_3\lambda_4\lambda_5\\
        = & -\sum_{cyc} (\lambda_1+\lambda_5)(\lambda_1-\lambda_2)(\lambda_2-\lambda_3)(\lambda_3-\lambda_4)(\lambda_4-\lambda_5),
    \end{split}
\]
and
\[
    \begin{split}
        & 4 \sum_{cyc}(\lambda_1^2\lambda_2^2\lambda_3\lambda_4+\lambda_1^2\lambda_2\lambda_3\lambda_4^2+\lambda_1\lambda_2\lambda_3^2\lambda_4^2-\lambda_1\lambda_2^2\lambda_3^2\lambda_4-\lambda_1^2\lambda_2^2\lambda_4^2-\lambda_1^2\lambda_2\lambda_3\lambda_4\lambda_5)\\
        = & -\sum_{cyc} (\lambda_1+\lambda_5)^2(\lambda_1-\lambda_2)(\lambda_2-\lambda_3)(\lambda_3-\lambda_4)(\lambda_4-\lambda_5),
    \end{split}
\]
implying that
\[
    \cos\theta(A,B) = \frac{\sum_{i=1}^5\frac{\lambda_{i+1}+\lambda_i}{\lambda_{i+1}-\lambda_i}}{\sqrt{\left(\sum_{i=1}^5\frac{1}{\lambda_{i+1}-\lambda_i}\right)\left(\sum_{i=1}^5\frac{(\lambda_{i+1}+\lambda_i)^2}{\lambda_{i+1}-\lambda_i}\right)}},
\]
which is a special case of equation \eqref{equ:main:2} when $k=5$. After obtaining such an equation for general $k$, it is notable that equations \eqref{equ:A:a1} and \eqref{equ:A:a2} are special cases of the equation \eqref{equ:main:2} when $k=3$ and $k=4$, respectively. This observation motivates us to prove that the function given by \eqref{equ:main:2} is an invariant angle function for every $k\geq 3$.
\section{Infinite-sidedness for the integer Heisenberg group}
We consider the Heisenberg group:
\begin{defn}
    The \textbf{Heisenberg group} over a ring $R$ is defined as
    \[
    H_3(R) = \left\{\left.\begin{pmatrix}
        1 & a & b\\ 0 & 1 & c\\ 0 & 0 & 1
    \end{pmatrix}\right|a,b,c\in R\right\}.
    \]
\end{defn}
The Heisenberg group over $\mathbb{R}$ is a $3$-dimensional non-Abelian nilpotent Lie group. The geometric structure of the Heisenberg group, known as the \textbf{Nil geometry}, is one of Thurston's eight model geometries\cite{scott1983geometries}.

It is worth exploring whether the Dirichlet-Selberg domain for the integer Heisenberg group $H_3(\mathbb{Z})$ - as a discrete subgroup of $SL(3,\mathbb{R})$ - is infinitely-sided when centered at a given point $X\in \mathcal{P}(3)$. At least for certain choices of $X$, the answer is affirmative:
\begin{prop}\label{heis}
Suppose that $X\in\mathcal{P}(3)$, with $X^{-1} = (x^{ij})_{i,j=1}^3$, and $x^{23}= 0$. Then the Dirichlet-Selberg domain $DS(X,H_3(\mathbb{Z}))$ is infinitely-sided.
\end{prop}

\begin{proof}
We view $H_3(\mathbb{Z})$ as a three-parameter family:
    \[
    g(k,l,m) = \begin{pmatrix}1 & m & k m+l\\ 0 & 1 & k\\ 0 & 0 & 1\end{pmatrix} = \begin{pmatrix}1 & -m & -l\\ 0 & 1 & -k\\ 0 & 0 & 1\end{pmatrix}^{-1},\ \forall (k,l,m)\in\mathbb{Z}^3,
    \]
    and express the function $s_{X,A}^g$ as:
    \[
    s_{X,A}^g(k,l,m) = x^{33}(a_{22}(k-k_c)^2+2a_{12}(k-k_c)(l-l_c)+a_{11}(l-l_c)^2) + x^{22}a_{11}(m-m_c)^2+const,
    \]
    where $A = (a_{ij})$ as usual, and
    \[
    k_c = \frac{a_{11}a_{23} - a_{12}a_{13}}{a_{11}a_{22}-a_{12}^2},\ l_c = \frac{x^{13}}{x^{33}}+\frac{a_{22}a_{13} - a_{12}a_{23}}{a_{11}a_{22}-a_{12}^2},\ m_c = \frac{a_{12}x^{22} + a_{11}x^{12}}{a_{11}x^{22}}.
    \]
    We separate the variables as $s_{X,A}^g(k,l,m) = s_1(k,l)+s_2(m)+const$, where
    \[
    s_1(k,l) = x^{33}(a_{22}(k-k_c)^2+2a_{12}(k-k_c)(l-l_c)+a_{11}(l-l_c)^2),\ s_2(m) = x^{22}a_{11}(m-m_c)^2.
    \]
    The function $s_1$ coincides with $s^g_{X,A}$ for two generated groups of type (1') that appeared in Section \ref{sec:s4}. As we have proven, for any coprime pair $(k_0,l_0)\in\mathbb{Z}^2$, there is a family of matrices $A = A(\epsilon) \in\mathcal{P}(3)$, $0<\epsilon<\epsilon_0$, such that $(0,0)$ and $(k_0,l_0)$ are the only minimal points of $s_1|_{\mathbb{Z}^2}$. Here, $\epsilon_0$ is a positive number dependent on $k_0$ and $l_0$, and the matrix $A(\epsilon)$ is dependent on $k_0$, $l_0$ as well as $x^{ij}$.

    Among these coprime pairs, infinitely many pairs $(k_0,l_0)\in \mathbb{Z}^2$ exist for which the matrix $A(\epsilon)$ constructed above ensures that $s_2|_\mathbb{Z}$ has a unique minimum point at $m = 0$ when $\epsilon$ approaches $0$. Indeed, $s_2(m)$ is a quadratic polynomial with a positive leading coefficient. Consequently, $m=0$ is the sole minimum point of $s_2|_\mathbb{Z}$ if and only if $|m_c|<1/2$; the latter condition yields that
    \[
    \frac{-x^{22} - 2x^{12}}{2 x^{22}}(\epsilon^2 l_0^2/k_0+k_0)<-(1-\epsilon^2) l_0 <\frac{x^{22} - 2x^{12}}{2 x^{22}}(\epsilon^2 l_0^2/k_0+k_0).
    \]
    As $\epsilon$ tends to $0$, the above inequality simplifies to:
    \[
    \left(\frac{x^{12}}{x^{22}}-\frac{1}{2}\right)k_0< l_0 <\left(\frac{x^{12}}{ x^{22}}+\frac{1}{2}\right)k_0,
    \]
    which holds for infinitely many coprime pairs of integers $(k_0,l_0)$. For each of these pairs, there exists a sufficiently small $\epsilon$, dependent on $k_0$ and $l_0$, such that a level curve of $s^g_{X,A(\epsilon)}$ solely encloses integer points $(0,0)$ and $(k_0,l_0)$. Thus, these are the only minimum points of $s_1|_{\mathbb{Z}^2}$, while $m=0$ serves as the only minimum point of $s_2|_\mathbb{Z}$. Consequently, $(k,l,m) = (0,0,0)$ and $(k,l,m) = (k_0,l_0,0)$ are the only minimum points of $s^g_{X,A}(k,l,m)$. Lemma \ref{lem:s4:1} with $\Lambda = \mathbb{Z}^3$ implies that the Dirichlet-Selberg domain $DS(X,H_3(\mathbb{Z}))$ is infinitely sided for any $X\in \{X = (x^{ij})^{-1}\in\mathcal{P}(3)|x^{23} = 0\}$.
\end{proof}
\section{Isometry classes of \texorpdfstring{$2$}{2}-planes in \texorpdfstring{$\mathcal{P}(3)$}{P(3)}}
Sections \ref{sec:s5} and \ref{sec:s3} discussed the nature of intersections of Selberg bisector pairs, which, when nonempty and transverse, are planes of codimension $2$ in $\mathcal{P}(n)$. On the other hand, intersections of more distinct hyperplanes are more complicated, making a direct treatment analogous to the real hyperbolic case for the algorithm in Section \ref{sec:s7} infeasible. Nevertheless, the symmetric space $\mathcal{P}(3)$ is exceptional: the $\operatorname{Isom}_0$-classes of $2$-planes in $\mathcal{P}(3)$ admit a classification involving one continuous parameter along with finitely many sporadic classes.

Extending $\mathcal{P}(3)$ to the entire $\mathbf{P}(\operatorname{Sym}_3(\mathbb{R}))$, a $2$-plane is expressed as the intersection of three hyperplanes $(P,Q,R)^\perp = P^\perp\cap Q^\perp\cap R^\perp$, where $P,Q,R\in\operatorname{Sym}_3(\mathbb{R})$ are linearly independent; a change of basis of $P,Q,R$ in $\operatorname{GL}(3,\mathbb{R})$ does not change this intersection $2$-plane, and an $\operatorname{SL}(3,\mathbb{R})$ congruence preserves the isometry class of its restriction to $\mathcal{P}(3)$. Since scalar matrices act trivially on projective space, $\operatorname{GL}(3,\mathbb{R})$ gives the same projective congruence orbits. Therefore, the moduli space of $2$-planes in $\mathbf{P}(\operatorname{Sym}_3(\mathbb{R}))$ under congruence is
\[
\operatorname{GL}(3,\mathbb{R})\backslash
\{(P,Q,R)\in\operatorname{Sym}_3(\mathbb{R})^3\mid
\dim\operatorname{span}(P,Q,R)=3\}
/\operatorname{GL}(3,\mathbb{R}).
\]
The spans of these triples, identified with spaces of quadratic forms, are known as \emph{nets of quadrics}. The $\operatorname{GL}(3,\mathbb{R})\times\operatorname{GL}(3,\mathbb{R})$ orbits of these triples, classified in \cite{wall1977nets}, correspond to the $\operatorname{SL}(3,\mathbb{R})$-orbits of $2$-planes in $\mathbf{P}(\operatorname{Sym}_3(\mathbb{R}))$. Certain orbits form a continuous family, with representatives
\[
P=\begin{pmatrix}
0&0&1\\
0&1&0\\
1&0&0
\end{pmatrix},\quad
Q=\begin{pmatrix}
0&0&0\\
0&0&1\\
0&1&0
\end{pmatrix},\quad
R_{g,c}=\begin{pmatrix}
1&0&-\frac{g}{3}\\
0&\frac{2g}{3}&0\\
-\frac{g}{3}&0&\frac{g^2}{2}-c
\end{pmatrix}.
\]
Here our $(g,c)$ correspond to $(3g,\frac{9g^2}{2}+c)$ in \cite{wall1977nets}. The triples satisfy $(P,Q,R_{g,c})\sim(P,Q,R_{sg,s^2c})$ for $s>0$ by the $\operatorname{diag}(s^{-1/2},1,s^{1/2})$-action.

The inner product on $\operatorname{Sym}_3(\mathbb{R})$ gives an orthogonality relation:
\[
(P,Q,R)\perp(P',Q',R')
\iff
\exists h\in\operatorname{SL}(3,\mathbb{R}),\ \operatorname{span}(P,Q,R)
=
h.\operatorname{span}(P',Q',R')^\perp.
\]
This duality is an involution on the moduli space of projective $2$-planes in $\mathbf{P}(\operatorname{Sym}_3(\mathbb{R}))$. In the continuous family, $(P,Q,R_{g,c})\perp(P,Q,R_{-g,-c})$
\cite[Fig.~1]{wall1977nets}. In Wall's classification, mutually dual types are denoted, for example, by $D_a$ and $D_a^*$.

The orbit represented by $(P,Q,R_{g,c})$ belongs to type $A$, $B$, or $B^*$ for $(g,c)\neq(0,0)$, or type $C$ for $(g,c)=(0,0)$. The correspondence is
\begin{center}
\begin{tabular}{c|c|c|c|c|c}
\hline
& $c>g^2/2$ & $c=g^2/2$ & $\lvert c\rvert<g^2/2$ & $c=-g^2/2$ & $c<-g^2/2$\\
\hline
$g>0$ & \multirow{3}{*}{$A_d$} & $B_a^*$ & $A_c$ & $B_c$ & \multirow{3}{*}{$A_b$}\\
\cline{1-1}\cline{3-5}
$g=0$ & & \multicolumn{3}{c|}{$C$} &\\
\cline{1-1}\cline{3-5}
$g<0$ & & $B_c^*$ & $A_a$ & $B_a$ &\\
\hline
\end{tabular}
\end{center}
A strictly definite triple $(P,Q,R)$ corresponds to the $2$-plane $\mathbf{P}(\operatorname{span}(P,Q,R))\cap\mathcal{P}(3)$. On the other hand, Lemma \ref{lem:1} implies that an indefinite triple $(P,Q,R)$ corresponds to the $2$-plane $(P,Q,R)^\perp$
in $\mathcal{P}(3)$. These correspondences associate the $\operatorname{Isom}_0$-orbits of $2$-planes in $\mathcal{P}(3)$ with the strictly definite orbits and indefinite orbits respectively. The definiteness of the orbits is summarized in \cite[Figs.~2--3]{wall1977nets}:
\begin{itemize}
\item Orbits of types $A_b$, $A_c$, $B_c$, $D_a$, $D_b$, $D_c^*$, $E_a$, and $F_b^*$ are strictly definite.
\item Orbits of types $A_a$, $A_d$, $B_c^*$, $D_a^*$, $D_b^*$, $D_c$, $E_a^*$, and $F_b$ are indefinite.
\item Orbits of types $B_a$, $B_a^*$, $C$, $E_b$, $E_b^*$, $F_a$, $F_a^*$, $G$, $G^*$, $H$, $I$, and $I^*$ are semidefinite but not strictly definite.
\end{itemize}
In particular, the strictly definite types are the duals of the indefinite types. Since the strictly definite types $A_b$, $A_c$, and $B_c$ form a one-parameter family, we have the following corollary of Wall's classification.
\begin{cor}
The $\operatorname{Isom}_0$-orbits of $2$-planes in $\mathcal{P}(3)$ form a one-parameter family, along with five sporadic types. In particular, the orbit of maximal flats in $\mathcal{P}(3)$ corresponds to the sporadic type $E_a$.
\end{cor}

\bibliography{y-reference}{}
\bibliographystyle{alpha}
\end{document}